\documentclass[a4paper,11pt,oneside]{article}
\usepackage{graphicx} 
\usepackage[a4paper]{geometry}
\usepackage[T1]{fontenc}
\usepackage[utf8]{inputenc}
\usepackage[english]{babel}
\usepackage{microtype}
\usepackage{quoting}
\usepackage{amsthm, amsmath, amssymb, amscd, mathrsfs, mathtools} 
\usepackage{dsfont}
\usepackage{booktabs, caption, graphicx} 
\usepackage{enumitem}
\usepackage{esint} 
\usepackage{framed}
\usepackage{braket}
\usepackage{tikz-cd}
\usepackage{centernot}
\usepackage{fancyhdr}
\usepackage{emptypage}
\usepackage{csquotes}
\usepackage{leftidx}
\usepackage{tensor}
\usepackage{stmaryrd}
\usepackage{setspace}
\usepackage{nicematrix}
\usepackage[normalem]{ulem} 

\usepackage[dvipsnames]{xcolor}

\usepackage[colorlinks,linkcolor=blue, urlcolor = blue]{hyperref}
\hypersetup{
citecolor=Green
,filecolor=Mulberry
,urlcolor=NavyBlue
,menucolor=BrickRed
,runcolor=Mulberry
,linkbordercolor=BrickRed
,citebordercolor=Green
,filebordercolor=Mulberry
,urlbordercolor=NavyBlue
,menubordercolor=BrickRed
,runbordercolor=Mulberry
}

\usepackage[backend=biber,style=alphabetic, sorting=nyt]{biblatex}
\usepackage{settings} 

\mathtoolsset{showonlyrefs}

\newcommand{\Addresses}{{
  \bigskip
  \footnotesize

  \textsc{Andrea Malchiodi}\par\nopagebreak
  Scuola Normale Superiore, Piazza dei Cavalieri 7, 56126 Pisa, Italy \par\nopagebreak
  \textit{E-mail address}: \texttt{andrea.malchiodi@sns.it}

  \medskip

  \textsc{Francesco Malizia}\par\nopagebreak
  Scuola Normale Superiore, Piazza dei Cavalieri 7, 56126 Pisa, Italy \par\nopagebreak
  \textit{E-mail address}: \texttt{francesco.malizia@sns.it}

}}

\author{Andrea Malchiodi and Francesco Malizia\thanks{
The authors are grateful to Matt Gursky for helpful discussions and suggestions. They are supported by the PRIN Project 2022AKNSE4 {\em Variational and Analytical aspects of Geometric PDE} and are members of \emph{Gruppo Nazionale per l'Analisi Matematica, la Probabilità e le loro Applicazioni} (GNAMPA), as part of the Istituto Nazionale di Alta Matematica. Andrea Malchiodi is also supported by the project {\em Geometric problems with loss of compactness} from Scuola Normale Superiore. 
}}

\title{Decreasing Weyl's energy by connected sums \\ with locally conformally flat manifolds}

\date{}
\begin{document}

\maketitle

{\footnotesize
		\begin{abstract}	
			\noindent We study the Weyl functional on connected sums of 
            two four-dimensional manifolds $(M,g_M)$ and $(Z,g_Z)$, 
            assuming $g_M$ is Bach-flat and $g_Z$ locally conformally flat. 
            We show that if $g_M$ is neither self-dual nor anti self-dual 
            and if $g_Z$ is of positive Yamabe class, there exists a metric $g_Y$ 
            on $Y := M \# Z$ with Weyl energy lower than that of $g_M$ (with the 
            trivial exception of $(Z,g_Z) = (\Sp^4, g_{\Sp^4})$). This result has a 
            relation to a conjecture by Singer and has a perspective application 
            to the minimization of  Weyl energy. 
            The proof relies on a simultaneous interplay of $W_M^+, W_M^-$
            and the topology of $Z$, and 
            also covers some orbifold cases.

			\vspace{3ex}
			
			\noindent{\it Key Words:} Weyl Functional, Self-Duality, Locally Conformally Flat Manifolds.

			\noindent{{\bf MSC 2020: }53C21, 53C25, 58E11, 49J99. } 
			
	\end{abstract}}

\begin{spacing}{0.75}
    {\tableofcontents}
\end{spacing}

\section{Introduction}\label{sec:intro}

In Riemannian geometry it is often  desirable to  look for {\em canonical metrics}:  
in four dimension one natural condition for a metric $g$ is to be Einstein, namely when the 
Ricci tensor satisfies 
\begin{equation*}
	\mathrm{Ric}(g)=\lambda g. 
\end{equation*}
Analytically, the Einstein condition (up to a suitable gauge transformation such as a local choice 
of  harmonic coordinates) gives rise to an elliptic system in the metric (see e.g. \cite[Chapter 5]{besse-book-2008}), 
allowing to obtain compactness properties up to a finite number of points (\cite{anderson-1989-JAMS-einsteincompactness}, \cite{bando-kasue-nakajima-1989-Inventiones}). This property is peculiar of the four-dimensional case, 
since in dimension $2$ and $3$ the Einstein condition is rigid and yielding constant sectional curvature, 
while in dimension greater or equal to five there might be an abundance of such metrics, see \cite{gromov-1999-GAFA-spaces-and-questions}.

However, there are known obstructions to the existence of Einstein metrics: for example, some topological 
ones are given by the  \emph{Hitchin-Thorpe inequality} (\cite{thorpe-1969-JMathMech-HitchinThorpe, hitchin-1974-JDG-HitchinThorpe}), 
but some of them might depend  as well on other properties, such as the differentiable structure, see  e.g. \cite{lebrun-1996-mathResLett-4manif-without-Einstein}, \cite{kotschick-1998-G&T-einstein-smooth-structure}. 
We refer the reader to the surveys \cite{anderson-2010-einstein-survey}, \cite{lebrun-2003-einstein-survey} for more detailed comments.

\

Other natural metrics to consider might arise  from extremizing 
curvature functionals (see \cite[Chapter 4]{besse-book-2008}), 
such as the integrals of the squares of the scalar curvature 
or of the norm of the Riemann tensor. An example 
without known obstructions, up to now, is \textit{Weyl's energy}.   In four dimensions, 
the curvature tensor admits the following $L^2$-orthogonal decomposition 
\begin{equation}\label{eq:curv-decomp-4}
	\mathrm{Riem}=W \, \oplus \, \frac{1}{2}E\owedge g \,\oplus\,\frac{R_g}{24}g\owedge g, 
\end{equation}
where $\owedge$ stands for the Kulkarni-Nomizu product (see e.g. \cite{viaclovsky-notes-ParkCity}), $E:=\mathrm{Ric}-\frac{R_g}{4}g$ for the traceless Ricci tensor and $W$ denotes the Weyl tensor, the \emph{conformally covariant} component of the curvature. 

The \emph{Weyl functional} is defined in this way: 
\begin{equation}\label{eq:weyl-funct}
	\w^M(g):=\int_M \abs{W^g}^2\,dV_g.
\end{equation}
The covariance of $W_g$ makes $\w^M(\cdot)$ conformally invariant, and therefore
it only depends on the conformal class $[g]$ of $g$. Critical points of \eqref{eq:weyl-funct} are  \emph{Bach-flat metrics}, characterized by 
the equation 
\begin{equation}\label{eq:bach-flat-equation}
	B_{ij}:=-4\big(\nabla^{k}\nabla^{l} W_{k i j l}+\frac{1}{2}R^{kl} W_{k i j l}\big)=0, 
\end{equation}
where $B=B_{ij}$ is the \emph{Bach tensor}, see \cite{bach-1921-mathz}, \cite{kobayashi-1985-J-Math-Soc-second-weyl-variation}.

Bach-flat metrics include conformally Einstein metrics, self-dual or anti self-dual metrics but not only, even though only 
few examples different from these ones are known, see  \cite{abbena-garbiero-salamon-2013-CRAS-bachflat-lie-groups}. 
To our knowledge, there are no obstructions to the existence of Bach-flat metrics on closed $4$-manifolds, so these might be a 
natural class of critical metrics to look for.

Still, the system \eqref{eq:bach-flat-equation} is not simple to treat for different reasons. First, as for 
the Einstein equation, there is a degeneracy due to the action of diffeomorphisms, and for the 
present situation also due to the conformal covariance of the problem. For this reason, 
the parabolic version of \eqref{eq:bach-flat-equation} has been treated in \cite{bahuaud-helliwell-2011-CPDE-fourth-order-flow-existence}, \cite{bahuaud-helliwell-2015-BullLMS-uniqueness-fourth-order} and \cite{chen-lu-qing-2023-AGAG-conf-Bach-flow} 
by applying suitable gauge transformations in the spirit of DeTurck's trick for the Ricci flow. 
As for Einstein metrics, it is also expected that sequences of Bach-flat metrics  (or approximate 
Bach-flat) might develop \emph{collapsing} (see  \cite{cheeger-gromov-1986-JDG-collapsing1,cheeger-gromov-1990-JGD-collapsing2}) 
or \emph{orbifold singularities}, as explained in \cite{anderson-2006-AsianJ-canonical-metrics}. 
Indeed, under suitable curvature and volume bounds one has orbifold precompactness (see \cite{anderson-2005-Math.Annalen-orbifoldcompactness}, \cite{tian-viaclovsky-2005-Inventiones}, \cite{tian-viaclovsky-2005-advances}, \cite{tian-viaclovsky-2008-CMH}, \cite{streets-2010-transactions-removal-singul-Bachflat}).

\bigskip

Our goal in this paper is to exclude some type of topological degenerations for Bach-flat metrics 
via energy comparison, showing that on suitable connected sums of Bach-flat 
manifolds one can find a metric with lower energy than the sum of the two initial ones. 
To put our results into perspective, we begin by recalling the following theorem.

\begin{maintheorem}\label{thm:main} (\cite{malchiodi-malizia-2025-pre-weyl}) 
	Let $(M,g_M)$, $(Z,g_Z)$ be two $4$-dimensional closed, connected, oriented manifolds, and let $Y:=Z\#M$ be their connected sum. Assume that $g_M, g_Z$ are not locally conformally flat and that they are not both self-dual or both anti-self-dual. Then there exists a metric $g_Y$ on $Y$ such that
	\begin{equation}\label{eq:en-inequality}
		\w^Y(g_Y)<\w^M(g_M)+\w^Z(g_Z).
	\end{equation}
\end{maintheorem}

\bigskip

The above result has a counterpart  for the \emph{Willmore functional} defined on immersions $f:\Sigma^\rho\to\R^n$  from 
surfaces of genus $\rho$, and given by 
\begin{equation*}
	\mathbb{W}(f):=\frac{1}{4}\int_{\Sigma^\rho}\abs{H}^2\,d\mu,
\end{equation*}
where $H$ is the mean curvature of the immersion. Images of immersions that are stationary for $\mathbb{W}$ are known as {\em Willmore surfaces}. 
In \cite{simon-1986-proceed-existence-willmore-min-1,simon-1993-CAG-existence-Willmore-min-2} 
direct methods were used to study minimizing sequences, but one could not exclude a-priori 
convergence in the Hausdorff distance to surfaces of \emph{lower genus} $\rho_0<\rho$. 
This possibility was ruled out  in \cite{bauer-kuwert-2003-IMRN} by showing that if $\rho_i \geq 1$, $i = 1, 2$ and 
if $f_i:\Sigma^{\rho_i}\to\R^n$ are Willmore immersions, then there exists  an immersed surface $f:\Sigma^\rho\to\R^n$ with the topology of the connected sum (i.e. $\Sigma^\rho\cong\Sigma^{\rho_1}\#\Sigma^{\rho_2}$) such that 
\begin{equation}\label{eq:willmore-gluing-ineq}
	\mathbb{W}(f)<\mathbb{W}(f_1)+\mathbb{W}(f_2)-4\pi.
\end{equation}
As a consequence, for a minimizing sequence of immersions in genus $\rho$ it is not 
convenient to {\em split} into immersions of lower genus.  
The above inequality was proved by joining the two surfaces $f_i(\Sigma^{\rho_i})$ at two non-umbilic 
points, properly oriented according to their principal directions,   Kelvin-inverting $\Sigma^{\rho_1}$, 
and biharmonically interpolating the two local graphs.

\bigskip

The strategy for proving Theorem \ref{thm:main} is in some way similar in spirit. 
In a first instance, given $p \in M$, one 
performs a \emph{conformal blow-up} of $g_M$ 
on $M \setminus \{p\}$, a replacement for Kelvin's inversion on immersed surfaces. More precisely, one  
sets $(N,g_N) = (M \setminus \{p\}, f^2 g_M)$, 
with $f(\cdot) \simeq d_g(\cdot,p)^{-2}$. If $x$ 
is a system of geodesic coordinates centered at $p$, 
in {\em inverted coordinates} $y = {x}/{|x|^2}$ 
one has $g_N \simeq dy^2$ for $|y|$ large, and therefore $g_N$ is {\em asymptotically flat}. Notice that this process involves a change 
of orientation, which is later accounted for by \uline{gluing  a 
reversed copy of $M$}.

For $a > 0$ small, one then {\em shrinks} 
$(N,g_N)$ into a manifold $(N,g_a)$ by setting  $g_a:=a^2g_N$, and at the same time 
{\em enlarges} $(Z,g_Z)$ into a manifold $(Z,g_{b})$ via the scaling $g_b=b^{-2}g_Z:=\lambda^{-2}a^{-2}g_Z$, for some $\lambda>0$.
Using then a nearly biharmonic interpolation of $g_a$ 
and $g_b$ in a proper annulus centered at $q \in Z$ and of relative width $\gamma\ll1$, one can define a metric $g_Y$ on the connected sum $Y = M \# Z$ such that 
\begin{equation}\label{eq:en-bal-naive}
	\w^Y(g_Y)-\w^M(g_M)-\w^Z(g_Z)=a^4\Big(C-\frac{4}{9}\pi^2 \lambda^2 W^{\overbar{M}}(p)\stell W^Z(q) +O(\lambda^2\gamma^2) \Big) + \mathrm{h.o.t.}. 
\end{equation}
Here $\overbar{M}$ stands for $M$ with the opposite orientation, and $W^{\overbar{M}}(p)\stell W^Z(q)$ 
denotes the quantity 
\begin{equation}\label{eq:weyl-star-weyl}
	W^{\overbar{M}}(p)\stell W^Z(q):=\sum_{i,j,k,l}W^Z_{kijl}(q)\big(W^{\overbar{M}}_{kijl}(p)+W^{\overbar{M}}_{lijk}(p)\big). 
\end{equation}
In this  procedure, we are identifying $T_p \overbar{M}$ with $T_q Z$, and we are expressing tensor components with respect to a 
common orthonormal basis. Notice that the right-hand side in \eqref{eq:weyl-star-weyl} \uline{depends upon the identification of the two tangent spaces}. 

It turns out that, under the assumptions of Theorem \ref{thm:main}, it is 
always possible to identify $T_p \overbar{M}$ and $T_q Z$ so that $W^{\overbar{M}}(p)\stell W^Z(q) > 0$, which by \eqref{eq:en-bal-naive} gives the desired result. In fact, if $W_\pm$ denote the self-dual and anti self-dual components of the Weyl tensors, one sees in \eqref{eq:weyl-star-weyl} an interaction between  $W_\pm^M(p)$ and $W_\mp^Z(q)$.  
The term $W^{\overbar{M}} \stell W^Z$ appeared in 
\cite{gursky-viaclovsky-2016-Advances}, where the authors constructed metrics on connected sums that are critical for proper linear combinations of $\mathcal{W}$ and the integral of the 
scalar curvature squared. In  \cite{taubes-1992-JDG-selfdual} 
a related interaction was used to construct anti self-dual metrics 
on  $M\#m\overbar{\mathbb{C}\mathbb{P}}^2$, where $M$ is 
arbitrary and $m$ is taken sufficiently large.

\

Theorem \ref{thm:main} does not apply when $g_M$ and $g_Z$ are both  SD
or ASD. Indeed, for such cases there is a conjecture by I.Singer, whose validity  implies that if the two manifolds are also of positive Yamabe class,  it 
should be possible to find a family of self-dual (anti self-dual) metrics on $M \# Z$ with  {\em arbitrarily small necks}. As such metrics would be all Bach-flat (indeed, global minima of $\mathcal{W}$ by the signature formula), 
their Weyl energies would all coincide along deformations, and thus one would have equality 
of both sides in \eqref{eq:en-inequality} through such a family. 
We refer to \cite{gover-gursky-2024-CRELLE} for recent progress about Singer's 
conjecture, to \cite{floer-1991-JDG-CPsum}  for a 
construction on connected sums of $\mathbb{C}\mathbb{P}^2$'s, and 
to \cite{donaldson-friedman-1989-nonlin-SDgluing} for the more general case in which the second cohomology groups associated to the SD deformation complexes vanish.

\bigskip

The purpose of the present paper is to understand the remaining case in 
which the assumptions of Theorem \ref{thm:main} are not satisfied, namely 
when one manifold, say $(M,g_M)$, is neither self-dual nor anti self-dual 
and the other manifold is locally conformally flat. Our main result is 
the following theorem.

\begin{theorem}\label{thm:main2} 
	Let $(M,g_M)$, $(Z,g_Z)$ be two four-dimensional closed, connected, oriented manifolds, and let $Y:=Z\#M$ be their connected sum. Assume that $g_M$ is Bach-flat and  
	neither self-dual nor anti self-dual and that $g_Z$ is locally conformally 
	flat and of positive Yamabe class. Then, in case $(Z,g_Z)$ is {
		\em not} globally conformal to $(\Sp^4,g_{\Sp^4})$, there exists a metric $g_Y$ on $Y$ such that
	\begin{equation}\label{eq:en-inequality-2}
		\w^Y(g_Y)<\w^M(g_M). 
	\end{equation}
\end{theorem}

Before discussing the main ideas of the proof, we comment on the assumptions of the theorem.

\begin{remark}\label{rem:main-thm-remark}
	\begin{itemize}
		\item[(a)]  If $(Z,g_Z) = (\Sp^4,g_{\Sp^4})$, the connected sum $Y$ would have trivially the topology of $M$, and an optimal choice of $g_Y$ would just be $g_M$. We recall that, by a result of Kuiper \cite{kuiper-1949-Annals-LCFsimplyconnected}, the sphere is the only locally conformally flat manifold that is simply connected, up to conformal diffeomorphisms. In our proof, 
		as discussed below, we exploit the {\em developing map} 
        in order to  distinguish the round sphere from any other locally conformally flat manifold of positive Yamabe class, just like the 
         {\em ADM mass} does. Also, as noted before, if Singer's conjecture is 
		true and $M$ is of positive Yamabe class, for \eqref{eq:en-inequality-2} to hold one needs \underline{both} $W^M_+$ and $W^M_-$ 
		to be non zero  somewhere. Therefore, we need to use crucially \underline{an interplay 
        among $W_M^+, W_M^-$ and the topology of $Z$ at the same time}  to {\em diffuse} 
        efficiently Weyl energy from $M$ into $Z$.
		\item[(b)] We are assuming both $M$ and $Z$ to be Bach-flat. If one of them is not, it is very easy to construct a metric on $Y$ as in \eqref{eq:en-inequality-2}, modifying first either $g_M$ or $g_Z$. 
		\item[(c)] For  Willmore surfaces in $\R^n$  the counterpart of local conformal flatness would be umbilicity, but in this case  one has rigidity and the 
		corresponding surface would be a round sphere.  
	\end{itemize}
\end{remark}

\bigskip

In view of the applications to the minimization of Weyl's energy, discussed at the end of this introduction, we expect Theorem \ref{thm:main2} to be most relevant in positive Yamabe class, where one controls the Sobolev inequality, in turn useful to avoid collapsing. It would be interesting to investigate the cases in which $Z$ has zero or negative scalar curvature.

\medskip

As we remarked before, under the assumptions of Theorem \ref{thm:main2} the main interaction term in \eqref{eq:en-bal-naive} vanishes identically, and therefore we need a more precise choice of the metric and a finer expansion of Weyl's energy. 
To give a more precise idea of our procedure, we focus first on the 
locally conformally flat manifold $Z$, where near $q$ we can assume that the metric is flat, by the conformal invariance of the problem. Identifying $T_p \overbar{M}$ 
and $T_q Z$ and denoting by $x$ normal coordinates  centered at $q$, 
we can also use $x$ as {\em inverted normal coordinates} on $T_p M$ (after a proper dilation). 

Denoting by $\tilde{g}_M$ the conformal blow-up  of $g_M$, 
we have the expansion 
\[
\tilde{g}_M \simeq \Big( \delta_{ij} + \tau \, {W}^{\overbar{M}}_{k i j l} \frac{x^k x^l}{|x|^4} \Big) dx^i dx^j  \qquad \hbox{ for } \; x \simeq 0, \; x \neq 0, 
\]
where $\tau > 0$ is a small coefficient.

At this point, to match the two metrics across the gluing region, one needs to find a correction $\tau F$ to $g_Z$ with  $F \simeq {W}^{\overbar{M}}_{k i j l} \frac{x^k x^l}{|x|^4} dx^i dx^j$ near $q$. Moreover, it is clearly convenient to keep the metric on $Z$ \emph{as Bach-flat as possible}, and therefore it is also natural to look for a correction that solves
\begin{equation}\label{eq:corr-MZ}
	 \dot{B}_{g_Z} (F) = 0 \quad \hbox{ on } Z \setminus \{q\}.    
\end{equation}
This is reminiscent of other gluing constructions, and in particular to the one in \cite{schoen-1984-JDG-Yamabesol}, 
where for the (scalar) Yamabe equation a highly concentrated {\em bubble profile} was attached to a suitable multiple of the Green's function  of the conformal Laplacian, with 
pole at the location of the bubble. 
In particular, by the symmetries of $W$, in case $(Z,g_Z) = (\Sp^4,g_{\Sp^4})$ this would be a \textit{global} and transverse-traceless (TT) solution to \eqref{eq:corr-MZ} expressed in  stereographic 
coordinates at $-q$.

Since on $\R^4$ the linearized Bach operator coincides with the bi-Laplacian on TT tensors, to have a solution $F$ as desired one can try to solve a distributional equation like 
\begin{equation}\label{eq:dotB-glob-Z}
	\dot{B}_{g_Z} (F) = {W}^{\overbar{M}}_{k i j l} \, \nabla^k \nabla^l \delta_q \, 
	dx^i dx^j \qquad \hbox{ on } Z, 
\end{equation}
where $\delta_q$ stands for the Dirac delta at $q$.  
Such an equation can be solved by abstract means (up to Lagrange multipliers), yielding a solution of the type 
\begin{equation} \label{eq:typehZ}
F = \left( {W}^{\overbar{M}}_{k i j l} \frac{x^k x^l}{|x|^4} + A_{ij}(x) \right) 
dx^i dx^j \qquad \hbox{ near } q, 
\end{equation}
and the prospect is that the extra term $A_{ij}$ might  have a role in the expansion of Weyl's energy on $Y$. 

While there is in general no hope to control the remainder term 
$A_{ij} $ with an abstract approach, the 
advantage of working with a locally conformally flat metric on $Z$ is that one can pass to the universal cover $\tilde{Z}$  of $Z$  and use the developing map $\Phi : \tilde{Z} \to \Sp^4$, which turns out to be injective in positive Yamabe class, see \cite{schoen-yau-1988-Inventiones-Conf-flat}. Here, 
the latter property allowed to prove positivity of the ADM mass of (non spherical) locally conformally flat manifolds by roughly writing the lift to $\tilde{Z}$ of the Green's function of the conformal Laplacian 
on $Z$ as a multi-pole singular solution. 

The hope in the present situation is to use a similar strategy, in order to write $A_{ij}$ via a formal series expansion. In fact, one could start from the 
above solution on $\Sp^4$, use the conformal covariance of $\dot{B}$  and in a first attempt sum over all pull-backs via deck transformations of $\tilde{Z}$, 
as in  \eqref{eq:serie-pullback}. If this 
would converge, one might project such tensor onto $Z$ and find the metric perturbation $F$ 
as desired.

\bigskip

A model case one can consider is the product $\Sp^1_t \times \Sp^3$, with $\Sp^1_t$ standing for the circle of length 
$\log t$, $t > 1$. The universal cover is $\R^4 \setminus \{0\}$ (or $\Sp^4$ with two antipodal points removed) 
and the deck transformation group is isomorphic to $\Z$. Proceeding as described above, 
from \eqref{eq:serie-pullback} we obtain the 
expression for $\tilde{A}$ given by 
\eqref{eq:A-nonconv}, see Section \ref{sss:sing-corr}. Unfortunately, the  
series constructed in this way does not converge: since however the kernel of the linearized Bach operator contains all Lie derivatives, as a second attempt we  modify each summand in the series by a non-TT element,  allowing to solve the convergence problem, see the discussion before Lemma \ref{lem:Abar-expansion}. We then 
obtain a solution of \eqref{eq:corr-MZ} that behaves as in \eqref{eq:typehZ}.

 Since we modified the metric on $Z$, which was locally flat, by a term like \eqref{eq:typehZ}, in order to properly glue the new metric on $M$ we need to adjust $g_M$ in order to match the extra term $A$.
If $y = \frac{x}{|x|^2}$ (with $x$ coordinates as before) are normal coordinates 
at $p \in M$, we need to add a suitable multiple of a correction $H$ that behaves like
\begin{equation}\label{eq:dotB-glob-M}
	H \simeq \frac{1}{2}\nabla^2_{kl} A_{ij}(0)  
	x^k x^l \, 
	dx^i dx^j \qquad \hbox{ near } p.  
\end{equation}
Using integration by parts and some delicate expansions, we see that the correction to Weyl's energy (compared to $\mathcal{W}(g_M)$) for the connected sum $Y$ with metric profiles as described above is proportional to  
$- W_{\overbar{M}}^{kijl} \nabla^2_{kl} A_{ij}$, 
so  qualitatively  $\nabla^2 A$ plays 
the role of a {\em tensorial mass} that couples 
to the Weyl's tensor of $M$. After a direct calculation, for $Z = \Sp^1_t \times \Sp^3$ the term  $W_{\overbar{M}}^{kijl} \nabla^2_{kl} A_{ij}$ becomes of the type 
\begin{equation}\label{eq:MWW}
 F(t) \; W^M(p) \stell W^{\overbar{M}}(p) 
\end{equation}
where $F(t)$ depends on the length of  $\Sp^1_t$, see Section \ref{subsec:inter-sign} for details. Qualitatively, $F(t)$, which is positive, behaves with respect to $t$ like the ADM mass of $Z = \Sp^1_t \times \Sp^3$, see Remark \ref{r:mass} (and in particular blows-up as $t\to 1^+$). Therefore, we get the desired comparison for the Weyl energies of $g_Y$ and $g_M$  whenever the interaction \eqref{eq:MWW} is \emph{positive}, which can be always achieved provided $M$ is not self-dual or anti self-dual. 

Before discussing the case of a general $Z$, we  comment on an aspect of the above gluing procedure that we find relevant. 
In order to perform a proper matching, we need to modify $g_Z$ so that the tensor $A_{ij}$ vanishes at $q$ together with its gradient
and so that
the term $W^{kijl} \nabla^2_{kl} A_{ij}$ is not affected by this change. To prove this fact we again employ a translational gauge (i.e., we add a Lie derivative), which also allows us to project the hessian of $A$ on symmetrized curvature tensors and streamline many computations.
This is analogous to the construction of normal coordinates by iterated polynomial change of variables, see Section 4 in \cite{wang-yin-2025-pre-ALEcoordinates}.

\bigskip

We next turn to the case of a general locally conformally flat manifold $(Z,g_Z)$ of positive Yamabe class. By a result in \cite{chen-tang-zhu-2012-JDG-4LCF-classification} (which completes the work in \cite{hamilton-1997-CAG-PIC}, \cite{izeki-1995-Inventiones-KleinianGroups}), all such manifolds are diffeomorphic to  connected sums of compact quotients of $\Sp^1 \times \Sp^3$. We exploit this information by choosing a 
special family of metrics on $Z$ which will 
lead to the same conclusions as above. First, fixing one of the summands as just described, we take the corresponding parameter $t$ to be sufficiently close to $1$, in order to obtain a strong interaction term in \eqref{eq:MWW} 
(i.e., $F(t)$ will be large). One can prove, see Proposition \ref{prop:F-exp-quotients}, that even passing to quotients, the interaction term $W^{kijl} \nabla^2_{kl} A_{ij}$ stays positive. Taking  connected sums with other quotients of this type, we choose sufficiently small necks, in order to obtain an interaction term very close to the original one, for capacity reasons, see  Proposition \ref{p:ZZ1}.

\begin{remark}
	Instead of working with a glued manifold $Y\cong Z\#M$, it will be easier in the next sections to perform a simpler gluing and consider $X\cong Z\# \overbar{M}$. This does not create any issue, as in the end it will be sufficient to perform the same procedure with $\overbar{M}$ replacing $M$, see Remark \ref{rem:gluing-diffeo-X}. 
\end{remark}

\bigskip

 Theorem \ref{thm:main2} can be extended to some orbifold settings:

\begin{theorem}\label{thm:gluing-orbifold}
   Let $(M,g_M)$, $(Z,g_Z)$ be two four-dimensional closed, connected, oriented orbifolds.
    Assume that $g_M$ is Bach-flat and that $g_Z$ is locally conformally 
	flat and with positive scalar curvature.
    Let $p\in M$ and $q\in Z$ be orbifold points, both with group $\Gamma\cong \Z_2$, and denote by $Y:=Z\#M$ the orbifold connected sum at $p,q$. Assume that the Weyl tensor of $g_M$ at $p$ is neither self-dual nor anti-self-dual and that the number of orbifold points on $Z$ with group isomorphic to $\Z_2$ is at least $4$. Then there exists a metric $g_Y$ on $Y$ such that
	\begin{equation*}
		\w^Y(g_Y)<\w^M(g_M). 
	\end{equation*}
\end{theorem}

We refer to Section \ref{subsec:orbifold-sums} for the proof and the definitions of orbifolds and orbifold connected sums.
The proof of Theorem \ref{thm:gluing-orbifold} follows along the same lines as that of Theorem \ref{thm:main2}; indeed, by the results in \cite{chen-tang-zhu-2012-JDG-4LCF-classification}, \cite{wang-2025-pre-LCForbifolds} we have a classification of locally conformally flat orbifolds with positive scalar curvature, Theorem \ref{thm:classif-LCF-PSC-orbifolds}, which is similar to that of their smooth counterpart.
We point out that orbifold points corresponding to any given group always arise in even numbers, see Section \ref{subsec:orbifold-sums}.

We expect the result of Theorem \ref{thm:gluing-orbifold} to hold in the more general case in which $Z$ is not conformally equivalent to a football orbifold, see Remark \ref{rem:orbif-gluing-theorem-generalization}.
This is important for what we believe to be one of the principal perspective applications of these gluing results, that is, to the blow-up analysis of (suitable) minimizing sequences for the Weyl energy, which will be briefly described next.

\subsection*{Connected sums and minimization of the Weyl energy}

As already outlined in \cite{malchiodi-malizia-2025-pre-weyl}, one of the most interesting applications of Theorems \ref{thm:main}, \ref{thm:main2} and \ref{thm:gluing-orbifold} concerns the minimization of Weyl's functional in the noncollapsing case.

Assume we have a minimizing sequence $(g_i)_i$ for the Weyl energy on a given closed four-manifold $Y$, and suppose that the volume of geodesic balls is uniformly bounded from below along the sequence\footnote{This can be achieved in suitable classes of four-manifolds, like the class $\mathcal{Y}_2^+$ defined in \cite{chang-gursky-zhang-2020-MathZ-conformallyGAPthm}.}
Then it is reasonable to expect the validity of a bubble convergence/orbifold compactness theorem for  suitable perturbations of Weyl's functional along $(g_i)_i$. Indeed, this holds true in the space of Bach flat metrics with (positive) constant scalar curvature by the works of Tian-Viaclovsky and Streets (\cite{tian-viaclovsky-2005-Inventiones}, \cite{tian-viaclovsky-2005-advances}, \cite{tian-viaclovsky-2008-CMH}, \cite{streets-2010-transactions-removal-singul-Bachflat}), and a possible approach for minimizing sequences of a Sacks-Uhlenbeck type approximation of Weyl's functional was for instance proposed in \cite{anderson-2006-AsianJ-canonical-metrics}.

If we further assume to have an energy identity (that is, the infimum of Weyl's energy on $Y$ coincides with the Weyl energy on the limit orbifold plus the Weyl energy of each bubble), then we could employ Theorems \ref{thm:main}, \ref{thm:main2}, \ref{thm:gluing-orbifold} to glue back together these spaces and obtain a manifold, diffeomorphic to $Y$, with a smooth metric whose Weyl energy is strictly smaller than the supposed infimum.
This would give a contradiction and thus show us that there could not have been bubbling in the first place. As a consequence, $g_i$ subconverges to a smooth Bach flat metric on $Y$.

In other words, we expect these theorems to provide \emph{obstructions} to the loss of topology/formation of orbifold points along minimizing sequences,  as it happens for the Willmore functional (see \cite{bauer-kuwert-2003-IMRN}). 

In this setting, after a suitable scaling the bubbles are expected to converge to Bach-flat and scalar-flat asymptotically locally Euclidean (ALE) orbifolds, and it has been shown in \cite{streets-2010-transactions-removal-singul-Bachflat}, \cite{ache-viaclovsky-2012-GAFA-obstruction-flatALE} that said manifolds are ALE of order at least $2$.

In the particular case of asymptotically Euclidean (AE) manifolds, then $2$ is also the maximal order of decay by the rigidity statement of positive mass theorem. Moreover, it has been shown that these manifolds can be \emph{conformally compactified} into  smooth manifolds (see again \cite{streets-2010-transactions-removal-singul-Bachflat}). Hence we are exactly in the setting of Theorems \ref{thm:main}, \ref{thm:main2}, which, coupled with a positive answer to Singer's conjecture, would provide a complete picture of bubbling via neckpinch in positive Yamabe class. In other words, the only expected bubbling behavior is that of a self-dual manifold which is developing self-dual bubbles, which is something obtainable by reversing the gluing constructions in \cite{donaldson-friedman-1989-nonlin-SDgluing}, \cite{floer-1991-JDG-CPsum}.

However, to have a complete picture we would also need counterparts of our previous theorems for orbifolds, and in these situations we also expect the necessity to consider \emph{higher order obstructions} (as the Weyl tensor might vanish at a given finite order at the orbifold point).

By analogy with smooth manifolds, we expect the bubbles to be self-dual (anti self-dual) orbifolds. This is indeed the situation for the Gibbons-Pope or Page construction (\cite{gibbons-pope-1979-CMP-K3-16-torus}, \cite{page-1978-16-torus-K3}, \cite{topilawa-1987-Inventiones-}, \cite{lebrun-singer-1994-mathannalen-orbifold-ASD-gluing}).
We also point out that a gluing result of self-dual orbifolds was established in \cite{kovalev-singer-2001-GAFA-ASD-orbifold-gluing}, \cite{ache-viaclovsky-2015-JGA-ASDorbifold-gluing} with analytical techniques, see in particular \cite{ache-viaclovsky-2015-JGA-ASDorbifold-gluing}, Theorems 1.13 and 1.14, which suggest that an analogue of Singer's conjecture might hold even in the orbifold setting.

\begin{remark}
    We notice that, although Theorem \ref{thm:gluing-orbifold} does not apply when $Z$ is a football orbifold, there are recent results showing that football orbifolds cannot arise as orbifold limits along bubbling sequences of Einstein manifolds (\cite{ozuch-2024-CPAM-obstructions-integrability-deform}) or locally conformally flat manifolds (\cite{wang-2025-pre-LCForbifolds}). Therefore, we might expect a similar obstruction even in our case.
\end{remark}

\medskip

We end this Introduction with a brief description of the contents of each Section. In Section \ref{sec:prelims} we fix the notation and recall some known facts and preliminary results. In Section \ref{sec:green-LCF}, we construct the correction $F$ on the locally conformally flat manifold $Z$ of the form $(\Sp^1 \times \Sp^3)/G$, starting from the model case of $\Sp^1_t \times \Sp^3$. After that, in Section \ref{sec:Weyl-en-Z} we compute the Weyl energy expansion for the modified metric for $Z$ of the above type and in Section \ref{sec:weyl-en-M} we define the correction $H$ on $M$ and compute the relative energy expansion. Finally, in Section \ref{sec:conclusion} we define the metric on the connected sum, study the sign of \eqref{eq:MWW} and conclude the proof of Theorems \ref{thm:main2} and \ref{thm:gluing-orbifold}.

\section{Preliminaries}\label{sec:prelims}
\subsection{Notation and definitions}

In order to keep computations as clear as possible, throughout the paper we adopt the Einstein summation convention, \emph{but only for repeated indices appearing once as a subscript and once as a superscript}. Thus, for instance, $\sum_i a_i b^i=a_ib^i$, but $\sum_i a_i b_i\not=a_i b_i$.  As a consequence, we will usually denote local coordinates of a vector $x\in \R^n$ as $x^i$, but sometimes we will also use the notation $x_i$.

\medskip

We define the Riemann tensor as
\begin{gather*}
    R(X,Y)Z:=\nabla_X\nabla_Y Z-\nabla_Y\nabla_X Z-\nabla_{[X,Y]}Z, \\
    R(X,Y,Z,W):=g\big(R(X,Y)Z,W\big).
\end{gather*}
In local coordinates,
\begin{equation*}
R_{ijkl}=R\indices{_{ijk}^{\alpha}}g_{\alpha l}, \qquad  R_{ij}=R_{kijl}g^{kl}=R_{iklj}g^{kl}=R\indices{_{kij}^k}, \qquad R_g=R_{ij}g^{ij}.
\end{equation*}
The Weyl tensor in dimension $4$ is defined by \eqref{eq:curv-decomp-4} and, in local coordinates, is given by
\begin{equation}\label{eq:weyl-local-coord-expr}
    W_{ijkl}=R_{ijkl}-\frac{1}{n-2}\big(R_{jk} g_{il}+R_{il}g_{jk}-R_{ik}g_{jl}-R_{jl}g_{ik}\big)+\frac{R_g}{(n-1)(n-2)}(g_{jk}g_{il}-g_{ik}g_{jl}).
\end{equation}

The curvature tensor can be regarded as a symmetric endomorphisms on $2$-forms
\begin{equation*}
    \widehat{R}\colon \Lambda^2(T^*M)\to \Lambda^2(T^*M),
\end{equation*}
cf. \cite[Lecture 5]{viaclovsky-notes-ParkCity} or \cite{besse-book-2008}.
Indeed, given a local frame $\{e_1,\dots, e_n\}$ with local coframe $\{e^1,\dots, e^n\}$, then the family $\{e^i\wedge e^j\}_{i<j}$ defines a local frame for $\Lambda^2(T^*M)$ and we can define
\begin{equation}\label{eq:curv-operator-def}
    \widehat{R}(e^i\wedge e^j):=\frac{1}{2}R_{ijlk}e^k\wedge e^l=\sum_{k<l}R_{ijlk}e^k\wedge e^l.
\end{equation}
Notice the position reversal between $k$ and $l$. 
In the same way, we can regard the Weyl tensor $W$ as a symmetric \emph{trace-free} endomorphism $\widehat{W}$ on $\Lambda^2(T^*M)$.
In particular, if $n=4$ and $M$ is oriented, then the Hodge star operator acting on 2-forms, $*\colon \Lambda^2(T^*M)\to \Lambda^2(T^*M)$, satisfies $*^2=\mathrm{Id}$ and, under its action, $\Lambda^2$ decomposes into
\begin{equation*}
    \Lambda^2(T^*M)=\Lambda^2_+(T^*M)\oplus \Lambda^2_-(T^*M),
\end{equation*}
which are the $\pm 1$-eigenspaces of $*$ respectively and satisfy $\dim(\Lambda^2_{\pm}(T^*M))=3$. The Weyl curvature operator decomposes accordingly as follows:
\begin{equation*}
    \widehat{W}=\widehat{W}_+ +\widehat{W}_-, \qquad \qquad \widehat{W}_{\pm}\colon \Lambda^2_{\pm}(T^*M)\to\Lambda^2_{\pm}(T^*M).
\end{equation*}
$\widehat{W}_+$, $\widehat{W}_-$ are called \emph{self-dual} and \emph{anti self-dual} Weyl curvatures respectively (as well as $W_+, W_-$ which are the associated $(0,4)$-tensors).

We recall that the norm of a $(0,4)$ curvature-type tensor $\abs{T}^2=T^{ijkl}T_{ijkl}$ differs by a factor of $4$ from the norm $\big\lVert{\widehat{T}}\big\rVert^2$ of $T$ seen as a symmetric endomorphism of $\Lambda^2(T^*M)$:
\begin{equation*}
    \abs{T}^2=4\big\lVert{\widehat{T}}\big\rVert^2.
\end{equation*}

\subsection{Decomposition of pairwise symmetric $4$-tensors and inverted coordinates}

The next lemma provides a decomposition of $4$-tensors which belong to $\mathrm{Sym}(2)\otimes \mathrm{Sym}(2)$. This is well-known, as it provides a way to recover the second-order expansion of the metric in normal coordinates without using Jacobi fields, see Proposition 2.4 and Section 4 in \cite{wang-yin-2025-pre-ALEcoordinates}. However, we include a  proof for the reader's convenience.

\begin{lemma}\label{lem:tensor-projection}
    Let $A=A_{kijl}$ be a covariant $4$-tensor in an $n$-dimensional (real) inner product space $V$, and assume that $A$ is symmetric in $\{i,j\}$ and in $\{k,l\}$, namely
    \begin{align}\label{eq:A-symmetries}
        A_{kijl}=A_{kjil}=A_{lijk} \qquad \forall i,j,k,l.
    \end{align}
    Then $A$ uniquely decomposes as the direct sum
    \begin{align}\label{eq:A-decomposition}
        A=T\oplus C,
    \end{align}
    where $T$ can be written as
    \begin{equation}\label{eq:T-component-def}
        T_{kijl}=\frac{2}{3}\big( R_{kijl}+R_{lijk}),
    \end{equation}
    for a suitable $R$ with the symmetries of Riemann's tensor, and
         $C$ is a gauge term that can be written as
        \begin{align}\label{eq:C-gauge-formula}
        C_{kijl}=X_{i,jkl}+X_{j,ikl},
        \end{align}
        where $X$ is symmetric in the last three entries.
\end{lemma}
\begin{proof}
    Define
    \begin{align}\label{eq:Tens-def}
        R_{kijl}=\mathcal{P}(A):=\frac{1}{4}\big( A_{kijl}-A_{ikjl}-A_{kilj}+A_{iklj}\big).
    \end{align}
    Then it is straightforward to verify that $R$ has the symmetries of Riemann's tensor; in particular, the antisymmetric property in the first and last two entries follows from the definition of $R$, while the symmetry $R_{kijl}=R_{jlki}$ and the first Bianchi identity are a consequence of both \eqref{eq:A-symmetries} and \eqref{eq:Tens-def}. 
    
    At this point, define $T$ as in \eqref{eq:T-component-def},
    and let $C:=A-T$.  Then 
    \begin{align*}
        C_{kijl}&=A_{kijl}-\frac{1}{6}\big(A_{kijl}-A_{ikjl}-A_{kilj}+A_{iklj}+A_{lijk}-A_{iljk}-A_{likj}+A_{ilkj}\big) \\
        &=\frac{2}{3}A_{kijl}-\frac{1}{3}A_{iklj}+\frac{1}{6}\big( A_{ikjl}+A_{kilj}+A_{iljk}+A_{likj}\big),
    \end{align*}
    and a direct computation (together with \eqref{eq:A-symmetries}) shows that 
    \begin{align}\label{eq:C-curv-project-zero}
        \mathcal{P}(C)=\frac{1}{4}\big( C_{kijl}-C_{ikjl}-C_{kilj}+C_{iklj}\big)=0,
    \end{align}
    namely $C$ has no curvature part.

    We next show that we can further decompose $C$ as in \eqref{eq:C-gauge-formula}.
    Define $X$ as
    \begin{align*}
        2X_{i,jkl}:=C_{kijl}+C_{jikl}-C_{ijkl}.
    \end{align*}
    Since $R$ has the symmetries of Riemann's tensor, it is immediate to see that $T$ satisfies the relations \eqref{eq:A-symmetries}, which are therefore also satisfied by $C$. As a consequence, 
    \begin{align*}
        2 X_{i,kjl}=C_{jikl}+C_{kijl}-C_{ikjl}=C_{jikl}+C_{kijl}-C_{ijkl}=2 X_{i,jkl}.
    \end{align*}
    Moreover, using \eqref{eq:C-curv-project-zero} we also got
    \begin{align*}
        2\big(X_{i,jkl}-X_{i,jlk}\big)&=C_{kijl}+C_{jikl}-C_{ijkl}-C_{lijk}-C_{jilk}+C_{ijlk} \\
        &=C_{jikl}-C_{ijkl}-C_{jilk}+C_{ijlk}=0.
    \end{align*}
    Hence $X$ is indeed symmetric in the last three entries, and one further has
    \begin{align*}
        2\big( X_{i,jkl}+X_{j,ikl}\big)=C_{kijl}+C_{jikl}-C_{ijkl}+C_{kjil}+C_{ijkl}-C_{jikl}=2 C_{kijl},
    \end{align*}
so that \eqref{eq:C-gauge-formula} holds.
Finally, we notice that any $T$ of the form \eqref{eq:T-component-def} trivially satisfies $\mathcal{P}(T)=T$, and that any $C$ of the form \eqref{eq:C-gauge-formula} satisfies $\mathcal{P}(C)=0$, so the decomposition \eqref{eq:A-decomposition} is indeed a direct sum.
\end{proof}

Having a tensor with the symmetries \eqref{eq:T-component-def} of $T$  (instead of only being pairwise symmetric as in \eqref{eq:A-symmetries}) is useful when considering inverted normal coordinates:

\begin{lemma}\label{lem:perturb-metr-inverse-exp}
    Let $g_0$ be a Riemannian metric on $M$ and let $g:=g_0+tH$ be a perturbation that, in normal coordinates at a point $p\in M$, has the following expression:
    \begin{align*}
        g(x)=\Big[\delta_{ij}-\frac{1}{3}R_{kijl}z^k z^l+t T_{kijl}\frac{z^k z^l}{\abs{z}^4}+ O^{(4)}(\abs{z}^3)\Big] dz^i dz^j , \qquad \text{as $\abs{z}\to 0$}.
    \end{align*}
    Assume that $T$ can be written as in \eqref{eq:T-component-def}, and let $f$ be a positive function on $M$ such that $f(z)=\abs{z}^{-2}$ for $\abs{z}$ small. Then, in inverted normal coordinates $y=z/\abs{z}^2$, the inverted metric $\tilde{g}:=f^2 g$ has the following expansion:
     \begin{align*}
        \tilde{g}(y)=\Big[\delta_{ij}-\frac{1}{3}R_{kijl}\frac{y^k y^l}{\abs{y}^4}+t T_{kijl}y^k y^l+O^{(4)}(\abs{y}^{-3})\Big]dy^i dy^j, \qquad \text{as $\abs{y}\to+\infty$}.
    \end{align*}
\end{lemma}
\begin{proof}
     If $I(y)=y/\abs{y}^2$ denotes the inversion map, then
    \begin{align*}
        \tilde{g}=(f\circ I)^2(y) I^*\Big(\Big[\delta_{ij}-\frac{1}{3}R_{kijl}z^k z^l+t T_{kijl}\frac{z^k z^l}{\abs{z}^4}+ O^{(4)}(\abs{z}^3)\Big] dz^i dz^j\Big).
    \end{align*}
    Recalling that
    \begin{equation*}
    I^*(dz^{\alpha})=d\Big(\frac{y^{\alpha}}{\abs{y}^2}\Big)=\frac{dy^{\alpha}}{\abs{y}^2}-2\frac{y^{\alpha}(y\cdot dy)}{\abs{y}^4}=\frac{1}{\abs{y}^2}\Big(\delta^{\alpha }_s-2\frac{y^{\alpha} y_s}{\abs{y}^2}\Big) dy^{s},
\end{equation*}
one obtains
\begin{align*}
        \tilde{g}=\abs{y}^4&\Big[\delta_{\alpha \beta}-\frac{1}{3}R_{k\alpha \beta l}\frac{y^k y^l}{\abs{y}^4}+t T_{k\alpha \beta l}y^k y^l+O^{(4)}(\abs{y}^{-3})\Big] \\
        & \qquad \qquad \times\frac{1}{\abs{y}^4}\Big(\delta^\alpha_\mu-2\frac{y^\alpha y_\mu}{\abs{y}^2}\Big)\Big(\delta^\beta_\nu-2\frac{y^\beta y_\nu}{\abs{y}^2}\Big)dy^\mu dy^\nu, \qquad \text{as $\abs{y}\to+\infty$}.
    \end{align*}
Computing the $dy^i dy^j$-coefficient of the above tensor yields
\begin{align}
\notag
    \tilde{g}_{ij}=\Big[\delta_{\alpha \beta}&-\frac{1}{3}R_{k\alpha \beta l}\frac{y^k y^l}{\abs{y}^4}+t T_{k\alpha \beta l}y^k y^l+O^{(4)}(\abs{y}^{-3})\Big] \\
    \label{eq:inv-coord-pert-general-form}
    &\qquad \times \Big(\delta_i^\alpha \delta_j^\beta -2\delta_i^\alpha\frac{y^\beta y_j}{\abs{y}^2}-2\delta^\beta_j\frac{y^\alpha y_i}{\abs{y}^2}+4\frac{y^\alpha y^\beta y_i y_j}{\abs{y}^4}\Big), \qquad \text{as $\abs{y}\to+\infty$}.
\end{align}
The result now follows by expanding the above expression and using the symmetries of Riemann's tensor and of $T$ (cf. \eqref{eq:T-component-def}) to make some terms vanish.
\end{proof}

\begin{remark}
    If we only assume $T$ to be pairwise symmetric as in \eqref{eq:A-symmetries}, then in \eqref{eq:inv-coord-pert-general-form} we get additional terms in $T$ which grow quadratically at infinity.
\end{remark}

\section{Construction of the modified metric on the LCF manifold}\label{sec:green-LCF}
The objective of this section is to define a correction term for the metric on the locally conformally flat manifold $(Z,g_Z)$ in some specific cases. This will allow us to perform the gluing with the other (non SD/ASD) manifold $(M,g_M)$.

To begin, recall that we are considering $4$-manifolds admitting a locally conformally flat metric with positive scalar curvature. These manifolds are classified according to the following theorem by Chen, Tang and Zhu (see also \cite{hamilton-1997-CAG-PIC}, \cite{izeki-1995-Inventiones-KleinianGroups}):

\begin{theorem}[\cite{chen-tang-zhu-2012-JDG-4LCF-classification}]\label{thm:CTZ-LCF-classification}
    Let $(Z,g_Z)$ be a closed, oriented, locally conformally flat $4$-manifold with positive scalar curvature.
    Then $Z$ is diffeomorphic to $\Sp^4$ or to a connected sum of manifolds of type $(\R\times \Sp^3)/G$, where $G$ is a cocompact, fixed point free, discrete subgroup of the isometry group of the standard metric on $\R \times \Sp^3$. In particular, if $Z\not\simeq\Sp^4$, then a finite cover of $Z$ is diffeomorphic to a connected sum of copies of $\Sp^1 \times \Sp^3$.
\end{theorem}

\begin{remark}
    The above statement slightly differs from the one of the Main Theorem in \cite{chen-tang-zhu-2012-JDG-4LCF-classification}, which concerns the classification of closed $4$-manifolds with \emph{positive isotropic curvature} (PIC). However, closed $4$-manifolds admit a PIC metric if and only if they admit a locally conformally flat metric with positive scalar curvature, cf. \cite[Corollary 2]{chen-tang-zhu-2012-JDG-4LCF-classification}. Moreover, the last statement in Theorem~\ref{thm:CTZ-LCF-classification} follows by Corollary~1 of the same paper.
\end{remark}

Theorem \ref{thm:CTZ-LCF-classification} will allow us to construct the modified metric rather explicitly as any $Z$ will be diffeomorphic to one of the manifolds classified as above. The model case is that of $\Sp^1 \times \Sp^3$ with its standard metric, whose universal cover $\R\times \Sp^3$ is conformal to $\R^4 \backslash\{0\}$, allowing us to define the metric correction as the series of pullbacks (with respect to the action of the fundamental group) of a suitably defined (singular) tensor on $\R^4$. We will then employ a \acc perturbative'' argument in order to extend the construction to other quotients of $\R\times \Sp^3$. Finally, in Section \ref{subsec:conn-sums} we will deal with the general case of connected sums.

\subsection{Construction of the modified metric on $\Sp^1 \times \Sp^3$}\label{subsec:green-cylinder}

We represent $\Sp^1 \times \Sp^3$ as a cylinder
\begin{align}\label{eq:cylinder-def}
    (Z,g):= \bigg(\frac{[-R/2,R/2]\times \Sp^3}{\sim}, g\bigg),
\end{align}
where $R>0$, and we are identifying the points $(-R/2,q)$ with $(R/2,q)$, for $q\in \Sp^3$. However, instead of the standard metric $g_0=dr^2+ g_{\Sp^3}$, we will consider the conformal metric 
\begin{align}\label{eq:g-metr-def}
    g:=\varphi(r) e^{2r}g_0 + (1-\varphi(r))g_0,
\end{align}
where $\varphi=\varphi_R$ is a smooth cutoff function with $\varphi(r)\equiv 1$ for $\abs{r}<R/100$ and $\varphi(r)\equiv 0$ for $\abs{r}>R/10$.
By our choice, $g$ is Euclidean in a tubular neighborhood of $\{0\}\times \Sp^3$. \\
Denote by $\tilde{g}$ the pullback of $g$ to the universal cover $\pi\colon\tilde{Z}:=\R \times \Sp^3\to Z$. Since the fundamental group $\pi_1(Z)\simeq\Z$ acts on $\R\times \Sp^3$ by translations (and a generator is the right translation $\gamma(r,q)=(r+R,q)$), $\tilde{g}$ will be Euclidean in a tubular neighborhood of $\big\{\{nR\}\times \Sp^3 \mid n\in \Z\big\}$.

Consider the map (here we regard $\Sp^3\subset \R^4$)
\begin{gather}\label{eq:Phi-def-conf}
    \Phi\colon \R\times \Sp^3 \to \R^4\backslash \{0\}, \\
    \notag
    (r,q)\to e^r q.
\end{gather}
$\Phi$ is conformal and $\bar{g}:=\Phi_*(\tilde{g})$ is a conformal metric on $\R^4\backslash\{0\}$ such that, letting $t:=e^R$,
\begin{align}\label{eq:gbar-def}
    \bar{g}(x)=e^{2\phi(x)}g_E=
    \begin{cases*}
        t^{-2n}g_{E} & \text{in a tubular neighborhood of $\partial B_{t^n}(0)$\footnotemark} \\
        \abs{x}^{-2} g_E & \text{away from $\{\partial B_{t^n}(0)\mid n\in\Z\}$.}
    \end{cases*}
\end{align}
\footnotetext{The Euclidean size of the neighborhood depends on $n$, while the size with respect to $\bar{g}$ is constant.}
Moreover, the fundamental group will act on $\R^4\backslash\{0\}$ by dilations (a generator is the map $\psi_\gamma(y)=\psi(y):=ty=e^R y$), and $\bar{g}$ is invariant under the action of such transformations.

\subsubsection{The singular correction}\label{sss:sing-corr}

On $\R^4$, define the singular tensor 
\begin{equation}\label{eq:h-eucl-def}
    h(x):=-\frac{1}{3}W_{kijl}\frac{(x+e_4)^k (x+ e_4)^l}{\abs{x+e_4}^4}\,dx^i dx^j,
\end{equation}
where $e_4=(0,0,0,1)\in \R^4$. Then $h$ is TT in $\R^4\backslash\{-e_4\}$, and a direct computation shows that
\begin{align}\label{eq:lin-bach-h}
    \dot{B}_{g_E}(h)=\Delta^2(h)=0 \qquad \text{in $\R^4\backslash\{-e_4\}$.}
\end{align}
Recalling that the Bach operator satisfies the covariance law $B_{e^{2\phi}g}=e^{-2\phi} B_g$, we deduce that
\begin{align}\label{eq:lin-bach-conf-cov}
    \dot{B}_{e^{2\phi} g}(T)=\lim_{s\to 0}\frac{B_{e^{2\phi}g+sT}-B_{e^{2\phi}g}}{s}=\lim_{s\to0}\frac{B_{e^{2\phi}(g+se^{-2\phi}T)}-B_{e^{2\phi}g}}{s}=e^{-2\phi}\dot{B}_g(e^{-2\phi} T).
\end{align}
Therefore, if we let
\begin{align}\label{eq:hbar-def}
    \bar{h}:=e^{2\phi}h,
\end{align}
then 
\begin{align}\label{eq:lin-bach-hbar}
    \dot{B}_{\bar{g}}(\bar{h})=0 \qquad \text{in $\R^4\backslash \{0,-e_4\}$,}
\end{align}
where here $\phi,\bar{g}$ are given by \eqref{eq:gbar-def} (but this holds more generally for any $g$ and $\bar{g}=e^{2\phi}g$).

We would now like to define a singular tensor on the manifold in \eqref{eq:cylinder-def}, whose behavior around the singular point is similar to that of $h$. To define such an object, we might consider the series over $n\in\Z\simeq\pi_1(Z)$ of the pullbacks $\psi_n^* \bar{h}$. This would (theoretically) provide us with a tensor which lifts through $\Phi^{-1}$ to an equivariant tensor on $\R\times\Sp^3$, allowing us to pass it to the quotient.

\medskip

\noindent
\textbf{First attempt.} On $\R^4\backslash\{0\}$ with the standard Euclidean coordinates, let us \emph{formally} define the (singular) tensor $\tilde{A}=\tilde{A}_{ij}(y)dy^i dy^j$ given by
\begin{align}\label{eq:serie-pullback}
    \tilde{A}_{ij}(y):=\sum_{n\in\Z,n\not=0}(\psi_n^*\bar{h})_{ij}(y).
\end{align}
By definition of pullback and of $\psi_n$,
\begin{align}
\notag
    (\psi_n^*\bar{h})_{ij}(y)=&e^{2(\phi\circ\psi_n)(y)}h_{\psi_n(y)}\big(d_y\psi_n(\p_i|_y),d_y\psi_n(\p_j|_y)\big) \\
    \notag
    =&e^{2(\phi\circ\psi_n)(y)}(\p_i\psi_n(y))^k(\p_j\psi_n(y))^l h_{\psi_n(y)}(\p_k|_{\psi_n(y)},\p_l|_{\psi_n(y)}) \\
    \label{eq:hbar-pullback-eq}
    =&e^{2(\phi\circ\psi_n)(y)}t^{2n}h_{ij}(t^ny).
\end{align}
Moreover, since $\psi_n^*\bar{g}=\bar{g}$, we see that, \underline{for $y$ close to $-e_4$} (i.e. where $e^{2\phi(y)}\equiv1$), one has
\begin{align}\label{eq:conf-fact-pullback-eq}
    e^{2(\phi\circ\psi_n)(y)}\psi_n^* g_E=e^{2\phi(y)}g_E=g_E \quad \Longrightarrow \quad e^{2(\phi\circ\psi_n)(y)}=\abs{\p_i \psi_n(y)}^{-2}=t^{-2n},
\end{align}
therefore
\begin{align*}
    (\psi_n^*\bar{h})_{ij}(y)=h_{ij}(t^n y), \qquad \text{for $y\sim -e_4$.}
\end{align*}
Substituting in the formula for $\tilde{A}_{ij}$ we formally get, by \eqref{eq:h-eucl-def},
\begin{align} \label{eq:A-nonconv}
    \tilde{A}_{ij}(y)=-\frac{1}{3}\sum_{n\in\Z,n\not=0}W_{\mu i j \nu}\frac{(t^n y +e_4)^\mu(t^n y + e_4)^\nu}{\abs{t^n y + e_4}^4}, \qquad \text{for $y\sim -e_4$.}
\end{align}
However, we immediately notice that, for any $t>1$ (fixed), when $n<0$ is large enough one has
\begin{align*}
  W_{\mu i j \nu} \frac{(t^n y +e_4)^\mu(t^n y + e_4)^\nu}{\abs{t^n y + e_4}^4}\simeq W_{4ij4}, \qquad \text{for $y\sim -e_4$}, 
\end{align*}
therefore the above series does not converge and $\tilde{A}$ is not well-defined.
However, this lack of convergence is only caused by a bad choice of gauge term for $h$, which will be adjusted below.

\medskip

\noindent
\textbf{Second attempt with a gauged $h$.} We start by recalling that, at a Bach-flat metric, all Lie derivatives of the metric are kernel elements for the linearized Bach operator.\footnote{Indeed, if $X$ is a vector field and $(\varphi_t)_t$ the associated flow at time $t$, then differentiating the identity $B_{\varphi_t^*g}=\varphi_t^* B_g=0$ at time $t=0$ yields $\dot{B}_g(\mathcal{L}_Xg)=\mathcal{L}_X B_g=0$.} The idea now is to gauge $h$ by adding suitable Lie derivatives of $g_E$ in such a way to obtain a well-defined and \acc nice'' tensor when considering the series of pullbacks as above.

Let $0\leq\eta_0,\eta_1\leq 1$ be smooth cutoff functions such that 
\begin{align}\label{eq:cutoffs-cylinder}
    \eta_0(s)=\begin{cases*}
        1 & \text{for $s\leq t^{-3/8}$} \\
        0 & \text{for $s \geq t^{-1/8}$,} 
        \end{cases*}
        \qquad \eta_1(s)=\begin{cases*}
            1 & \text{for $t^{-1/8}\leq s \leq t^{1/8}$} \\
            0 & \text{for $s\leq t^{-3/8}$ or $s\geq t^{3/8}$.}
    \end{cases*}
\end{align}
Define the $1-$forms
\begin{align}\label{eq:omega_0-def}
    \omega_0(y):=\frac{1}{6}\eta_0(\abs{y}) W_{4ik4}y^k dy^i,
\end{align}
and 
\begin{align}
\notag
    \omega_1(y):=\frac{1}{3}\eta_1(\abs{y})\Big\{&\frac{1}{2}C_0(t) W_{4ik4}(y+e_4)^k+ C_1(t)\Big[\big(W_{kil4}+W_{4ilk}\big)(y+e_4)^k(y+e_4)^l \\
    \label{eq:omega_1-def}
    & \qquad \quad -2W_{4ik4}(y+e_4)^k(y+e_4)_4 
    +W_{4kl4}\delta_{i4}(y+e_4)^k (y+e_4)^l\Big]\Big\}\,dy^i,
\end{align}
where the constants $C_0(t)$ and $C_1(t)$ are given by (cf. proof of Lemma \ref{lem:Abar-expansion})
\begin{align}\label{eq:C_0C_1-def}
    C_0(t):=\sum_{n>0}\bigg[\frac{1}{(1-t^n)^2}+\frac{1}{(1-t^{-n})^2}-1\bigg], \qquad C_1(t):=\sum_{n>0}\bigg[\frac{t^n}{(1-t^n)^3}+\frac{t^{-n}}{(1-t^{-n})^3}\bigg].
\end{align}
Let $X_0, X_1$ be the dual vector fields (with respect to $g_E$) to $\omega_0,\omega_1$, and define the gauged tensor
\begin{align}\label{eq:-zeta-h-gauged-eucl}
    \xi:=h+\mathcal{L}_{X_0}g_E + \mathcal{L}_{X_1}g_E.
\end{align}
Recalling that $(\mathcal{L}_{X}g)_{ij}=\nabla_i\omega_j+\nabla_j\omega_i$, a straightforward computation shows 
\begin{align*}
    \xi_{ij}(y)=&-\frac{1}{3}W_{\mu i j \nu}\frac{(y+e_4)^\mu(y+e_4)^\nu}{\abs{y+e_4}^4} \\
    &+\begin{cases*}
        \frac{1}{3}W_{4ij4} & \text{near $y=0$,} \\
        \frac{1}{3}\Big\{C_0(t) W_{4ij4}+C_1(t)\Big[(W_{kij4}+W_{4ijk})(y+e_4)^k-4 W_{4ij4}(y+e_4)_4\Big]\Big\} & \text{near $y=-e_4$.}
    \end{cases*}
\end{align*}
In particular, we see that $\xi$ is TT both near the origin and near $-e_4$ (however, it is not globally TT since it is not transverse in the cutoff regions where $\eta_0,\eta_1\neq0,1$), and, by construction, $\xi$ is again a solution for $\dot{B}(\xi)=0$ in $\R^4\backslash\{-e_4\}$.  

Let now $\bar{\xi}:=e^{2\phi}\xi$ in $\R^4\backslash\{0\}$. By \eqref{eq:lin-bach-conf-cov}, $\bar{\xi}$ is a solution to \eqref{eq:lin-bach-hbar}. Define the tensor
\begin{align}\label{eq:T-def}
\bar{T}:=\sum_{n\in\Z}\psi_n^*\bar{\xi}=\sum_{n\in\Z}\psi_n^*(e^{2\phi}\xi).
\end{align}
Using \eqref{eq:hbar-pullback-eq} (with $\bar{\xi}$ in place of $\bar{h}$), \eqref{eq:conf-fact-pullback-eq} and the definition of $\bar{\xi}$, we see that, \underline{for $y\sim -e_4$}, there holds
\begin{align}
\notag
    \bar{T}_{ij}(y)=\sum_{n\in\Z}\xi_{ij}(t^ny)=&-\frac{1}{3}W_{\mu i j \nu}\frac{(y+e_4)^\mu (y+e_4)^\nu}{\abs{y+e_4}^4}-\frac{1}{3}\sum_{n>0}W_{\mu i j \nu}\frac{(t^n y +e_4)^\mu(t^n y + e_4)^\nu}{\abs{t^n y + e_4}^4} \\
    \notag
    &-\frac{1}{3}\sum_{n>0}\bigg[W_{\mu i j \nu}\frac{(t^{-n} y +e_4)^\mu(t^{-n} y + e_4)^\nu}{\abs{t^{-n} y + e_4}^4}-W_{4ij4}\bigg] \\
    \label{eq:barT-local-formula}
    +\frac{1}{3}\Big\{&C_0(t) W_{4ij4}+C_1(t)\Big[(W_{kij4}+W_{4ijk})(y+e_4)^k-4 W_{4ij4}(y+e_4)_4\Big]\Big\}. 
\end{align}
The expression will be different at a generic point $y\in\R^4\backslash\{0\}$ (due to the non-constant conformal factor $e^{2\phi}$ and the change in gauge terms); however, it is not difficult to see that $\bar{T}$ is pointwisely convergent and well-defined on $\R^4\backslash\big(\{0\}\cup \{-t^n e_4\mid n\in\Z\}\big)$. Moreover, we got
\begin{align}\label{eq:Tbar-lin-bach-eq}
    \dot{B}_{\bar{g}}(\bar{T})=0    \qquad \text{in} \,\,\,\R^4\backslash\big(\{0\}\cup \{-t^n e_4\mid n\in\Z\}\big). 
\end{align}

\begin{lemma}\label{lem:Abar-expansion}
Let us write, for $y\sim -e_4$,
\begin{align*}
    \bar{T}_{ij}(y)=-\frac{1}{3}W_{\mu i j \nu}\frac{(y+e_4)^\mu (y+e_4)^\nu}{\abs{y+e_4}^4}+\bar{A}_{ij}(y),
\end{align*}
where $\bar{T}$ is given by \eqref{eq:T-def}. Then $\bar{A}_{ij}(y)$ is TT for $y$ close to $-e_4$, and moreover
\begin{align}\label{eq:AgradAbar-vanish}
    \bar{A}(-e_4)=0, \qquad \quad \nabla \bar{A}({-e_4})=0,
\end{align}
\begin{align}
\notag
    \p^2_{kl}\bar{A}_{ij}(-e_4)=-\frac{1}{3}{C_2(t)}\bigg[&\big(W_{kijl} + W_{lijk}\big)-4\delta_{4k}\big(W_{lij4}+W_{4ijl}\big)-4\delta_{4l}\big(W_{kij4}+W_{4ijk}\big) \\
    \label{eq:hessian-Abar-formula}
    &\qquad  \qquad \qquad \qquad \qquad -4\delta_{kl} W_{4ij4} +24 \delta_{4k}\delta_{4l} W_{4ij4}\bigg],
\end{align}
where
\begin{align}\label{eq:C_2(t)}
    C_2(t):=\sum_{n\in\Z,n\not=0}\frac{t^{2n}}{(1-t^n)^4}=2\sum_{n>0}\frac{t^{2n}}{(1-t^n)^4}.
\end{align}
\end{lemma}
\begin{proof}
By \eqref{eq:barT-local-formula}, we see that $\bar{A}_{ij}$ is indeed TT near $-e_4$.
    Again from \eqref{eq:barT-local-formula} we have
    \begin{align*}
        \bar{A}_{ij}(-e_4)=-\frac{1}{3}W_{4ij4}\sum_{n>0}\bigg[\frac{1}{(1-t^n)^2}+\frac{1}{(1-t^{-n})^2}-1\bigg]+\frac{1}{3}C_0(t) W_{4ij4}=0,
    \end{align*}
    where the last identity follows by the definition of $C_0(t)$ in \eqref{eq:C_0C_1-def}. The other identities are similarly proved by straightforward computations.
\end{proof}

By Lemma \ref{lem:Abar-expansion}, we see that $\bar{T}$ is almost the singular correction we are looking for (cf. Lemma \ref{lem:inter-term-formula}). However, when taking connected sum, in order to match the hessian of $\bar{A}$, we would also like to add a singular perturbation to the metric on $M$, which, once inverted, will provide a quadratic term corresponding to the hessian of $\bar{A}$. However, this is not possible in general as, in order for the inverted perturbation to have the right form, we would need $\p^2_{kl}\bar{A}_{ij}$ to possess the same symmetries of Riemann's tensor (or, more precisely, of a symmetrized version of it, see \ref{lem:tensor-projection}) and this property is not satisfied in general.
To fix this problem, we can add a (final) gauge term to $\xi$ which will allow us to project the hessian at $-e_4$ onto the right space of tensors.
This will allow us to obtain the following result:

\begin{proposition}\label{prop:F-expansion-cylinder}
    Let $(Z,g)$ be given as in \eqref{eq:cylinder-def}, \eqref{eq:g-metr-def} and let $q:=(0,S)\in Z$, where $S=-e_4\in\Sp^3$ denotes the south pole. There exists a symmetric tensor $F$, which is smooth in $Z\backslash\{q\}$, such that $\dot{B}_g(F)=0$ in $Z\backslash\{q\}$ and, in suitable Euclidean coordinates centered at $q$, 
    \begin{align}\label{eq:F-tensor-cylinder}
        F_{ij}(x)=-\frac{1}{3}W_{\mu i j \nu}\frac{x^\mu x^\nu}{\abs{x}^4}+A_{ij}(x),
    \end{align}
    where $A$ is smooth in a neighborhood of zero and satisfies
    \begin{align}\label{eq:AgradA-vanish}
        A(0)=0, \qquad \qquad \nabla A(0)=0,
    \end{align}
    \begin{align}
    \notag
        W^{kijl}\p^2_{kl} A_{ij}(0)=&-\frac{1}{3}C_2(t)\bigg[W^{kijl}\big(W_{kijl}+W_{lijk}\big)-4W^{4ijl}\big(W_{4ijl}+W_{lij4}\big) \\
    \label{eq:A-interaction-form}
    &\qquad \qquad \qquad-4 W^{kij4}\big(W_{kij4}+W_{4ijk}\big)+24W^{4ij4}W_{4ij4}\bigg], 
    \end{align}
    where $C_2(t)$ is given by \eqref{eq:C_2(t)}.
    Moreover, if we let
    \begin{align}\label{eq:T-hessian-def}
        T_{kijl}:=\frac{1}{2}\p^2_{kl}A_{ij}(0),
    \end{align}
    then $T$ can be written as in \eqref{eq:T-component-def}.
\end{proposition}

\begin{remark}\label{r:mass}
    As we will see in Lemma \ref{lem:hessian-identity-W^P},  the main interaction 
    term for the Weyl energy is indeed given by \eqref{eq:A-interaction-form}. We notice that the 
    ADM mass of (a scalar-flat conformal blow-up of) $\Sp^1_t \times \Sp^3$ 
    is proportional to $\sum_{n\in\Z,n\not=0}\frac{t^{n}}{(1-t^n)^2}$. 
    Although not identical, the expression is similar to the one in \eqref{eq:C_2(t)}, 
    with a related qualitative behavior in the limits $t \to 1^+$ or $t \to +\infty$. 
    \end{remark}

\begin{proof}
To begin, we notice that the tensor $\bar{T}$ defined in \eqref{eq:T-def} is equivariant (by construction) under the action of $\psi_n$, $\forall n\in\Z$. As a consequence, we could lift $\bar{T}$ to an equivariant tensor on $\R\times\Sp^3$, which therefore defines a singular tensor on $Z$ (which is only singular at $(0,S)$). Moreover, by Lemma \ref{lem:Abar-expansion}, the tensor $\bar{A}$ satisfies (after a translation) properties \eqref{eq:AgradA-vanish} and \eqref{eq:A-interaction-form}. The only problem is that the hessian of $\bar{A}$ does not possess the symmetries of \eqref{eq:T-component-def}, but we can correct this last issue by adding another gauge term.

Let $A_{kijl}:=\p^2_{kl}\bar{A}_{ij}(-e_4)$. Then $A_{kijl}$ satisfies \eqref{eq:A-symmetries}, so, by Lemma \ref{lem:tensor-projection}, we can decompose $A$ as in \eqref{eq:A-decomposition}, \eqref{eq:C-gauge-formula}.
Let us define
\begin{align*}
    \omega_2(y):=-\frac{1}{6}\eta_1(\abs{y})\big[X_{i,stu} (y+e_4)^s (y+e_4)^t (y+e_4)^u\big]\,dy^i,
\end{align*}
where $\eta_1$ is given by \eqref{eq:cutoffs-cylinder} and $X$ is given by \eqref{eq:C-gauge-formula}. Then, letting $X_2$ be the vector field dual to $\omega_2$, we got, for $y\sim -e_4$,
\begin{align*}
    (\mathcal{L}_{X_2} g_E)_{ij}(y)=-\frac{1}{2}\big(X_{i,jkl}+X_{j,ikl}\big) (y+e_4)^k (y+e_4)^l, \qquad \p^2_{kl}(\mathcal{L}_{X_2} g_E)_{ij}(y)=-X_{i,jkl}-X_{j,ikl}.
\end{align*}
Therefore, if we define
\begin{align*}
    \zeta:=\xi + \mathcal{L}_{X_2} g_E=h+\mathcal{L}_{X_0} g_E+\mathcal{L}_{X_1} g_E+\mathcal{L}_{X_2} g_E, \qquad \bar{\zeta}:=e^{2\phi}\zeta,
\end{align*}
    and
    \begin{align*}
        \bar{F}:=\sum_{n\in\Z}\psi_n^* \bar{\zeta},
    \end{align*}
    we see that, for $y\sim -e_4$, 
    \begin{align*}
        \bar{F}_{ij}(y)=\bar{T}_{ij}(y)-\frac{1}{2}\big(X_{i,jkl}+X_{j,ikl}\big) (y+e_4)^k (y+e_4)^l.
    \end{align*}
    Now, as for $\bar{T}$, the tensor $\bar{F}$ is equivariant under the action of $\psi_n$ $\forall n \in\Z$, so the lift $\tilde{F}:=\Phi^*\bar{F}$ defines a tensor on $\R\times \Sp^3$ which is equivariant under translations of length $nR$, $\forall n\in\Z$. Therefore $\tilde{F}$ descends to a singular tensor $F:=\pi_*(\tilde{F})$ on $Z$, which is smooth outside $q=(0,S)$. Moreover, by \eqref{eq:Tbar-lin-bach-eq} and the fact that $e^{2\phi} \mathcal{L}_X g_E\in \ker(\dot{B}_{\bar{g}})$, we also have $\dot{B}_g(F)=0$ in $Z\backslash\{q\}$. If we take Euclidean coordinates $x^i:=y^i+e_4$ centered at $q$, we see that $F$ decomposes as in \eqref{eq:F-tensor-cylinder}, where now the tensor $A$ satisfies
    \begin{align}\label{eq:AAbar-formula}
        {A}_{ij}(x)=\bar{A}_{ij}(x-e_4)-\frac{1}{2}\big( X_{i,jkl}+X_{j,ikl}\big)x^k x^l.
    \end{align}
    Hence \eqref{eq:AgradA-vanish} holds and the Hessian in \eqref{eq:T-hessian-def} has the symmetries of \eqref{eq:T-component-def}.
    Finally, by the symmetries of $X$ and the (anti-)symmetry of Weyl's tensor, we got
    \begin{align}\label{eq:weyl-X-zero-prod}
      W^{kijl}\big(X_{i,jkl}+X_{j,ikl}\big)=0,
    \end{align}
    so that 
    \eqref{eq:A-interaction-form} follows from \eqref{eq:hessian-Abar-formula}, \eqref{eq:AAbar-formula}.
\end{proof}

\begin{remark}\label{rem:A-non-TT-property}
    We point out that, while $\bar{A}$ is TT (cf. Lemma \ref{lem:Abar-expansion}), the tensor $A$ in \eqref{eq:F-tensor-cylinder} is not TT, and in general is not even traceless. To see this we notice that, by \eqref{eq:AAbar-formula},
    \begin{align*}
        \p^2_{kl}A_{ij}(0)=\p^2_{kl}\bar{A}_{ij}(-e_4)- X_{i,jkl}-X_{j,ikl},
    \end{align*}
    so, recalling \eqref{eq:T-component-def}, \eqref{eq:Tens-def}, the fact that $\bar{A}$ is TT and \eqref{eq:hessian-Abar-formula}, one has 
    \begin{align*}
        \p^2_{kl} \tr(A)(0)=\frac{1}{3}(\Delta \bar{A})_{kl}(-e_4)=\frac{8}{9}C_2(t)W_{4ij4},
    \end{align*}
    which does not vanish in general for all $i,j$.
\end{remark}

\subsection{Construction of the modified metric on $(\R \times \Sp^3)/G$}\label{subsec:green-quotients}

\subsubsection{The geometric setting}\label{subsec:quotients-geom-setting}

As pointed out in Theorem \ref{thm:CTZ-LCF-classification}, all manifolds of diffeomorphism type $Z\simeq(\R \times \Sp^3)/G$ (where, we recall,  $G$ is a cocompact fixed point free discrete subgroup of the isometry group of the standard metric on $\R\times \Sp^3$) are finitely covered by $\Sp^1 \times \Sp^3$. 
Therefore, given any $R>0$, we can indeed represent $Z$ as a quotient $(\R\times \Sp^3)/G$ and, without loss of generality, we might assume that there exists a normal subgroup $H\trianglelefteq G$, of finite index and isomorphic to $\Z$, which acts on $\R \times \Sp^3$ via translations of length $mR$, for $m\in\Z$.
\begin{remark}\label{rem:H-subgroup-translations}
    Without loss of generality, we can assume that $H$ is \emph{exactly} the subgroup consisting of elements $g\in G$ which act on $\R\times \Sp^3$ as a translation on the first factor (as that would indeed be a normal subgroup of $G$ of finite index containing $H$ and isomorphic to $\Z$, so we could replace $H$ with it).
\end{remark}
The quotient $(\R \times \Sp^3) /H$ is isometric to $([-R/2,R/2] \times \Sp^3)/\sim$, where the equivalence relation is the usual identification of the boundaries. Moreover, the (finite) quotient group $N:=G/H$ acts on $([-R/2,R/2] \times \Sp^3)/\sim$, and its quotient is isometric to $(\R\times \Sp^3)/G$.
Let $\pi\colon \R\times \Sp^3\to Z$ denote the projection and consider $\pi((0,S))=:q\in Z$ (where $S=(0,0,0,-1)\in\Sp^3$).
If we modify the standard metric on $Z$ in a conformal metric $g$ that is locally Euclidean around $q$, we can pull back  $g$ through $\pi$ to get a metric which is locally Euclidean around $\pi^{-1}(q)$. Let us call $\tilde{g}:=\pi^*(g)$.
We have the following characterization of the action of $\gamma\in N$ in $([-R/2,R/2] \times \Sp^3)/\sim$:

\begin{lemma}\label{lem:quotient-map-cylinder}
    Any $\gamma\in N$ is a composition of translations and rotations by a matrix $O\in \mathrm{SO}(4)$.
\end{lemma}
\begin{proof}
    Since we are assuming that the quotient $Z$ is oriented, it is known that the action of each element $\gamma\in N$ must be orientation-preserving. This automatically excludes maps which act on $([-R/2,R/2] \times \Sp^3)/\sim $ as an inversion (plus translation) on the first factor and as a rotation in $\mathrm{SO}(4)$ on the second factor, as well as maps acting as a translation on the first factor and as a rotation  in $\mathrm{O}(4)\backslash \mathrm{SO}(4)$ in the second one (as they are both orientation-reversing).
    Finally, let us consider maps in which we have an inversion on the first factor and a rotation $O\in \mathrm{O}(4)\backslash \mathrm{SO}(4)$ in the second one, namely $\gamma\in N$ satisfies
    \begin{align*}
        \gamma(r,w)=(L-r,O(w)), 
    \end{align*}
    for $\abs{L}\leq R/2$. Since $\mathrm{det}(O)=-1$, one has
    \begin{align*}
        \mathrm{det}(I-O)=\mathrm{det}(I-O^{-1})=\det(O^{-1})\mathrm{det}(O-I)=-\det(I-O),
    \end{align*}
    which implies that $\exists v\in \Sp^3$ satisfying $O(v)=v$. Then $(L/2,v)$ is a fixed point for $\gamma$, contradicting the fact that the action is free.
\end{proof}

By the above construction, if we regard $([-R/2,R/2] \times \Sp^3)$ as a subset of $\R\times \Sp^3$, then $\pi^{-1}(q)$ will be given by the  $N$-orbit of $(0,S)$ together with all of its translations by $mR, m\in\Z$, and the metric $\tilde{g}$ will be Euclidean in a neighborhood of all such points (and also equivariant under the action of $G$).

As before, we can consider the conformal map $\Phi$ into $\R^4\backslash\{0\}$ given by \eqref{eq:Phi-def-conf}, and we let
$\bar{g}:=\Phi_*(\tilde{g})=:e^{2\phi}g_E$ be a metric on $\R^4\backslash\{0\}$ conformal to $g_E$. This time, the dilations $\psi_n(y)=t^n y$, $t:=e^R$, correspond to the action of $H\simeq \Z$, while, by the above claim, the other elements $\gamma\in G$ correspond, in $\R^4\backslash\{0\}$, to dilations composed with (nontrivial) rotations in $\mathrm{SO}(4)$.

\subsubsection{The singular correction}

For any given $R>0$, let $t:=e^R>1$.
To begin, by the previous claim it is possible to write the $G$-orbit of a point $(r,w)\in\R\times \Sp^3$ as
\begin{align*}
    G((r,w)):=\bigcup_{n\in\Z}\bigcup_{a=1}^{N_0}\gamma_{n,a}((r,w))=\bigcup_{n\in\Z}\bigcup_{a=1}^{N_0} (r+r_a+nR,O_a(w)),
\end{align*}
where $N_0:=\#(N)$ is the cardinality of the quotient group $N$, $-R/2<r_1\leq \dots\leq r_{N_0}\leq R/2$ and $O_a\in \mathrm{SO}(4)$ $\forall a=1,\dots, N_0$.
Passing to $\R^4\backslash\{0\}$ via $\Phi$, we see that the $G$-orbit of a point $0\not=y\in\R^4$ is
\begin{align}\label{eq:G-orbit-euclidean-sp}
   G(y):= \bigcup_{n\in \Z}\bigcup_{a=1}^{N_0}\psi_{n,a}(y)= \bigcup_{n\in \Z}\bigcup_{a=1}^{N_0}  t^n s_a O_a(y),
\end{align}
where $\psi_{n,a}=\psi_{\gamma_{n,a}} \in\mathcal{M}_4$ is the M\"obius map $\psi_{n,a}(y):=t^n s_a O_a(y)$, and $s_a:= e^{r_a}$ satisfies $t^{-1/2}<s_1\leq \dots\leq s_{N_0}\leq t^{1/2}$ .

On $\R^4$, let $h$ be given by \eqref{eq:h-eucl-def} and define $\bar{h}:=e^{2\phi}h$. Then $\dot{B}_{\bar{g}}(\bar{h})=0$ in $\R^4\backslash\{0,-e_4\}$ by the conformal covariance \eqref{eq:lin-bach-conf-cov}. Given any $\gamma_{n,a}\in G$, let $\psi_{n,a}\in \mathcal{M}_4$ be the corresponding M\"obius map, which, by construction, is an isometry for the metric $\bar{g}$. Therefore, writing $O\in \mathrm{SO}(4)$ as $O=O^k_l \p_k\otimes dx^l$ and arguing as in \eqref{eq:hbar-pullback-eq}, \eqref{eq:conf-fact-pullback-eq}, we see that
\begin{equation}\label{eq:mobius-pullback-quotient}
    (\psi_{n,a}^* \bar{h})_{ij}(y)=(O_a)_i^\alpha (O_a)_j^\beta h_{\alpha \beta}(\psi_{n,a}(y)), \qquad \text{for $y$ close to $-e_4$, $\forall \gamma_{n,a}\in G$.}
\end{equation}

Now, as in the case of $\Sp^1 \times \Sp^3$, we would like to suitably gauge $\bar{h}$ in order to make the series of its pullbacks convergent and \acc well-behaved'' around $-e_4$. 
Let $0\leq\tilde{\eta}_0,\tilde{\eta}_1\leq 1$ be smooth cutoff functions on $\R\times \Sp^3$ such that 
\begin{align*}
    \tilde{\eta}_0(r,w)=\begin{cases*}
        1 & \text{for $r\leq r_1-\frac{3}{4}(r_1+R/2)$} \\
        0 & \text{for $r \geq r_1-\frac{1}{4}(r_1+R/2)$,} 
        \end{cases*}
        \qquad \tilde{\eta}_1(r,w)=\begin{cases*}
            1 & \text{for $\mathrm{dist}_{g_{\R\times \Sp^3}}((r,w),(0,S))\leq d_0/4$} \\
            0 & \text{for $\mathrm{dist}_{g_{\R\times \Sp^3}}((r,w),(0,S))\geq d_0 /2$,}
    \end{cases*}
\end{align*}
where 
\begin{align*}
d_0:= \min\{\mathrm{dist}_{g_{\R\times \Sp^3}}((r_a,O_a(S)),(0,S))\mid a=1,\dots,N_0, (r_a,O_a(S))\not=(0,S)\}
\end{align*}
denotes the minimal distance (in the standard metric) between $(0,S)$ and any other point in its $G$-orbit.
We can then define, on $\R^4\backslash\{0\}$,
\begin{align}\label{eq:cutoff-quotients-def}
    \eta_i(y):=\Phi(\tilde{\eta}_i(\Phi^{-1}(y))), \qquad i=0,1,
\end{align}
which are smooth cutoffs in the Euclidean space. In particular, $\eta_0$ will be identically equal to $1$ around all points $x\in G(-e_4)$ with $\abs{x}\leq t^{-1/2}$ and identically zero around the others, while $\eta_1$ will be identically $1$ around $-e_4$ and will vanish around all the other points in $G(-e_4)$.

Let us define the (preliminary) tensor
\begin{align}\label{eq:zeta_0-quotients-def}
    \xi_0:=h+ \mathcal{L}_{X_0}g_E,
\end{align}
where $X_0$ is the vectorfield dual to $\omega_0$ defined in \eqref{eq:omega_0-def} (now the cutoff $\eta_0$ in \eqref{eq:omega_0-def} is the one defined in \eqref{eq:cutoff-quotients-def}).
Then $\xi_0$ is TT both near the origin and near $-e_4$, and $\dot{B}_{g_E}(\xi_0)(y)=0$ if $y\not={-e_4}$. Let $\bar{\xi}_0:=e^{2\phi}\xi_0$; by conformal covariance, $\dot{B}_{\bar{g}}(\bar{\xi}_0)(y)=0$ if $y\not= 0,-e_4$. Define 
\begin{align}\label{eq:Tzerobar-quotients-def}
\bar{T}_0:=\sum_{a=1}^{N_0}\sum_{n\in\Z}\psi_{n,a}^*\bar{\xi}_0,
\end{align}
where $\psi_{n,a}$ is given by \eqref{eq:G-orbit-euclidean-sp}. By \eqref{eq:mobius-pullback-quotient}, \eqref{eq:zeta_0-quotients-def} we see that, \underline{for $y\sim -e_4$}, there holds
\begin{align}
\notag
    (\bar{T}_0)_{ij}(y)&=\sum_{a=1}^{N_0}\sum_{n\in\Z}(O_a)_i^\alpha (O_a)_j^\beta(\xi_0)_{\alpha \beta}(t^n s_aO_a(y)) \\
    \notag
    &=-\frac{1}{3}\sum_{a=1}^{N_0}\sum_{n\geq0}(O_a)_i^\alpha (O_a)_j^\beta W_{\mu \alpha \beta \nu}\frac{(t^n s_a O_a(y)+e_4)^\mu (t^n s_a O_a(y)+e_4)^\nu}{\abs{t^n s_a O_a(y)+e_4}^4} \\
    \label{eq:T_0-bar-formula-quotient}
    &\quad -\frac{1}{3}\sum_{a=1}^{N_0}\sum_{n>0}(O_a)_i^\alpha (O_a)_j^\beta \bigg[ W_{\mu \alpha \beta  \nu}\frac{(t^{-n} s_a O_a(y)+e_4)^\mu (t^{-n} s_a O_a(y)+e_4)^\nu}{\abs{t^{-n} s_a O_a(y)+e_4}^4}-W_{4\alpha \beta 4}\bigg].
\end{align}
Hence, being $t>1$, we see that the above series is convergent and well-defined at any point $y\not=-e_4$ close to $-e_4$ (indeed, the terms in the first sum behave as $t^{-2n}$, while the ones in the second sum behave as $t^{-n}$). 
The expression will be different for a generic point $y\in\R^4\backslash\{0\}$ due to the non-constant conformal factor $e^{2\phi}$ and the cutoff functions; however, it is not difficult to see that $\bar{T}_0$ is pointwisely convergent and well-defined on $\R^4\backslash\big(\{0\}\cup G(-e_4)\big)$. We also notice that, for $y\sim -e_4$, $\bar{T}_0(y)$ is also a TT tensor, as $\bar{\xi}_0$ is and the pullback by rotations and dilations preserve the TT condition with respect to the Euclidean metric $g_E$($=\bar{g}$ near $-e_4$).
Now, let us write 
\begin{align*}
    (\bar{T}_0)_{ij}(y)=-\frac{1}{3}W_{\mu i j \nu}\frac{(y+e_4)^\mu (y+e_4)^\nu}{\abs{y+e_4}^4}+(\bar{A}_0)_{ij}(y).
\end{align*}
By definition, $(\bar{A}_0)_{ij}$ is smooth and TT in a small neighborhood of $-e_4$.
However, in general we have $(\bar{A}_0)_{ij}(-e_4)\not=0$ and $\nabla (\bar{A}_0)_{ij}(-e_4)\not=0$, so, in order to perform the match between the metrics on $M$ and $Z$, we need to further gauge $\bar{\xi}_0$.

Let us define 
\begin{align*}
    S_{ij}:=(\bar{A}_0)_{ij}(-e_4), \qquad \qquad S_{ij,k}:=\p_k (\bar{A}_0)_{ij}(-e_4),
\end{align*}
\begin{align}\label{eq:omega_1-quotient-def}
    \omega_1(y):=-\eta_1(y)\Big\{\frac{1}{2}S_{is}(y+e_4)^s+\frac{1}{4}\big[S_{is,t}+S_{it,s}-S_{st,i}\big] (y+e_4)^s (y+e_4)^t \Big\}\,dy^i,
\end{align}
where $\eta_1$ is given by \eqref{eq:cutoff-quotients-def}.
Then, if we let $X_1$ be the vector field dual to $\omega_1$, it is straightforward to check that
\begin{align}\label{eq:gauge-term-identities-quotient}
    (\mathcal{L}_{X_1}g_E)_{ij}(-e_4)=-(\bar{A}_0)_{ij}(-e_4), \qquad \qquad \nabla (\mathcal{L}_{X_1}g_E)_{ij}(-e_4)=-\nabla (\bar{A}_0)_{ij}(-e_4).
\end{align}
Define
\begin{align}\label{eq:xi-quotient-def}
\xi:=h+\mathcal{L}_{X_0}g_E+ \mathcal{L}_{X_1}g_E,
\end{align}
and
\begin{align}\label{eq:Tbar-quotients-def}
\bar{T}:=\sum_{a=1}^{N_0}\sum_{n\in\Z}\psi_{n,a}^*\bar{\xi}, \qquad \bar{\xi}:=e^{2\phi} \xi.
\end{align}
As for $\bar{T}_0$, $\bar{T}$ is well-defined on $\R^4\backslash\big(\{0\}\cup G(-e_4)\big)$, and it satisfies by construction
\begin{align}\label{eq:Tbar-lin-bach-eq-quotient}
    \dot{B}_{\bar{g}}(\bar{T})=0    \qquad \text{in} \,\,\,\R^4\backslash \big(\{0\}\cup G(-e_4)\big). 
\end{align}

\begin{lemma}\label{lem:Abar-exp-quotients}
    Let us write, for $y\sim -e_4$,
    \begin{equation*}
        \bar{T}_{ij}=-\frac{1}{3}W_{\mu i j \nu}\frac{(y+e_4)^\mu (y+e_4)^\nu}{\abs{y+e_4}^4}+\bar{A}_{ij}(y),
    \end{equation*}
    where $\bar{T}$ is given by \eqref{eq:Tbar-quotients-def}.
    Then $\bar{A}_{ij}(y)$ is TT for $y$ close to $-e_4$, and one further has
    \begin{align}\label{eq:vanishing-Abar-quotient}
        \bar{A}(-e_4)=0, \qquad \nabla \bar{A}(-e_4)=0,
    \end{align}
    \begin{align}
\notag
    \p^2_{kl}\bar{A}_{ij}(-e_4)=-\frac{1}{3}{C_2(t)}\bigg[&\big(W_{kijl} + W_{lijk}\big)-4\delta_{4k}\big(W_{lij4}+W_{4ijl}\big)-4\delta_{4l}\big(W_{kij4}+W_{4ijk}\big) \\
    \label{eq:hessian-Abar-quotients-formula}
    &   -4\delta_{kl} W_{4ij4} +24 \delta_{4k}\delta_{4l} W_{4ij4}\bigg] + O\big((t-1)^{-1}\big), \qquad \text{as $t\to 1^+$.}
\end{align}
Here $C_2(t)$ is given by \eqref{eq:C_2(t)}; in particular,
\begin{align}\label{eq:C_2(t)-asymptotics}
    C_2(t)=\frac{\pi^4}{45(t-1)^4}+ O\big( (t-1)^{-3}\big), \qquad \text{as $t\to 1^+$.}
\end{align}
\end{lemma}

\begin{remark}
    Differently from the case of $\Sp^1 \times \Sp^3$ (see Lemma \ref{lem:Abar-expansion}), here we do not have an exact (useful) identity for the hessian of $\bar{A}$. However, by letting $t\to 1^+$ (that is, by shrinking the length $R$ of the cylinder to zero), we see that the leading term in the hessian is again given by the quantity appearing in \eqref{eq:hessian-Abar-formula}. This is motivated by the fact that, differently from the poles coming from pure dilations, the poles of $\bar{T}$ coming from dilations and rotations do not accumulate toward $-e_4$ as $t\to 1^+$.
\end{remark}

\begin{proof}
    The identities \eqref{eq:vanishing-Abar-quotient} and the fact that $\bar{A}$ is TT directly follow by construction, see \eqref{eq:Tzerobar-quotients-def}, \eqref{eq:T_0-bar-formula-quotient}, \eqref{eq:gauge-term-identities-quotient}, \eqref{eq:Tbar-quotients-def}.

    Let us turn our attention to \eqref{eq:hessian-Abar-quotients-formula}. Since the gauge terms in \eqref{eq:xi-quotient-def} are given (around $-e_4$) by polynomials of order $\leq1$, we see that the hessian of $\bar{A}$ is equal to that of $\bar{A}_0$.
    Look at \eqref{eq:T_0-bar-formula-quotient}, and, for $y\sim -e_4$, consider a single term of the series
    \begin{align*}
        V_{ij}(y):=-\frac{1}{3}W_{\mu i j \nu}\frac{(t^n s O(y)+e_4)^\mu (t^n s O(y)+e_4)^\nu}{\abs{t^n s O(y)+e_4}^4},
    \end{align*}
    where we dropped the subscript $a$ for convenience. We have
    \begin{align}
    \notag
        \p^2_{kl}V_{ij}(y)=&-\frac{1}{3} t^{2n}s^2W_{\mu i j \nu}\Bigg[\frac{O^\mu_k O^\nu_l+O^\mu_l O^\nu_k}{\abs{t^nsO(y)+e_4}^4}-4(t^nsy_l+O^4_l)\frac{O^\mu_k(t^nsO^\nu_\beta y^\beta+\delta_4^\nu)+(t^nsO^\mu_\alpha y^\alpha+\delta_4^\mu)O^\nu_k}{\abs{t^nsO(y)+e_4}^6} \\
        \notag
        -4\delta_{kl}&\frac{(t^ns O^\mu_\alpha y^\alpha+\delta^\mu_4)(t^nsO^\nu_\beta y^\beta+\delta_4^\nu)}{\abs{t^nsO(y)+e_4}^6}-4(t^nsy_k+O^4_k)\frac{O^\mu_l(t^nsO^\nu_\beta y^\beta+\delta_4^\nu)+(t^nsO^\mu_\alpha y^\alpha+\delta_4^\mu)O^\nu_l}{\abs{t^nsO(y)+e_4}^6} \\
        \label{eq:hess-V-formula}
        &+24(t^ns y_k+O^4_k)(t^ns y_l +O^4_l)\frac{(t^ns O^\mu_\alpha y^\alpha+\delta^\mu_4)(t^nsO^\nu_\beta y^\beta+\delta_4^\nu)}{\abs{t^nsO(y)+e_4}^8}\Bigg].
    \end{align}
    Let us now look at $\p^2_{kl}V_{ij}(-e_4)$.

\medskip

    \noindent $\bullet$ Assume $O=\mathrm{Id}$ and $s=1$. Then
    \begin{align*}
        \p^2_{kl}V_{ij}(-e_4)=-\frac{1}{3}\frac{t^{2n}}{(1-t^n)^4}\bigg[&\big(W_{kijl} + W_{lijk}\big)-4\delta_{4k}\big(W_{lij4}+W_{4ijl}\big)-4\delta_{4l}\big(W_{kij4}+W_{4ijk}\big) \\
    &   -4\delta_{kl} W_{4ij4} +24 \delta_{4k}\delta_{4l} W_{4ij4}\bigg],
    \end{align*}
    so, when summing over $n\in\Z, n\not=0$, we obtain exactly the first term of \eqref{eq:hessian-Abar-quotients-formula}, cf. \eqref{eq:C_2(t)}.

\medskip

    \noindent $\bullet$ Assume $O\not=\mathrm{Id}$ and $t^{-1/2}\leq s\leq t^{1/2}$. Notice that, by virtue of Remark \ref{rem:H-subgroup-translations}, we cannot have $s\not=1$ and $O=\mathrm{Id}$, so all the remaining cases are included in this second one. Being $O\in \mathrm{SO}(4)$, looking at \eqref{eq:hess-V-formula} we deduce
    \begin{align*}
        \abs{\p^2_{kl}V_{ij}(-e_4)}\leq C\frac{t^{2n}s^2}{\big|-t^nsO_4^\beta e_\beta+e_4\big|^4}= C\frac{t^{2n}s^2}{\big(t^{2n}s^2-2t^nsO_4^4+1\big)^2},
    \end{align*}
    where $C>0$ does not depend on $n\in\Z$, neither on $O\in \mathrm{SO}(4)$ or $s\in[t^{-1/2},t^{1/2}]$.
    Being $\mathrm{Id}\not=O\in \mathrm{SO}(4)$ with free action, then $\tau:=O^4_4\in[-1,1)$ necessarily.
    Consider the series over $n\in\Z$ of the above coefficients. Setting $\epsilon=\log t$ we have
    \begin{align*}
        \sum_{n\in\Z}\frac{t^{2n}s^2}{\big(t^{2n}s^2-2t^nsO_4^4+1\big)^2}&=\frac{1}{s^2}\sum_{n\in\Z}\frac{e^{2n\eps }}{\big( e^{2n\eps}-2\tau s^{-1}e^{n\eps}+s^{-2}\big)^2} \\
        &=\frac{1}{\eps s^2}\bigg(\sum_{n\in\Z}\eps \frac{e^{2n\eps }}{\big( e^{2n\eps}-2\tau s^{-1}e^{n\eps}+s^{-2}\big)^2}\bigg)=:\frac{1}{\eps s^2}S(\eps).
    \end{align*}
    If $t\to1^+$, then $\eps\to 0^+$ and the Riemann sum converges to 
    \begin{align*}
        S(\eps)\xrightarrow{\eps\to 0^+} \int_{-\infty}^{-\infty}\frac{e^{2x}}{\big(e^{2x}-2\tau s^{-1}e^x+s^{-2}\big)^2}\,dx=\int_{0}^{+\infty}\frac{r}{\big(r^2-2\tau  s^{-1}r+s^{-2}\big)^2}\,dr=:I(\tau,s),
    \end{align*}
    where $I(\tau,s)$ is finite for any $\tau\in[-1,1)$ and $s\in[1/2,2]$, with $I(\tau,s)\to+\infty$ as $\tau\to 1^-$.
    Hence (recall that $\tau=O^4_4$)
    \begin{align*}
        \sum_{n\in\Z}\abs{\p^2_{kl}V_{ij}(-e_4)}\leq \frac{C(O_4^4)}{(t-1)}, \qquad \text{as $t\to 1^+$.}
    \end{align*}

    The above considerations prove \eqref{eq:hessian-Abar-quotients-formula}.

    \medskip

    Finally, we prove \eqref{eq:C_2(t)-asymptotics}. Letting again $\eps:=\log t$, we rewrite
    \begin{align*}
        C_2(t)=2\sum_{n>0}\frac{t^{2n}}{(t^n-1)^4}=2\sum_{n>0}\frac{e^{2n\eps}}{(e^{n\eps}-1)^4}=:2\sum_{n>0}\bigg[\frac{1}{n^4 \eps^4}+f(n\eps)\bigg],
    \end{align*}
    where 
    \begin{align*}
        f(x):=\frac{e^{2x}}{(e^x-1)^4}-\frac{1}{x^4}.
    \end{align*}
    One has
    \begin{align*}
       \sum_{n=1}^{+\infty}\frac{1}{n^4\eps^4}=\frac{\pi^4}{90(t-1)^4}+O\big( (t-1)^{-3}\big), \qquad \text{as $t\to 1^+$.}
    \end{align*}
    As for the other term, we notice that $f$ is continuous in $(0,+\infty)$, $f(x)=O(x^{-3})$ as $x\to 0^+$ and $f(x) = O(x^{-4})$ as $x\to +\infty$. Therefore there exists a constant $C>0$ such that $\abs{f(x)}\leq C x^{-3} \,\, \forall x>0$, which implies
    \begin{align*}
        \sum_{n=1}^{+\infty}f(n\eps)\leq \frac{C}{\eps^3}\sum_{n=1}^{+\infty}\frac{1}{n^3}\leq \frac{C'}{(t-1)^3}, \qquad \text{as $t\to 1^+$,}
    \end{align*}
    so that \eqref{eq:C_2(t)-asymptotics} holds.
\end{proof}

Finally, we have the extension of Proposition \ref{prop:F-expansion-cylinder} to this setting:

\begin{proposition}\label{prop:F-exp-quotients}
    Let $(Z,g)$ and $q\in Z$ be given as in Section \ref{subsec:quotients-geom-setting}. There exists a symmetric tensor $F$, which is smooth in $Z\backslash\{q\}$, such that $\dot{B}_g(F)=0$ in $Z\backslash\{q\}$ and, in suitable Euclidean coordinates centered at $q$, \eqref{eq:F-tensor-cylinder} holds,
    where $A$ is smooth up to zero and satisfies
    \begin{align}\label{eq:AgradA-vanish-quotients}
        A(0)=0, \qquad \qquad \nabla A(0)=0,
    \end{align}
    \begin{align}
    \notag
        W^{kijl}\p^2_{kl} A_{ij}(0)=&-\frac{1}{3}C_2(t)\bigg[W^{kijl}\big(W_{kijl}+W_{lijk}\big)-4W^{4ijl}\big(W_{4ijl}+W_{lij4}\big) \\
    \label{eq:A-interaction-form-quotients}
    &\quad   -4 W^{kij4}\big(W_{kij4}+W_{4ijk}\big)+24W^{4ij4}W_{4ij4}\bigg] + O\big( (t-1)^{-1}\big), \quad \text{as $t\to 1^+$,}
    \end{align}
    where $C_2(t)$ is given by \eqref{eq:C_2(t)} and satisfies \eqref{eq:C_2(t)-asymptotics}.
    Moreover, if we let
    \begin{align}\label{eq:T-hessian-def-quotients}
        T_{kijl}:=\frac{1}{2}\p^2_{kl}A_{ij}(0),
    \end{align}
    then $T$ can be written as in \eqref{eq:T-component-def}.
\end{proposition}

The proof is identical to that of Proposition \ref{prop:F-expansion-cylinder}. As in Remark \ref{rem:A-non-TT-property}, we notice that $A$ will not be TT in general around $0$.

\section{Weyl energy of the modified metric on \texorpdfstring{$Z=(\R\times \Sp^3)/G$}{Z}}\label{sec:Weyl-en-Z}

We know that any locally conformally flat manifold $(Z,g_Z)$ is diffeomorphic to one of the models described in Theorem \ref{thm:CTZ-LCF-classification}. Since we are only interested in the diffeomorphism class of the connected sum $Z\# M$, we might assume without loss of generality that $(Z,g_Z)$ is indeed one of the model manifolds described in Section \ref{sec:green-LCF}. Then, by Propositions \ref{prop:F-expansion-cylinder}, \ref{prop:F-exp-quotients}, we know that, in the particular case in which $Z=(\R\times \Sp^3)/G$, there exists a point $q\in Z$ with $g_Z$ Euclidean around $q$, and there exists a tensor $F$, singular at $q$, which satisfies the properties described in the aforementioned propositions. 

Let us consider the modified metric $g_b:=g_Z+b^2 F$, where $b\ll 1$ is a small parameter. By virtue of \eqref{eq:F-tensor-cylinder}, we have the following expansion of $g_b$ in Euclidean coordinates at $q$:
\begin{equation}\label{eq:g_b-definition}
    (g_{b})_{ij}(x)=\delta_{ij}+b^2 \bigg(-\frac{1}{3} W^M_{kijl}\frac{x^kx^l}{\abs{x}^4}+ A_{ij}(x)\bigg) \quad \text{in $B_{\delta}(0),$}
\end{equation}
for $\delta>0$ small enough.
Consider a small parameter $\gamma$ such that $0<a,b\ll\gamma\ll\ \delta/2$; we now want to estimate 
\begin{equation*}
   \w^Z_\gamma(g_b):= \int_{Z\backslash B_{\gamma}^{g_Z}(q)}\abs{W^{g_b}}^2\,dV_{g_b}.
\end{equation*}
To begin, Taylor-expanding $g_b$ with respect to $t:=b^2$, we get 
\begin{equation*}
     \w^Z_\gamma(g_b)=\w^Z_\gamma(g_Z)+b^2\frac{d}{dt}\Big\vert_{t=0}\w^Z_\gamma(g_b)+\frac{b^4}{2}\frac{d^2}{dt^2}\Big\vert_{t=0}\w^Z_\gamma(g_b)+O(b^6).
\end{equation*}
Notice that the error term $O(b^6)$ also depends upon $\gamma$; however, since $b\ll\gamma$, the error term remains uniformly controlled by, e.g., $b^5$, and hence still of higher order.
Being $g_Z$ locally conformally flat, we have $\w^Z_\gamma(g_Z)=0$ and
\begin{equation}\label{eq:wetl-first-var}
    \frac{d}{dt}\Big\vert_{t=0}\w^Z_\gamma(g_b)=-4\int_{Z\backslash B_{\gamma}^{g_Z}(q)}\big(W^{g_Z}_{\alpha i j \beta}\nabla^\alpha \nabla^\beta F^{ij}+\frac{1}{2}W^{g_Z}_{\alpha i j \beta}R^{\alpha \beta}F^{ij}\big)\,dV_{g_Z}=0,
\end{equation}
so we are left with
\begin{equation}\label{eq:weyl-g_b-expans}
    \w^Z_\gamma(g_b)=\frac{b^4}{2}\frac{d^2}{dt^2}\Big\vert_{t=0}\w^Z_\gamma(g_b)+O(b^6).
\end{equation}
Using again the fact that $g_Z$ is locally conformally flat, we see from \eqref{eq:wetl-first-var} that 
\begin{equation}\label{eq:second-der-weyl}
    \frac{d^2}{dt^2}\Big\vert_{t=0}\w^Z_\gamma(g_b)=-4\int_{Z\backslash B_{\gamma}^{g_Z}(q)}\big(\dot{W}_{\alpha i j \beta}\nabla^\alpha \nabla^\beta F^{ij}+\frac{1}{2}\dot{W}_{\alpha i j \beta}R^{\alpha \beta}F^{ij}\big)\,dV_{g_Z}.
\end{equation}
Integrating by parts twice, one has
\begin{align}
\notag
    \int_{Z\backslash B_{\gamma}}\big(\dot{W}_{\alpha i j \beta}&\nabla^\alpha \nabla^\beta F^{ij}+\frac{1}{2}\dot{W}_{\alpha i j \beta}R^{\alpha \beta}F^{ij}\big)\,dV_{g_Z} \\
    \notag
    &=-\int_{Z\backslash B_{\gamma}}\big(\nabla^\alpha\dot{W}_{\alpha i j \beta}\nabla^\beta F^{ij}+\frac{1}{2}\dot{W}_{\alpha i j \beta}R^{\alpha \beta}F^{ij}\big)\,dV_{g_Z}-\int_{\partial B_\gamma}\dot{W}_{n i j \beta}\nabla^\beta F^{ij}\,d\sigma_{g_Z} \\
    \notag
    &=\int_{Z\backslash B_{\gamma}}\big(\nabla^\beta\nabla^\alpha\dot{W}_{\alpha i j \beta}F^{ij}+\frac{1}{2}\dot{W}_{\alpha i j \beta}R^{\alpha \beta}F^{ij}\big)\,dV_{g_Z}-\int_{\partial B_\gamma}\dot{W}_{n i j \beta}\nabla^\beta F^{ij}\,d\sigma_{g_Z} \\
    \label{eq:weyl-second-der-parts}
    &\quad \quad \quad \quad +\int_{\partial B_\gamma}\nabla^\alpha\dot{W}_{\alpha i j n}F^{ij}\,d\sigma_{g_Z},
\end{align}
where the $n$ subscript denotes the tensor component with respect to the \emph{inward pointing} unit normal (so it points towards the \emph{exterior} of $B_\gamma$). 

We now observe the following:
\begin{claim*}
    Let $(g_t)_t$ be a smooth family of metrics. If $g=g_0$ is locally conformally flat, then 
    \begin{equation*}
        \frac{d}{dt}\Big\vert_{t=0}\big(\nabla^\alpha\nabla^\beta W^{g_t}_{\alpha i j \beta}\big)=\nabla^\alpha\nabla^\beta\dot{W}^g_{\alpha i j \beta}.
    \end{equation*}
\end{claim*}
\begin{proof}
Let $h=\frac{d}{dt}\mid_{t=0} g_t$. Being the metric tensor and its inverse parallel at each time and using the fact that $W^{g}\equiv0$, we have
\begin{align*}
    \frac{d}{dt}\Big\vert_{t=0}\big(\nabla^\alpha\nabla^\beta W_{\alpha i j \beta}\big)&=\frac{d}{dt}\Big\vert_{t=0}\big(g^{\alpha \mu}g^{\beta \nu}\nabla_\mu \nabla_\nu W_{\alpha i j \beta}\big) \\
    =&g^{\alpha \mu}g^{\beta \nu}\frac{d}{dt}\Big\vert_{t=0}\big(\nabla_\mu\nabla_\nu W_{\alpha i j \beta}\big)-\big(h^{\alpha\mu}g^{\beta\nu}+g^{\alpha \mu}h^{\beta\nu}\big)\nabla_\mu\nabla_\nu W_{\alpha i j \beta} \\
    =&g^{\alpha \mu}g^{\beta \nu}\frac{d}{dt}\Big\vert_{t=0}\Big[\partial_\mu\big(\nabla_\nu W_{\alpha i j \beta}\big)-\Gamma_{\mu \nu}^s\nabla_s W_{\alpha i j \beta}-\Gamma_{\mu \alpha}^s\nabla_\nu W_{s i j \beta} \\
    &\quad \quad-\Gamma_{\mu i}^s\nabla_\nu W_{\alpha s j \beta}-\Gamma_{\mu j}^s\nabla_\nu W_{\alpha i s \beta}-\Gamma_{\mu \beta}^s\nabla_\nu W_{\alpha i j s}\Big] \\ 
    =&g^{\alpha\mu}g^{\beta\nu}\Big[\nabla_\mu \big(\frac{d}{dt}\Big\vert_{t=0}(\nabla_\nu W_{\alpha i j \beta})\big)-\dot{\Gamma}_{\mu \nu}^s\nabla_s W_{\alpha i j \beta}-\dot{\Gamma}_{\mu \alpha}^s\nabla_\nu W_{s i j \beta} \\
    &\quad \quad-\dot{\Gamma}_{\mu i}^s\nabla_\nu W_{\alpha s j \beta}-\dot{\Gamma}_{\mu j}^s\nabla_\nu W_{\alpha i s \beta}-\dot{\Gamma}_{\mu \beta}^s\nabla_\nu W_{\alpha i j s}\Big] \\
    =&g^{\alpha \mu}g^{\beta \nu}\nabla_\mu \Big[\frac{d}{dt}\Big\vert_{t=0}\big(\partial_\nu W_{\alpha i j \beta}-\Gamma_{\nu\alpha}^sW_{s i j \beta}-\Gamma_{\nu i}^sW_{\alpha s j \beta}-\Gamma_{\nu j}^sW_{\alpha i s \beta}-\Gamma_{\nu\beta}^sW_{\alpha i j s}\big)\Big] \\
    =&g^{\alpha \mu}g^{\beta \nu}\nabla_\mu\Big[\nabla_\nu\dot{W}_{\alpha i j \beta}-\dot{\Gamma}_{\nu\alpha}^sW_{s i j \beta}-\dot{\Gamma}_{\nu i}^sW_{\alpha s j \beta}-\dot{\Gamma}_{\nu j}^sW_{\alpha i s \beta}-\dot{\Gamma}_{\nu\beta}^sW_{\alpha i j s}\Big] \\
    =&g^{\alpha \mu}g^{\beta \nu}\nabla_\mu \nabla_\nu \dot{W}_{\alpha i j \beta}=\nabla^\alpha\nabla^\beta\dot{W}_{\alpha i j \beta},
\end{align*}
thus the claim holds.
\end{proof}

As a consequence of the previous claim, we see that, at our locally conformally flat metric $g_Z$, one has 
\begin{equation*}
    \dot{B}_{ij}(F)=-4\big(\nabla^\alpha\nabla^\beta \dot{W}_{\alpha i j \beta}+\frac{1}{2}R^{\alpha\beta}\dot{W}_{\alpha i j \beta}\big).
\end{equation*}
Using this identity together with \eqref{eq:weyl-g_b-expans}, \eqref{eq:second-der-weyl}, \eqref{eq:weyl-second-der-parts} and the fact that $\dot{B}_{g_Z}(F)(p)=0$ for $p\not=q$ (see e.g. Proposition \ref{prop:F-expansion-cylinder}), we deduce that
\begin{equation*}
    \w^Z_\gamma(g_b)=2b^4\bigg(\int_{\partial B_\gamma}\dot{W}_{n i j \beta}\nabla^\beta F^{ij}\,d\sigma_{g_Z}-\int_{\partial B_\gamma}\nabla^\alpha \dot{W}_{\alpha i j n}F^{ij}\,d\sigma_{g_Z}\bigg) +O(b^6).
\end{equation*}
In particular, since we are assuming $g_Z$ to be Euclidean inside $B_\delta(q)$ and since $\gamma<\delta/2$, we finally get
\begin{equation}\label{eq:weyl-g_b-boundary-formula}
    \w^Z_\gamma(g_b)=2b^4\bigg(\int_{\partial B_\gamma}\dot{W}_{\eta i j \beta}\partial^\beta F^{ij}\nu^\eta\,d\sigma-\int_{\partial B_\gamma}\partial^\alpha \dot{W}_{\alpha i j \eta}F^{ij}\nu^\eta\,d\sigma\bigg) +O(b^6),
\end{equation}
where $\nu=(\nu^1,\dots,\nu^4)=x/\abs{x}$ is the unit normal and $d\sigma$ is the Euclidean surface element.

\medskip

We now want to expand both integrals in \eqref{eq:weyl-g_b-boundary-formula}, starting from the second one.
To begin, recall the formula for the linearized Weyl tensor at the Euclidean metric (see \cite{malchiodi-malizia-2025-pre-weyl}, equation (2.8) for the linearized of Weyl tensor in the same notation employed here; in alternative, see e.g. \cite{catino-mastrolia-book}, which employs a different convention):
\begin{align}
\notag
    \dot{W}_{\alpha i j \beta}=&\frac{1}{2}\Big[\partial^2_{\alpha j}F_{i\beta}+\partial^2_{i \beta} F_{\alpha j}-\partial^2_{\alpha \beta} F_{ij}-\partial^2 _{ij} F_{\alpha \beta}\Big]+\frac{1}{4}\Big[\delta_{\alpha\beta} \Delta F_{ij}+\delta_{ij}\Delta F_{\alpha\beta}-\delta_{i\beta}\Delta F_{\alpha j}-\delta_{\alpha j}\Delta F_{i\beta}\Big] \\
    \notag
    &+\frac{1}{4}\Big[\big(\partial_\alpha(\delta F)_j +\partial_j(\delta F)_\alpha -\p^2_{\alpha j}(\tr F)\big)\delta_{i\beta}+\big(\partial_i(\delta F)_\beta+\partial_\beta(\delta F)_i-\p^2_{i\beta}(\tr F)\big)\delta_{\alpha j} \\
    \notag
    &\quad \quad -\big(\partial_\alpha(\delta F)_\beta+\partial_\beta(\delta F)_\alpha-\p^2_{\alpha \beta} (\tr F)\big)\delta_{ij} -\big(\partial_i(\delta F)_j+\partial_j(\delta  F)_i-\p^2_{ij}(\tr F)\big) \delta_{\alpha\beta}\Big] \\
    \label{eq:weyl-lin-tfree-eucl}
    & \qquad \quad  +\frac{1}{6}\big(-\Delta (\tr F)+(\delta^2 F)\big) (\delta_{\alpha\beta} \delta_{ij}-\delta_{\alpha j}\delta_{\beta i}),
\end{align}
where $\delta F$ denotes the (Euclidean) divergence of $F$. 
We can also compute the Euclidean divergence of \eqref{eq:weyl-lin-tfree-eucl}, obtaining
\begin{align}
\notag
    \partial^\alpha\dot{W}_{\alpha i j \beta}=&\frac{1}{4}\Big[\partial_j \Delta F_{i\beta }-\partial_\beta \Delta F_{ij}+\partial^2_{i\beta}(\delta F )_j-\partial^2_{ij}(\delta F)_\beta\Big] \\
    \label{eq:weyl-lin-div-alpha}
    &+\frac{1}{12}\Big[\big(\partial_j(\delta^2 F)-\p_j \Delta(\tr F)\big)\delta_{i\beta}-\big(\partial_\beta(\delta^2 F)-\p_\beta \Delta(\tr F)\big)\delta_{ij}\Big].
\end{align}

Let us write (cf. \eqref{eq:F-tensor-cylinder})
\begin{equation}\label{eq:F-KA-splitting}
    F_{ij}(x)=K_{ij}(x)+A_{ij}(x), \qquad K_{ij}(x):=-\frac{1}{3}W_{\mu i j \nu}\frac{x^\mu x^\nu}{\abs{x}^4}.
\end{equation}
While it is immediate to check that $K(x)$ is a TT tensor with respect to the Euclidean metric for any $x\not=0$, the tensor $A$ in general has nontrivial trace and divergence near the origin, see e.g. Remark \ref{rem:A-non-TT-property}.
Moreover, from the definition of $K$ we easily compute
\begin{equation}\label{eq:lin-weyl-f1}
    \partial_\beta \Delta K_{ij}=\partial_\beta\Big(\frac{8}{3}W_{kijl}\frac{x^k x^l}{\abs{x}^6}\Big)=\frac{8}{3}W_{kijl}\Big(\frac{\delta_\beta^kx^l+ x^k\delta_\beta^l}{\vert x|^6}-6\frac{x_\beta x^k x^l}{\abs{x}^8}\Big),
\end{equation}
while, since the singular term of $F$ is TT, we further have
\begin{equation}\label{eq:lin-weyl-f2}
    \partial^2_{i\beta}(\delta F)_j=\partial^2_{i\beta}(\delta A)_j, \qquad \quad \partial_j(\delta^2 F)=\partial_j(\delta^2 A), \qquad \quad \p_j\Delta (\tr F)=\p_j\Delta (\tr A).
\end{equation}

\begin{lemma}\label{lem:Z-en-bal-1}
    In the above setting, it holds
    \begin{equation}\label{eq:divint-expansion}
        \int_{\partial B_\gamma}\partial^\alpha\dot{W}_{\alpha i j \beta}F^{ij}\nu^\beta\,d\sigma=\mathcal{C}_1\gamma^{-4}+\frac{\pi^2}{6}\partial^{kl} A^{ij}(0)W_{kijl}+O(\gamma^2), \qquad \text{as $\gamma\to0^+$,}
    \end{equation}
    for a real constant $\mathcal{C}_1=\mathcal{C}_1(W^M(p))$.
\end{lemma}
\begin{proof}
From \eqref{eq:weyl-lin-div-alpha} one has
\begin{align*}
    \int_{\partial B_\gamma}\partial^\alpha \dot{W}_{\alpha i j \beta} F^{ij}\nu^\beta\,d\sigma=&\frac{1}{4}\int_{\partial B_\gamma}\Big[\partial_j \Delta F_{i\beta }-\partial_\beta \Delta F_{ij}+\partial^2_{i\beta}(\delta F )_j-\partial^2_{ij}(\delta F)_\beta\Big]F^{ij}\nu^\beta\,d\sigma \\
    +\frac{1}{12}\int_{\partial B_\gamma}&\Big[\big(\partial_j(\delta^2 F)-\p_j \Delta(\tr F)\big)\delta_{i\beta}F^{ij}\nu^\beta-\big(\partial_\beta(\delta^2 F)-\p_\beta \Delta(\tr F)\big)(\tr F)\nu^\beta\Big]\,d\sigma. 
\end{align*}
Using now \eqref{eq:F-KA-splitting}, \eqref{eq:lin-weyl-f2}, the above expression becomes
\begin{align}
\label{eq:divint-mainord}
    \int_{\partial B_\gamma}\partial^\alpha \dot{W}_{\alpha i j \beta} F^{ij}&\nu^\beta\,d\sigma= \frac{1}{4}\int_{\partial B_\gamma}\Big[\p_j \Delta K_{i\beta}-\p_\beta \Delta K_{ij}\Big] K^{ij}\nu^\beta\,d\sigma \\
    \label{eq:divint-order1-term}
    &  +\frac{1}{4}\int_{\partial B_\gamma}\Big[\p_j \Delta K_{i\beta}-\p_\beta \Delta K_{ij}\Big]A^{ij}\nu^\beta\,d\sigma \\
    \label{eq:divint-hot1}
    &  +\frac{1}{4}\int_{\partial B_\gamma}\Big[ \partial_j \Delta A_{i \beta}-\p_\beta \Delta A_{ij}+\partial^2_{i\beta}(\delta A)_j-\partial^2_{ij}(\delta A)_{\beta}\Big]K^{ij}\nu^\beta\,d\sigma \\
    \label{eq:divint-hot2-1}
    & +\frac{1}{4}\int_{\partial B_\gamma}\Big[ \partial_j \Delta A_{i \beta}-\p_\beta \Delta A_{ij}+\partial^2_{i\beta}(\delta A)_j-\partial^2_{ij}(\delta A)_{\beta}\Big] A^{ij}\nu^\beta\,d\sigma  \\
    \label{eq:divint-hot3}
    &  +\frac{1}{12}\int_{\partial B_\gamma}\big(\partial^j(\delta^2A)-\p_j\Delta(\tr A)\big) \delta_{i\beta}K^{ij}\nu^\beta\,d\sigma \\
    \label{eq:divint-hot2-2}
    &+\frac{1}{12}\int_{\partial B_\gamma}\Big[\big(\partial_j(\delta^2 A)-\p_j \Delta(\tr A)\big)\delta_{i\beta}A^{ij}\nu^\beta-\big(\partial_\beta(\delta^2 A)-\p_\beta \Delta(\tr A)\big)(\tr A)\nu^\beta\Big]\,d\sigma.
\end{align}
We now consider separately at each  of the above integrals.

\textbullet \, By looking at  \eqref{eq:F-KA-splitting}, \eqref{eq:lin-weyl-f1}, it is easy to see (after a change of variables) that the RHS of \eqref{eq:divint-mainord} is of order $\gamma^{-4}$. Since we are going to show below that all the other integrals provide higher order contributions, we conclude that
\begin{align}\label{eq:divint-mainexp}
    \frac{1}{4}\int_{\partial B_\gamma}\Big[\p_j \Delta K_{i\beta}-\p_\beta \Delta K_{ij}\Big] K^{ij}\nu^\beta\,d\sigma=: \mathcal{C}_1 \gamma^{-4},
\end{align}
where $\mathcal{C}_1=\mathcal{C}_1(W^M(p))$ is the constant appearing in \eqref{eq:divint-expansion}.

\textbullet \, Let us consider \eqref{eq:divint-order1-term}. To begin, recalling \eqref{eq:lin-weyl-f1} and using the symmetries of Weyl's tensor and Bianchi's identity we got
\begin{align*}
    & W_{ki\beta l}\Big(\frac{\delta_j^k x^l + x^k \delta_j^l}{\abs{x}^6}-6\frac{x_j x^k x^l}{\abs{x}^8}\Big)-W_{kijl}\Big(\frac{\delta_\beta^k x^l + x^k \delta_\beta^l}{\abs{x}^6}-6\frac{x_\beta x^k x^l}{\abs{x}^8}\Big) \\
    =& (W_{ji\beta l} +W_{li\beta j}-W_{\beta i j l}-W_{lij\beta})\frac{x^l}{\abs{x}^6}-6W_{ki \beta l}\frac{x_j x^k x^l}{\abs{x}^8}+6  W_{kijl}\frac{x_\beta x^k x^l}{\abs{x}^8} \\
    =&-3W_{lij\beta}\frac{x^l}{\abs{x}^6}-6W_{ki \beta l}\frac{x_j x^k x^l}{\abs{x}^8}+6  W_{kijl}\frac{x_\beta x^k x^l}{\abs{x}^8}.
\end{align*}
Substituting this formula in \eqref{eq:divint-order1-term} and Taylor-expanding $A_{ij}$, we get
\begin{multline*}
    \frac{2}{3}\int_{\partial B_\gamma}\Big[ -3W_{lij\beta}\frac{x^l}{\abs{x}^6}-6W_{ki \beta l}\frac{x_j x^k x^l}{\abs{x}^8}+6  W_{kijl}\frac{x_\beta x^k x^l}{\abs{x}^8}\Big]\times \\
    \times\Big(A^{ij}(0)+\partial_s A^{ij}(0)x^s+\frac{1}{2}\p^2_{st} A^{ij}(0) x^s x^t+\frac{1}{6}\p^3_{stu} A^{ij}(0) x^s x^t x^u+O(\abs{x}^4)\Big)\frac{x^\beta}{\abs{x}}\,d\sigma.
\end{multline*}
Next, we notice that $W_{ki\beta l}x^k x^lx^\beta=0$ by skew-symmetry and that all the terms involving an \emph{odd} number of derivatives of $A^{ij}$ are zero because we got an integral over a sphere of an \emph{odd} number of coordinate functions (which vanishes because odd in at least one direction). Moreover, we have $A(0)=0$, $\nabla A(0)=0$, see \eqref{eq:AgradA-vanish}, \eqref{eq:AgradA-vanish-quotients}. Hence we obtain
\begin{align*}
    \eqref{eq:divint-order1-term}=\frac{1}{3}\int_{\partial B_\gamma}\Big[ -3W_{lij\beta}\frac{x^l}{\abs{x}^6}+6  W_{kijl}\frac{x_\beta x^k x^l}{\abs{x}^8}\Big]\partial^2_{st}A^{ij}(0)\frac{x^sx^tx^\beta}{\abs{x}}\,d\sigma+ O(\gamma^2), 
\end{align*}
which, after a change of variables, can be rewritten as
\begin{align*}
    \eqref{eq:divint-order1-term}=W_{kijl}\partial^2_{st}A^{ij}(0)\int_{\Sp^3} z^k z^l z^s z^t\,d\sigma +O(\gamma^2).
\end{align*}
We can now employ the identity (cf. \cite{brendle-2008-JAMS-Yamabe}, Corollary 29)
\begin{equation}\label{eq:brendle-form}
    \int_{\Sp^3}x^\mu x^\nu x^k x^l\,d\sigma_{\Sp^3}=\frac{\pi^2}{12}(\delta^{\mu\nu}\delta^{kl}+\delta^{\mu k}\delta^{\nu l}+\delta^{\mu l}\delta^{\nu k}),
\end{equation}
to finally obtain that
\begin{align}\label{eq:divint-exp-1}
    \eqref{eq:divint-order1-term}=\frac{\pi^2}{6}W_{kijl}\partial^{kl} A^{ij}(0)+O(\gamma^2).
\end{align}

\textbullet \, Consider now \eqref{eq:divint-hot1}. Taylor-expanding the terms involving derivatives of $A$ and performing a subsequent change of variables, one gets
\begin{equation*}
    \eqref{eq:divint-hot1}=O(\gamma^2).
\end{equation*}
Notice that there are no terms of order $\gamma$ because the integrals involve an odd number of coordinate functions. The same argument also shows that
\begin{equation*}
    \eqref{eq:divint-hot3}=O(\gamma^2).
\end{equation*}

\textbullet \, Finally, a change of variables shows that $\eqref{eq:divint-hot2-1}+\eqref{eq:divint-hot2-2}=O(\gamma^3)$.

\medskip

Combining the last two items and \eqref{eq:divint-mainexp}, \eqref{eq:divint-exp-1}, we finally obtain \eqref{eq:divint-expansion},  concluding our proof.
\end{proof}

\medskip

We now focus on the first integral in \eqref{eq:weyl-g_b-boundary-formula}; using \eqref{eq:weyl-lin-tfree-eucl}, \eqref{eq:F-KA-splitting} and the fact that $K$ is TT, we compute
\begin{align}
\notag
    \dot{W}_{\eta i j \beta}\partial^\beta F^{ij}&\nu^\eta=\frac{1}{2}\Big[\partial^2_{\eta j}F_{i\beta}+\partial^2_{i \beta}F_{\eta j}-\partial^2_{\eta \beta}F_{ij}-\partial^2_{i j}F_{\eta \beta}\Big]\partial^\beta F^{ij}\nu^\eta+\frac{1}{4}\Big[\Delta F_{ij}\partial^\beta F^{ij}\nu_\beta-\Delta F_{i \beta}\partial^\beta F^{ij}\nu_j\Big] \\
    \notag
    +\frac{1}{4}\Big[&\Delta F_{\eta\beta}\p^\beta (\tr A) \nu^\eta-\Delta F_{\eta j}(\delta A)^j\nu^\eta\Big] \\
    \notag
    +\frac{1}{4}\Big[&\p_i(\delta A)_\beta +\p_\beta(\delta A)_i-\p^2_{i\beta}(\tr A)\Big]\p^\beta F^{ij}\nu_j-\frac{1}{4}\Big[\p_i(\delta A)_j +\p_j(\delta A)_i-\p^2_{ij}(\tr A)\Big]\p^\beta F^{ij}\nu_\beta \\
    \notag
    +\frac{1}{4}\Big[&\p_\eta(\Delta A)_j+\p_j(\delta A)_\eta-\p^2_{\eta j}(\tr A)\Big](\delta A)^j\nu^\eta-\frac{1}{4}\Big[\p_\eta(\Delta A)_\beta+\p_\beta(\delta A)_\eta-\p^2_{\eta \beta}(\tr A)\Big]\p^\beta(\tr A)\nu^\eta \\
    \label{eq:term-prod}
    +\frac{1}{6}\Big[&-\Delta(\tr A)+(\delta^2 A)\Big]\big(\p^\beta(\tr A)\nu_\beta-(\delta A)^j\nu_j\big).
\end{align}


To begin, we show that the integrals of all trace and divergence terms of \eqref{eq:term-prod} generate higher-order contributions.

\begin{lemma}
    There holds
    \begin{align}
    \notag
        \int_{\p B_\gamma}\dot{W}_{\eta i j \beta}\partial^\beta F^{ij}\nu^\eta\,d\sigma=&\frac{1}{2}\int_{\p B_\gamma}\big[\partial^2_{\eta j}F_{i\beta}+\partial^2_{i \beta}F_{\eta j}-\partial^2_{\eta \beta}F_{ij}-\partial^2_{i j}F_{\eta \beta}\big]\partial^\beta F^{ij}\nu^\eta\,d\sigma \\
        \label{eq:div-terms-est}
        +\frac{1}{4}&\int_{\p B_\gamma}\big[\Delta F_{ij}\partial^\beta F^{ij}\nu_\beta-\Delta F_{i \beta}\partial^\beta F^{ij}\nu_j\big]\,d\sigma+ O(\gamma^2), \qquad \text{as $\gamma\to 0^+$.}
    \end{align}
\end{lemma}

\begin{proof}
We consider each divergence or trace term of \eqref{eq:term-prod} separately.

\textbullet \, Consider the term in the second line of \eqref{eq:term-prod}.  Integrating and recalling \eqref{eq:F-KA-splitting}, one has
\begin{align*}
    \frac{1}{4}&\int_{\p B_\gamma}\Big[\Delta F_{\eta\beta}\p^\beta (\tr A) -\Delta F_{\eta j}(\delta A)^j\Big]\nu^\eta\,d\sigma \\
    &=\frac{1}{4}\int_{\p B_\gamma}\Big[\Big(\frac{8}{3}W_{k\eta \beta l}\frac{x^k x^l}{\abs{x}^6}+\Delta A_{\eta \beta}(x)\Big)\p^\beta(\tr A)-\Big(\frac{8}{3}W_{k\eta j l}\frac{x^k x^l}{\abs{x}^6}+\Delta A_{\eta j}(x)\Big)(\delta A)^j\Big]\frac{x^\eta}{\abs{x}}\,d\sigma \\
    &=\frac{2}{3}\int_{\p B_\gamma}\Big[\big( W_{k \eta \beta l} \p^\beta (\tr A)- W_{k\eta j l}(\delta A)^j\big)\frac{x^k x^l x^\eta}{\abs{x}^7}\,d\sigma + O(\gamma^3)=O(\gamma^3),
\end{align*}
due to symmetry properties of Weyl's tensor.

\medskip

\textbullet \, Consider the first term in the third line of \eqref{eq:term-prod}, namely $\frac{1}{4}\big[\p_i(\delta A)_\beta +\p_\beta(\delta A)_i-\p^2_{i\beta}(\tr A)\big]\p^\beta F^{ij}\nu_j$. One has
\begin{align*}
    &\frac{1}{4}\int_{\p B_\gamma}\Big[\p_i(\delta A)_\beta(x) +\p_\beta(\delta A)_i(x)-\p^2_{i\beta}(\tr A)(x)\Big]\p^\beta\Big(-\frac{1}{3}W\indices{_k^{ij}_l}\frac{x^k x^l}{\abs{x}^4}+A^{ij}(x)\Big)\frac{x_j}{\abs{x}}\,d\sigma \\
    &=-\frac{1}{12}\int_{\p B_\gamma}\Big[\p_i(\delta A)_\beta(x) +\p_\beta(\delta A)_i(x)-\p^2_{i\beta}(\tr A)(x)\Big]W\indices{_k^{ij}_l}\Big(\frac{\delta^{\beta k}x^l +x^k \delta^{\beta l}}{\abs{x}^4}-4\frac{x^\beta x^k x^l}{\abs{x}^6}\Big)\frac{x_j}{\abs{x}}\,d\sigma + O(\gamma^3).
\end{align*}
Taylor-expanding the $A$-terms we get
\begin{align*}
    -\frac{1}{12}\int_{\p B_\gamma}&\Big[\p_i(\delta A)_\beta(0) +\p_\beta(\delta A)_i(0)-\p^2_{i\beta}(\tr A)(0)\Big]\Big(\frac{W\indices{^{\beta i j}_k}x^k+W\indices{_k^{i j \beta}}x^k}{\abs{x}^4}-4 W\indices{_k^{ij}_l}\frac{x^\beta x^k x^l}{\abs{x}^6}\Big)\frac{x_j}{\abs{x}}\,d\sigma \\
    &-\frac{1}{12}\int_{\p B_\gamma}\bigg\{\Big[\p^2_{is}(\delta A)_\beta(0) +\p^2_{\beta s}(\delta A)_i(0)-\p^3_{i\beta s}(\tr A)(0)\Big]x^s \times \\
    &\qquad \qquad \qquad \qquad \qquad \qquad \times 
    \Big(\frac{W\indices{^{\beta i j}_k}x^k+W\indices{_k^{i j \beta}}x^k}{\abs{x}^4}-4 W\indices{_k^{ij}_l}\frac{x^\beta x^k x^l}{\abs{x}^6}\Big)\frac{x_j}{\abs{x}}\bigg\}\,d\sigma +O(\gamma^2).
\end{align*}
However, the first integral above vanishes due to the skew-symmetry and trace-free property of Weyl's tensor, while the second integral vanishes because of the odd number of coordinate functions. Hence
\begin{align*}
    \frac{1}{4}\int_{\p B_\gamma}\Big[\p_i(\delta A)_\beta(x) +\p_\beta(\delta A)_i(x)-\p^2_{i\beta}(\tr A)(x)\Big]\p^\beta F^{ij}\nu_j\,d\sigma=O(\gamma^2)
\end{align*}
    
\medskip

\textbullet \, A computation similar to the previous one also shows that the second term in the third line of \eqref{eq:term-prod} satifies
\begin{equation*}
    -\frac{1}{4}\int_{\p B_\gamma}\Big[(\p_i(\delta A)_j +\p_j(\delta A)_i-\p^2_{ij}(\tr A)\Big]\p^\beta F^{ij}\nu_\beta\,d\sigma=O(\gamma^2).
\end{equation*}

\medskip

 Finally, it is immediate to see that the terms in the last two lines of \eqref{eq:term-prod} are of order $O(\gamma^3)$.
This concludes the proof.
\end{proof}

\medskip

We now compute the integral over $\partial B_\gamma$ of a generic term in the first line of \eqref{eq:term-prod}; recalling \eqref{eq:F-KA-splitting} one has
\begin{align*}
    \int_{\partial B_\gamma} \partial^2_{\alpha \beta}F_{ij}\partial_\tau F_{st}\nu_\eta\,d\sigma&=\int_{\partial B_\gamma} \partial^2_{\alpha \beta}K_{ij}\partial_\tau K_{st}\nu_\eta\,d\sigma \\
    &-\frac{1}{3}\int_{\partial B_\gamma} \partial^2_{\alpha \beta} A_{ij}(x) W_{\mu s t \nu}\bigg(\frac{\delta^\mu_\tau x^\nu +x^\mu \delta_\tau^\nu}{\abs{x}^4}-4\frac{x_\tau x^\mu x^\nu}{\abs{x}^6}\bigg)\frac{x^\eta}{\abs{x}}\,d\sigma \\
    -\frac{1}{3}\int_{\partial B_\gamma}W_{kijl}&\bigg(\frac{\delta_\alpha^k \delta_\beta^l+\delta_\beta^k \delta_\alpha^l}{\abs{x}^4}-4\frac{x_\beta(\delta_\alpha^k x^l+x^k\delta_\alpha^l)}{\abs{x}^6}-4\frac{\delta_{\alpha\beta}x^k x^l+x_\alpha(\delta_\beta^k x^l+x^k \delta_\beta^l)}{\abs{x}^6} \\
    & \quad \quad \quad +24\frac{x_\alpha x_\beta x^k x^l}{\abs{x}^8}\bigg)\partial_\tau A_{st}(x)\frac{x^\eta}{\abs{x}}\,d\sigma+\int_{\partial B_\gamma}\partial^2_{\alpha \beta}A_{ij}(x)\partial_\tau A_{st}(x)\frac{x^\eta}{\abs{x}}\,d\sigma.
\end{align*}
Taylor-expanding $\partial_\tau A_{st}$ and $\partial^2_{\alpha \beta} A_{ij}$ at $0$, using the symmetry of the integrals and a change of variables, one gets
\begin{align}
\notag
    \int_{\partial B_\gamma} \partial^2_{\alpha \beta}F_{ij}\partial_\tau F_{st}\nu_\eta\,d\sigma&=\int_{\partial B_\gamma} \partial^2_{\alpha \beta}K_{ij}\partial_\tau K_{st}\nu_\eta\,d\sigma \\
    \notag
    &-\frac{1}{3}\int_{\Sp^3} \partial^2_{\alpha \beta} A_{ij}(0) W_{\mu s t \nu}\big(\delta^\mu_\tau x^\nu +x^\mu \delta_\tau^\nu-4x_\tau x^\mu x^\nu\big)x^\eta\,d\sigma \\
    \notag
    -\frac{1}{3}\int_{\Sp^3}W_{kijl}&\Big[\delta_\alpha^k \delta_\beta^l+\delta_\beta^k \delta_\alpha^l-4x_\beta(\delta_\alpha^k x^l+x^k\delta_\alpha^l)-4\big(\delta_{\alpha\beta}x^k x^l+x_\alpha(\delta_\beta^k x^l+x^k \delta_\beta^l)\big) \\
    \label{eq:bd-exp-order1}
    & \quad \quad \quad \quad \quad +24x_\alpha x_\beta x^k x^l\Big]\partial^2_{\tau\xi} A_{st}(0)x^\eta x^\xi\,d\sigma+ O(\gamma^2).
\end{align}

\begin{lemma}\label{lem:Z-en-bal-2}
    There holds
    \begin{align}\label{eq:nondivint-expansion}
        \int_{\p B_\gamma}\dot{W}_{\eta i j \beta}\partial^\beta F^{ij}\nu^\eta\,d\sigma=\mathcal{C}_2\gamma^{-4}+\frac{\pi^2}{6}\p^{kl}A^{ij}(0)W_{kijl}+O(\gamma^2), \qquad \text{as $\gamma\to 0^+$,}
    \end{align}
    for a real constant $\mathcal{C}_2=\mathcal{C}_2(W^M(p))$.
\end{lemma}

\begin{proof}
By looking at  \eqref{eq:F-KA-splitting}, \eqref{eq:lin-weyl-f1}, it is easy to see (after a change of variables) that the first term on the RHS of \eqref{eq:bd-exp-order1} is of order $\gamma^{-4}$. Since from \eqref{eq:bd-exp-order1} it is clear that all the other terms are of order $1$, we conclude from  \eqref{eq:div-terms-est} and \eqref{eq:bd-exp-order1} that
\begin{align}
\notag
    \frac{1}{2}\int_{\p B_\gamma}\big[\partial^2_{\eta j}K_{i\beta}&+\partial^2_{i \beta}K_{\eta j}-\partial^2_{\eta \beta}K_{ij}-\partial^2_{i j}K_{\eta \beta}\big]\partial^\beta K^{ij}\nu^\eta\,d\sigma \\
    \label{eq:nondivint-mainexp}
        +\frac{1}{4}&\int_{\p B_\gamma}\big[\Delta K_{ij}\partial^\beta K^{ij}\nu_\beta-\Delta K_{i \beta}\partial^\beta K^{ij}\nu_j\big]\,d\sigma=: \mathcal{C}_2 \gamma^{-4},
\end{align}
where $\mathcal{C}_2=\mathcal{C}_2(W^M(p))$ is the constant appearing in \eqref{eq:nondivint-expansion}.

We now focus on the contribution of order $1$;
by looking again at \eqref{eq:div-terms-est} and \eqref{eq:bd-exp-order1}, we see that it is given by the following integrals:
\begin{multline}
\label{eq:big-formula}
    -\frac{1}{6}\int_{\Sp^3}\big(\partial^{\eta j}A^{i\beta}(0)+\partial^{i \beta}A^{\eta j}(0)-\partial^{\eta \beta}A^{ij}(0)-\partial^{ij}A^{\eta \beta}(0)\big)W_{\mu i j \nu}(\delta_\beta^\mu x^\nu +x^\mu \delta_\beta^\nu-4x_\beta x^\mu x^\nu)x_\eta \\
    -\frac{1}{12}\int_{\Sp^3}\Big[W_{\mu i j \nu}(\delta_\beta^\mu x^\nu + x^\mu \delta_\beta^\nu-4 x_\beta x^\mu x^\nu)\big( \Delta A^{ij}(0) x^\beta -\Delta A^{i\beta}(0) x^j\big) \\
    -\frac{1}{6}\int_{\Sp^3}W_{ki\beta l}\Big[\delta_\eta^k \delta_j^l+\delta_j^k \delta_\eta^l-4x_\eta(\delta_j^k x^l+x^k\delta_j^l)-4\big(\delta_{j\eta}x^k x^l+x_j(\delta_\eta^k x^l+x^k \delta_\eta^l)\big)+24x_j x_\eta x^k x^l\Big]\partial^{\beta\xi} A^{ij}(0)x^\eta x_\xi \\
    -\frac{1}{6}\int_{\Sp^3}W_{k\eta jl}\Big[\delta_i^k \delta_\beta^l+\delta_\beta^k \delta_i^l-4x_\beta(\delta_i^k x^l+x^k\delta_i^l)-4\big(\delta_{i\beta}x^k x^l+x_i(\delta_\beta^k x^l+x^k \delta_\beta^l)\big)+24x_i x_\beta x^k x^l\Big]\partial^{\beta\xi} A^{ij}(0)x^\eta x_\xi \\
    +\frac{1}{6}\int_{\Sp^3}W_{kijl}\Big[\delta_\eta^k \delta_\beta^l+\delta_\beta^k \delta_\eta^l-4x_\beta(\delta_\eta^k x^l+x^k\delta_\eta^l)-4\big(\delta_{\eta\beta}x^k x^l+x_\eta(\delta_\beta^k x^l+x^k \delta_\beta^l)\big)+24x_\eta x_\beta x^k x^l\Big]\partial^{\beta\xi} A^{ij}(0)x^\eta x_\xi \\
    +\frac{1}{6}\int_{\Sp^3}W_{k\eta \beta l}\Big[\delta_i^k \delta_j^l+\delta_j^k \delta_i^l-4x_j(\delta_i^k x^l+x^k\delta_i^l)-4\big(\delta_{ij}x^k x^l+x_i(\delta_j^k x^l+x^k \delta_j^l)\big)+24x_i x_j x^k x^l\Big]\partial^{\beta\xi} A^{ij}(0)x^\eta x_\xi \\
    -\frac{1}{12}\sum_\alpha\int_{\Sp^3}W_{k i j l}\Big[2\delta_\alpha^k \delta_\alpha^l-4x_\alpha(\delta_\alpha^k x^l+x^k\delta_\alpha^l)-4x^k x^l-4x_\alpha(\delta_\alpha^k x^l+x^k \delta_\alpha^l)+24 (x_\alpha)^2 x^k x^l\Big]\partial^{\beta\xi} A^{ij}(0)x_\beta x_\xi \\
    +\frac{1}{12}\sum_\alpha\int_{\Sp^3}W_{k i \beta l}\Big[2\delta_\alpha^k \delta_\alpha^l-4x_\alpha(\delta_\alpha^k x^l+x^k\delta_\alpha^l)-4x^k x^l-4x_\alpha(\delta_\alpha^k x^l+x^k \delta_\alpha^l)+24 (x_\alpha)^2 x^k x^l\Big]\partial^{\beta\xi} A^{ij}(0)x_j x_\xi
\end{multline}
    We need to compute each term in \eqref{eq:big-formula}. We start noticing that
    \begin{align}
    \notag
        \int_{\Sp^3}W_{\mu i j \nu}(\delta_\beta^\mu x^\nu +x^\mu \delta_\beta^\nu -4 x_\beta x^\mu x^\nu)x_\eta&=\int_{\Sp^3}\big( W_{\beta i j \nu}x^\nu +W_{\mu i j \beta }x^\mu\big)x_\eta -4W_{\mu i j \nu}\int_{\Sp^3}x_\beta x^\mu x^\nu x_\eta \\
        \notag
        &=\big( W_{\beta i j \eta} +W_{\eta i j \beta}\big)\int_{\Sp^3}(x_\eta)^2-4W_{\mu i j \nu}\int_{\Sp^3}x_\beta x^\mu x^\nu x_\eta \\
        \notag
        &=\frac{\pi^2}{2}\big( W_{\beta i j \eta}+W_{\eta i j \beta}\big)-\frac{\pi^2}{3}W_{\mu i j \nu}(\delta_{\beta \eta}\delta^{\mu \nu}+\delta_\beta^\mu \delta_\eta^\nu +\delta_\beta^\nu \delta_\eta^\mu) \\
        \label{eq:bigf-1}
        &=\frac{\pi^2}{6}\big( W_{\beta i j \eta} +W_{\eta i j \beta}\big),
    \end{align}
    where we used (cf. \cite[Corollary 29]{brendle-2008-JAMS-Yamabe}) 
\begin{align}\label{eq:brendleform-order2}
    \int_{\Sp^3}x_\eta x_\nu=\frac{\pi^2}{2}\delta_{\eta \nu},
    \end{align}
    \eqref{eq:brendle-form} and the trace-free property of Weyl tensor. Similar computations also show that
    \begin{equation}\label{eq:bigf-2}
        \int_{\Sp^3}W_{\mu i j \nu}(\delta_\beta^\mu x^\nu +x^\mu \delta_\beta^\nu -4 x_\beta x^\mu x^\nu)x^\beta=0=\int_{\Sp^3}W_{\mu i j \nu}(\delta_\beta^\mu x^\nu +x^\mu \delta_\beta^\nu -4 x_\beta x^\mu x^\nu)x^j.
    \end{equation}
Next, we got
\begin{align}
\notag
    \int_{\Sp^3}W_{ki\beta l}&\Big[\delta_\eta^k \delta_j^l+\delta_j^k \delta_\eta^l-4x_\eta(\delta_j^k x^l+x^k\delta_j^l)-4\big(\delta_{j\eta}x^k x^l+x_j(\delta_\eta^k x^l+x^k \delta_\eta^l)\big)+24x_j x_\eta x^k x^l\Big]\partial^{\beta\xi} A^{ij}x^\eta x_\xi \\
    \notag
    &=\int\big( W_{\eta i \beta j} +W_{j i \beta \eta}\big)\partial^{\beta\xi}A^{ij}x^\eta x_\xi -4\int\big(W_{j i \beta k}+W_{ki\beta j}\big)\partial^{\beta\xi}A^{ij}x^kx_\xi \\
    \notag
    -4\int& \Big[ W_{ki\beta l}\partial^{\beta \xi}A^{ij}x^k x^l x_j x_\xi +\big( W_{\eta i \beta k}+W_{ k i \beta \eta}\big)\partial^{\beta \xi} A^{ij}x^k x^\eta x_j x_\xi\Big]+24 W_{ki\beta l}\partial^{\beta \xi}A^{ij}\int x^k x^l x_j x_\xi \\
    \notag
    &=-\frac{3}{2}\pi^2 \partial^{\beta \xi}A^{ij}\big(W_{\xi i \beta j}+W_{j i \beta \xi}\big)-\frac{\pi^2}{3}\partial^{\beta \xi}A^{ij}\big(3W_{ ji \beta \xi}+ 3W_{\xi i \beta j}\big) +2\pi^2 \big(W_{j i \beta \xi}+W_{\xi i \beta j}\big) \\
    \label{eq:bigf-3}
    &=-\frac{\pi^2}{2}\partial^{\beta\xi}A^{ij}(0)\big(W_{j i \beta \xi}+W_{\xi i \beta j}\big).
\end{align}
Similarly,
\begin{align}
    \notag\int_{\Sp^3}&W_{k\eta jl}\Big[\delta_i^k \delta_\beta^l+\delta_\beta^k \delta_i^l-4x_\beta(\delta_i^k x^l+x^k\delta_i^l)-4\big(\delta_{i\beta}x^k x^l+x_i(\delta_\beta^k x^l+x^k \delta_\beta^l)\big)+24x_i x_\beta x^k x^l\Big]\partial^{\beta\xi} A^{ij}x^\eta x_\xi \\
    \notag
    &=\int \big( W_{i \eta j \beta}+W_{\beta \eta j i}\big)\partial^{\beta\xi} A^{ij}x^\eta x_\xi-4\int\big( W_{i \eta j k}+W_{k \eta j i}\big)\partial^{\beta\xi} A^{ij} x_\beta x^k x^\eta x_\xi \\
    \notag
    &-4 \sum_i  W_{k\eta j l}\partial^{i \xi}A^{ij}\int x^k x^l x^\eta x_\xi-4\int\big(W_{\beta \eta j k}+W_{k \eta j \beta}\big)\partial^{\beta\xi} A^{ij} x_i x^k x^\eta x_\xi +0 \\
    \notag
    &=\frac{\pi^2}{2}\partial^{\beta\xi} A^{ij}\big( W_{i \xi j \beta}+ W_{\beta \xi j i}\big)-\frac{\pi^2}{3}\partial^{\beta\xi} A^{ij}\big(W_{ i \xi j \beta}+W_{\beta \xi j i}+W_{ i \beta j \xi}+W_{ \xi \beta j i}\big) \\
    \notag
    &-0-\frac{\pi^2}{3}\partial^{\beta\xi} A^{ij}\big( W_{\beta \xi j i }+W_{ i \xi j \beta}+W_{\beta i j \xi}+W_{\xi i j \beta}\big) \\
    \label{eq:bigf-4}
    &=\frac{\pi^2}{6}\partial^{\beta\xi} A^{ij}\big(W_{i \xi j \beta}+W_{\beta \xi j i}\big).
\end{align}
Also, by exchanging the indices $\beta $ and $j$ in \eqref{eq:bigf-3} (except on $A$), we immediately see that
\begin{align}
\notag
    \int_{\Sp^3}W_{kijl}&\Big[\delta_\eta^k \delta_\beta^l+\delta_\beta^k \delta_\eta^l-4x_\beta(\delta_\eta^k x^l+x^k\delta_\eta^l)-4\big(\delta_{\eta\beta}x^k x^l+x_\eta(\delta_\beta^k x^l+x^k \delta_\beta^l)\big)+24x_\eta x_\beta x^k x^l\Big]\partial^{\beta\xi} A^{ij}x^\eta x_\xi \\
    \label{eq:bigf-5}
    &=-\frac{\pi^2}{2}\partial^{\beta\xi} A^{ij}(0)\big(W_{\xi i j \beta}+W_{\beta i j \xi}\big).
\end{align}
Exchanging instead $\beta$ and $j$ in \eqref{eq:bigf-4} (except on $A$), we got
\begin{align}
    \notag
    \int_{\Sp^3}W_{k\eta \beta l}&\Big[\delta_i^k \delta_j^l+\delta_j^k \delta_i^l-4x_j(\delta_i^k x^l+x^k\delta_i^l)-4\big(\delta_{ij}x^k x^l+x_i(\delta_j^k x^l+x^k \delta_j^l)\big)+24x_i x_j x^k x^l\Big]\partial^{\beta\xi} A^{ij}x^\eta x_\xi \\
    \label{eq:bigf-6}
    &=\frac{\pi^2}{6}\partial^{\beta \xi}A^{ij}(0)\big(W_{i \xi \beta j}+W_{j \xi \beta i}\big).
\end{align}
Finally, we also have
\begin{align}
    \notag
    \sum_\alpha\int_{\Sp^3}W_{k i j l}&\Big[2\delta_\alpha^k \delta_\alpha^l-4x_\alpha(\delta_\alpha^k x^l+x^k\delta_\alpha^l)-4x^k x^l-4x_\alpha(\delta_\alpha^k x^l+x^k \delta_\alpha^l)+24 (x_\alpha)^2 x^k x^l\Big]\partial^{\beta\xi} A^{ij}x_\beta x_\xi \\
    \label{eq:bigf-7}
    &=-8\partial^{\beta \xi}A^{ij}W_{kijl}\int x^k x^l x_\beta x_\xi=-\frac{2}{3}\pi^2\partial^{\beta \xi}A^{ij}\big(W_{\beta i j \xi}+W_{\xi i j \beta}\big),
\end{align}
and
\begin{align}
    \notag
    \sum_\alpha\int_{\Sp^3}W_{k i \beta l}&\Big[2\delta_\alpha^k \delta_\alpha^l-4x_\alpha(\delta_\alpha^k x^l+x^k\delta_\alpha^l)-4x^k x^l-4x_\alpha(\delta_\alpha^k x^l+x^k \delta_\alpha^l)+24 (x_\alpha)^2 x^k x^l\Big]\partial^{\beta\xi} A^{ij}x_j x_\xi \\
    \label{eq:bigf-8}
    &=-8\partial^{\beta\xi} A^{ij} W_{ k i \beta l}\int x^k x^l x_j x_\xi=-\frac{2}{3}\pi^2\partial^{\beta\xi} A^{ij}\big( W_{ j i \beta \xi}+W_{\xi i \beta j}\big).
\end{align}
Substituting \eqref{eq:bigf-1}, \eqref{eq:bigf-2}, \eqref{eq:bigf-3}, \eqref{eq:bigf-4}, \eqref{eq:bigf-5}, \eqref{eq:bigf-6}, \eqref{eq:bigf-7} and \eqref{eq:bigf-8} inside \eqref{eq:big-formula} and using the symmetries of the curvature tensor and the fact that $A$ is symmetric, we get
\begin{align*}
    &-\frac{1}{6}\Big[\big(\partial^{\eta j}A^{i\beta}(0)+\partial^{i \beta}A^{\eta j}(0)-\partial^{\eta \beta}A^{ij}(0)-\partial^{ij}A^{\eta \beta}(0)\big)\frac{\pi^2}{6}\big( W_{\beta i j \eta} + W_{\eta i j \beta}\big)\Big] \\
    &-\frac{1}{6}\Big[-\frac{\pi^2}{2}\partial^{\beta \xi}A^{ij}(0)\big(W_{j i \beta \xi}+W_{\xi i \beta j}\big)+\frac{\pi^2}{6}\partial^{\beta \xi}A^{ij}(0)\big( W_{i \xi j \beta}+W_{\beta \xi j i}\big)\Big] \\
    &+\frac{1}{6}\Big[-\frac{\pi^2}{2}\partial^{\beta \xi}A^{ij}(0)\big(W_{\xi i j \beta}+W_{\beta i j \xi}\big)+\frac{\pi^2}{6}\partial^{\beta \xi}A^{ij}(0)\big(W_{i \xi \beta j}+ W_{j \xi \beta i }\big)\Big] \\
    &+\frac{\pi^2}{18}\partial^{\beta \xi}A^{ij}(0)\big(W_{\beta i j \xi}+ W_{\xi i j \beta}\big) -\frac{\pi^2}{18}\partial^{\beta \xi}A^{ij}(0)\big( W_{j i \beta \xi}+W_{ \xi i \beta j}\big)  \\
    &=-\frac{1}{6}\Big[\frac{\pi^2}{3}\big(\partial^{i\beta}A^{\eta j}(0)-\partial^{\eta \beta}A^{ij}(0)\big)\big( W_{\beta i j \eta}+W_{\eta i j \beta}\big)-\frac{\pi^2}{3}\partial^{\beta \xi} A^{ij }(0)\big(W_{\xi i \beta j}+W_{j i \beta \xi}\big)\Big] \\
    &-\frac{\pi^2}{18}\partial^{\beta \xi}A^{ij}(0)\big( W_{\xi i j \beta}+W_{\beta i j \xi}\big)+\frac{\pi^2}{18}\partial^{\beta \xi}A^{ij}(0)\big(W_{\beta i j \xi}+ W_{\xi i j \beta}\big) -\frac{\pi^2}{18}\partial^{\beta \xi}A^{ij}(0)\big( W_{j i \beta \xi}+W_{ \xi i \beta j}\big) \\
    &=-\frac{\pi^2}{18}\partial^{i \beta}A^{\eta j}(0)\big( W_{\beta i j \eta}+W_{\eta i j \beta}\big)+\frac{\pi^2}{18}\partial^{\beta \xi}A^{ij}(0)\big( W_{\beta i j \xi}+W_{\xi i j \beta}\big).
\end{align*}
Finally, using again the symmetries of $W$ and $A$ together with Bianchi identity, we obtain the second term in the RHS of \eqref{eq:nondivint-expansion}.

Since the remainder is $O(\gamma^2)$ (see \eqref{eq:bd-exp-order1}), the result follows.
\end{proof}

It only remains to compute explicitly the coefficients $\mathcal{C}_1$ and $\mathcal{C}_2$ appearing in \eqref{eq:divint-expansion} and \eqref{eq:nondivint-expansion}.

\begin{lemma}\label{lem:C_1+C_2-term}
    Let $\mathcal{C}_1, \mathcal{C}_2$ be defined as in \eqref{eq:divint-mainexp}, \eqref{eq:nondivint-mainexp} respectively. Then
    \begin{equation}\label{eq:C_1+C_2}
        \mathcal{C}_2-\mathcal{C}_1=\frac{\pi^2}{4}\abs{W^M(p)}^2.
    \end{equation}
\end{lemma}
\begin{proof}
    By definition, we got
    \begin{align*}
        (\mathcal{C}_2-\mathcal{C}_1)\gamma^{-4}=&\frac{1}{2}\int_{\p B_\gamma}\big[\partial^2_{\eta j}K_{i\beta}+\partial^2_{i \beta}K_{\eta j}-\partial^2_{\eta \beta}K_{ij}-\partial^2_{i j}K_{\eta \beta}\big]\partial^\beta K^{ij}\nu^\eta\,d\sigma \\
       & +\frac{1}{4}\int_{\p B_\gamma}\big[\Delta K_{ij}\partial^\beta K^{ij}\nu_\beta-\Delta K_{i \beta}\partial^\beta K^{ij}\nu_j\big]\,d\sigma. \\
       &-\frac{1}{4}\int_{\partial B_\gamma}\Big[\p_j \Delta K_{i\beta}-\p_\beta \Delta K_{ij}\Big] K^{ij}\nu^\beta\,d\sigma 
    \end{align*}
    
    Consider the family of metrics $g_t:=g_E+t K$. Expanding the Weyl energy of $g_t$ at $t=0$ in $\R^4\backslash B_\gamma(0)$ (as in \eqref{eq:weyl-g_b-expans}, \eqref{eq:second-der-weyl}), we have
    \begin{align*}
        \int_{\R^4\backslash B_\gamma}\abs{W^{g_t}}^2\,dV_{g_t}=-2t^2\int_{\R^4\backslash B_\gamma}\dot{W}_{\alpha i j \beta}\p^\alpha \p^\beta K^{ij}\,dx+O(t^3), \qquad \text{as $t\to 0$.}
    \end{align*}
    By looking at the definition of $K$ \eqref{eq:F-KA-splitting} and the formula for the linearized Weyl tensor \eqref{eq:weyl-lin-tfree-eucl}, we notice that all the integrands on the RHS decay as $\abs{x}^{-8}$ as $\abs{x}\to\infty$, therefore we can integrate by parts twice as in \eqref{eq:weyl-second-der-parts}. Recalling further that $K$ is TT and biharmonic (so in the kernel of the linearized Weyl at $g_E$), as in \eqref{eq:weyl-g_b-boundary-formula} we get
\begin{align*}
     \int_{\R^4\backslash B_\gamma}\abs{W^{g_t}}^2\,dV_{g_t}=2t^2\bigg(\int_{\p B_\gamma} \dot{W}_{\eta i j \beta}\p^\beta K^{ij}\nu^\eta \,d\sigma-\int_{\p B_\gamma}\p^\alpha \dot{W}_{\alpha i j \eta}K^{ij} \nu^\eta\,d\sigma\bigg)+O(t^3),
\end{align*}
where $\nu=x/\abs{x}$ points toward the interior of $\R^4\backslash B_\gamma$.
But now, using again \eqref{eq:weyl-lin-tfree-eucl} we see that indeed
\begin{align}\label{eq:C_1C_1lem-id1}
    \int_{\R^4\backslash B_\gamma}\abs{W^{g_t}}^2\,dV_{g_t}=2\frac{t^2}{\gamma^4}(\mathcal{C}_2-\mathcal{C}_1)+ O(t^3).
\end{align}

Let us define
\begin{equation*}
    g(x)=\Big(\delta_{ij}-\frac{1}{3}W_{kijl}x^k x^l\Big)\,dx^i dx^j.
\end{equation*}
If we let $y:=I(x)={x}/{\abs{x}^2}$, then
\begin{equation*}
    \tilde{g}(y):=I^*(\abs{x}^{-4}g)(y)=\Big(\delta_{ij}-\frac{1}{3}W_{kijl}\frac{y^k y^l}{\abs{y}^4}\Big)\,dy^i dy^j,
\end{equation*}
and therefore $t \tilde{g}=g_t$. Hence
\begin{align}\label{eq:C_1C_2lem-id2}
    \int_{\R^4\backslash B_\gamma}\abs{W^{g_t}}^2\,dV_{g_t}=\int_{\R^4\backslash B_{\gamma / \sqrt{t}}} \abs{W^{\tilde{g}}}^2\, dV_{\tilde{g}}=\int_{ B_{\sqrt{t}/\gamma}}\abs{W^g}^2\,dV_g=\frac{\pi^2}{2}\frac{t^2}{\gamma^4}\abs{W^M(p)}^2+O\Big(\frac{t^3}{\gamma^6}\Big).
\end{align}
The result then follows from \eqref{eq:C_1C_1lem-id1} and \eqref{eq:C_1C_2lem-id2}.
\end{proof}

As a consequence of \eqref{eq:weyl-g_b-boundary-formula}, Lemma \ref{lem:Z-en-bal-1} and Lemma \ref{lem:Z-en-bal-2} and Lemma \ref{lem:C_1+C_2-term}, we obtain the expansion for the Weyl energy on $Z$:

\begin{corollary}\label{cor:Weyl-energy-g_b-Z}
    Let $g_b$ be defined as in \eqref{eq:g_b-definition}. Then
    \begin{equation}\label{eq:weyl-energy-g_b-expansion}
    \w^Z_\gamma(g_b)=\int_{Z\backslash B_\gamma^{g_Z}(q)}\abs{W^{g_b}}^2\,dV_{g_b}=\frac{\pi^2}{2}\frac{b^4}{\gamma^4}\abs{W^M(p)}^2+O(b^4\gamma^2), \qquad \text{as $0\ll b\ll\gamma\ll1$.}
\end{equation}
\end{corollary}

\section{Weyl energy of the perturbed non LCF metric}\label{sec:weyl-en-M}

On the Bach-flat manifold $(M,g_M)$, we want to adjust the metric in order to match its inverted and scaled counterpart to the modified metric $g_b$ on $Z$ (defined by \eqref{eq:g_b-definition}). The idea is that, since $A(0)=0$ and $\nabla A(0)=0$ (see \eqref{eq:AgradA-vanish}, \eqref{eq:AgradA-vanish-quotients}), we want to add a perturbation to $g_M$ to match $\nabla^2_{kl} A_{ij}(0)x^kx^l$.

In conformal normal coordinates at $p\in M$, there exists $\epsilon>0$ (small) such that
\begin{align*}
    g_M=\Big(\delta_{ij}-\frac{1}{3}W_{kijl}z^k z^l+\zeta_{ij}(z)\Big)dz^i dz^j, \qquad \abs{z}< 4\epsilon,
\end{align*}
where $\zeta_{ij}(z)=O^{(4)}(\abs{z})^3$ as $\abs{z}\to 0$.
Let us consider a smooth cutoff function $\varphi$ satisfying $0\leq \varphi(r)\leq 1$ $\forall r\geq 0$, $\phi(r)\equiv0$ for $r\geq 2\eps$ and $\varphi(r)\equiv 1$ for $r\leq \eps$. We define a smooth tensor $H$ in $M\backslash\{p\}$ as
\begin{align}\label{eq:H-tens-def}
    H_{ij}(z):=\frac{1}{2}\p^2_{kl}A_{ij}(0) \frac{z^k z^l}{\abs{z}^4}\varphi(\abs{z})=: T_{kijl}\frac{z^k z^l}{\abs{z}^4}\varphi(\abs{z}), \qquad \text{for $\abs{z}\leq 4\eps$,}
\end{align}
with $H\equiv 0$ in $M\backslash B^{g_M}_{2\eps}(p)$. 
Letting $\tilde{g}_M:=g_M+a^4H$, one has
\begin{align*}
    \tilde{g}_M=\Big(\delta_{ij}-\frac{1}{3}W_{kijl}z^k z^l+a^4 T_{kijl}\frac{z^k z^l}{\abs{z}^4}+\zeta_{ij}(z)\Big)dz^i dz^j, \qquad \abs{z}<\eps.
\end{align*}
If we now multiply the metric $\tilde{g}_M$ via a positive function $f$ on $M\backslash\{p\}$, satisfying $f(z)=\abs{z}^{-2}$ near $p$, and call $y={z}/{\abs{z}^2}$ the inverted conformal normal coordinates, by Lemma \ref{lem:perturb-metr-inverse-exp} we get
\begin{align*}
    f^2 \tilde{g}_M=:\tilde{g}_N=\Big(\delta_{ij}-\frac{1}{3}W_{kijl}\frac{y^k y^l}{\abs{y}^4}+a^4T_{kijl}y^k y^l+\tilde{\zeta}_{ij}(y)\Big)dy^i dy^j, \qquad \abs{y}>\frac{1}{\epsilon},
\end{align*}
where $\tilde{\zeta}_{ij}(y)=O^{(4)}(\abs{y}^{-3})$ as $\abs{y}\to+\infty$.
Letting now $g_a:=a^2\tilde{g}_N$, then $x:=ay$ are the inverted normal coordinates for $g_a$ and one has
\begin{align}\label{eq:g_a}
    g_a=\Big(\delta_{ij}-\frac{1}{3}a^2 W_{kijl}\frac{x^k x^l}{\abs{x}^4}+a^2 T_{kijl}x^k x^l+\zeta^a_{ij}(x)\Big)dx^i dx^j, \qquad \abs{x}>\frac{a}{\epsilon},
\end{align}
where now $\zeta^a$ satisfies
\begin{align}\label{eq:zeta^a-err-def}
    \big|\nabla^k \zeta^a_{ij}(x)\big|\leq C a^3 \abs{x}^{-3-k}, \qquad \text{for $\abs{x}>\frac{a}{\eps}$, $k=0,\dots,4$.}
\end{align}

From \eqref{eq:g_a}, \eqref{eq:g_b-definition}, \eqref{eq:H-tens-def} and \eqref{eq:AAbar-formula} we see that, letting $a=b$ we can match $g^M_a$ and $g^Z_b$ in a very precise way (see Section \ref{subsec-gluing-metrics} for the explicit construction).

We now want to compute the Weyl energy of the modified metric $g_a$ defined in \eqref{eq:g_a}. If we glue $(N=M\backslash\{p\},g_a)$ to $(Z,g_b)$ at scale $\gamma$ (with $a=b\ll \gamma\ll1$), then, being $\gamma=\abs{x}=a\abs{z}^{-1}$, we got $\abs{z}=a\gamma^{-1}$ and therefore, using the pointwise conformal invariance of $\abs{W^g}^2\,dV_g$, we need to compute
\begin{align*}
    \int_{M\backslash B_{a\gamma^{-1}}^{g_M}(p)}\abs{W^{\tilde{g}_M}}^2\,dV_{\tilde{g}_M}.
\end{align*}
Recalling the definition of $\tilde{g}_M$, we let $t:=a^4$ and employ the notation 
\begin{align}\label{eq:V(t)-def}
    \tilde{g}_M=g_t:=g_M+t H, \qquad \qquad \qquad V(t):=\int_{M\backslash B_{t^{1/4}\gamma^{-1}}^{g_M}(p)}\abs{W^{g_t}}^2\,dV_{g_t}.
\end{align}
We would now like to Taylor-expand $V$ at $0$, obtaining
\begin{align}\label{eq:V(t)-taylor-exp}
    V(t)=\w^M(g_M)+t\frac{d}{dt}\Big\vert_{t=0} V(t) +o(t).
\end{align}
However, it is not immediately clear whether $V$ and $V'$ are continuous up to zero.
\begin{lemma}
    $V$ is continuous up to zero and $V(0)=\w^M(g_M)$.
\end{lemma}
\begin{proof}
By definition of $g_t$, for $\abs{z}$ close to $0$ we got 
\begin{align}\label{eq:metr-err-g_s}
    g_t= g_M + t O(\abs{z}^{-2}), \qquad \nabla g_t= \nabla g_M + t O(\abs{z}^{-3}), \qquad \nabla^2 g_t=\nabla^2 g_M + t O(\abs{z}^{-4}),
\end{align}
where $\nabla$ denotes the Euclidean gradient (w.r.t. the coordinates $z$).
Moreover, for $a=t^{1/4}<\abs{z}<\epsilon$, the metric $g_t^{-1}$ has the following expansion:
\begin{align*}
    g_t^{-1}=\Big(\delta_{ij}+\frac{1}{3}W_{kijl}z^k z^l-a^4 T_{kijl}\frac{z^k z^l}{\abs{z}^4}+O^{(3)}(\abs{z}^{3})\Big)dz^i dz^j, \qquad a=t^{1/4}<\abs{z}<\epsilon,
\end{align*}
so one also has
\begin{align}\label{eq:metr-err-g_s-inv}
    g_t^{-1}=g_M^{-1}+tO(\abs{z}^{-2}), \qquad \nabla g_t^{-1}=\nabla g_M^{-1}+t O(\abs{z}^{-3}), \qquad t^{1/4}<\abs{z}<\epsilon.
\end{align}
Notice that here the error terms do not depend on $t$ (but the range of $\abs{z}$ in which the expansion holds does).
Now, the Weyl tensor contains up to second order derivatives of the metric, and, if we look closely at $\abs{W^{g_t}}^2$, we see that it has (at worst) additional terms given by perturbations of order $\approx t^2 \abs{z}^{-8}+t\abs{z}^{-4}$. In other words,
\begin{align}\label{eq:weyl-norm-perturb}
    \abs{W^{g_t}}^2_{g_t}=\abs{W^{g_M}}^2_{g_M}+t^2O(\abs{z}^{-8})+t O(\abs{z}^{-4}), \qquad t^{1/4}<\abs{z}<\epsilon.
\end{align}
But now
\begin{align*}
    \int_{B_\epsilon\backslash B_{t^{1/4}\gamma^{-1}}}(t^2\abs{z}^{-8}+t\abs{z}^{-4})\,dz\approx t \int_{t^{1/4}\gamma^{-1}}^\epsilon r^{-1}\,dr =O(t(\log t)) \qquad \text{as $t\to0^+$,}
\end{align*}
which implies that $V(t)\to V(0)=\int_M \abs{W^{g_M}}^2\,dV_{g_M}=\w^M(g_M)$ as $t\to 0^+$ (the error given by the perturbation vanishes in the limit). 
\end{proof}

\begin{lemma}\label{lem:V'(0)-computation}
    $V$ is differentiable at $0$ and one has
    \begin{align}\label{eq:V'(0)}
    V'(0)=-\frac{\pi^2}{2}\gamma^{-4}\abs{W^M(p)}^2+\frac{4}{3}W^M_{kijl}T^{kijl} +O(\gamma^4).
\end{align}
\end{lemma}
\begin{proof}
In order to compute $V'(0)$, we first compute $V'(s)$ for $s>0$ and then take the limit as $s\to0^+$. Using the chain rule, one has
\begin{align}\label{eq:v-deriv-wm}
    V'(s)=-4\int_{M \backslash B_{s^{1/4}\gamma^{-1}}^{g_M}}\big( W^{g_s}_{\alpha i j \beta}\nabla^\alpha \nabla^\beta H^{ij}+\frac{1}{2}R_{g_s}^{\alpha \beta}W^{g_s}_{\alpha i j \beta}H^{ij}\big)\,dV_{g_s}-\int_{\partial B_{s^{1/4}\gamma^{-1}}^{g_M}}\abs{W^{g_s}}^2\frac{1}{4\gamma s^{3/4}}\,d\sigma_{g_s}.
\end{align}
We now claim that, up to a (small) error term, it is possible to substitute $g_s=g_M + sH$ by $g_M$ in the above expression. Indeed, consider for instance the first term of the first integral. Then, recalling \eqref{eq:metr-err-g_s}, \eqref{eq:metr-err-g_s-inv} and the formula for Weyl's tensor in local coordinates \eqref{eq:weyl-local-coord-expr}, one sees that
\begin{align}
\notag
    \Big\vert\int_{M \backslash B_{s^{1/4}\gamma^{-1}}^{g_M}} W^{g_s}_{\alpha i j \beta}&\nabla^\alpha \nabla^\beta H^{ij}\,dV_{g_s}-\int_{M \backslash B_{s^{1/4}\gamma^{-1}}^{g_M}}W^{g_M}_{\alpha i j \beta}\nabla^\alpha \nabla^\beta H^{ij}\,dV_{g_M}\Big\vert \\
    \label{eq:wm-est-1}
    &\leq Cs+C\int_{B_\epsilon\backslash B_{s^{1/4}\gamma^{-1}}^{g_M}} s(\abs{z}^{-2}+\abs{z}^{-3}+\abs{z}^{-4})\abs{z}^{-4}\,dV_{g_M}\leq Cs+ C\gamma^4,
\end{align}
where we also used that $dV_{g_s}=dV_{g_M}+O(s^{1/4})dx$. Arguing in the same way, we also see that
the second term of the first integral generates an error of order $\sqrt{s}$. As for the last integral in \eqref{eq:v-deriv-wm}, using \eqref{eq:weyl-norm-perturb} we only obtain 
\begin{align*}
    \int_{\partial B_{s^{1/4}\gamma^{-1}}^{g_M}}\abs{W^{g_s}}^2\frac{1}{4\gamma s^{3/4}}\,d\sigma_{g_s}=\int_{\partial B_{s^{1/4}\gamma^{-1}}^{g_M}}\abs{W^{g_M}}^2\frac{1}{4\gamma s^{3/4}}\,d\sigma_{g_M}+O(1),
\end{align*}
so we actually need to be more precise. Since $d\sigma_{g_s}=d\sigma_{g_M}+O(s^{1/4})\,d\sigma_{\Sp^3}$, we focus only on the squared norm of the Weyl tensor. Being $\gamma\ll 1$, we can Taylor-expand this quantity in $\partial B_{s^{1/4}\gamma^{-1}}^{g_M}$ and obtain
\begin{align*}
    \abs{W^{g_s}}^2=\abs{W^{g_M}}^2+\underbrace{s\Big[-4\big( W^{g_M}_{\alpha i j \beta}\nabla^\alpha \nabla^\beta H^{ij}+\frac{1}{2}R_{g_M}^{\alpha \beta}W^{g_M}_{\alpha i j \beta}H^{ij}\big)-\frac{1}{2}\abs{W^{g_M}}^2g_M^{ij}H_{ij}\Big]}_{=O(\gamma^4)}+O(\gamma^8).
\end{align*}
Then, since $\abs{\nabla^2 H_{ij}}=O(\abs{z}^{-4})$, one has
\begin{align*}
    \int_{\partial B_{s^{1/4}\gamma^{-1}}^{g_M}}\abs{W^{g_s}}^2\,d\sigma_{g_s}=&\int_{\partial B_{s^{1/4}\gamma^{-1}}^{g_M}}\abs{W^{g_M}}^2\,d\sigma_{g_M}
    -4s\int_{\partial B_{s^{1/4}\gamma^{-1}}^{g_M}} W^{g_M}_{\alpha i j \beta}\nabla^\alpha \nabla^\beta H^{ij}\,d\sigma_{g_M}+O(\gamma^5 s^{3/4}).
\end{align*}
From this formula and a change of variables, it follows that
\begin{align}\label{eq:wm-est-form}
    \frac{1}{4\gamma s^{3/4}}\int_{\partial B_{s^{1/4}\gamma^{-1}}^{g_M}}\abs{W^{g_s}}^2\,d\sigma_{g_s}=\frac{\pi^2}{2}\gamma^{-4}\abs{W^{g_M}(p)}^2-\int_{\Sp^3}W^{g_M}_{\alpha i j \beta}(p)\partial^\alpha \partial^\beta \tilde{H}^{ij}\,d\sigma +O(\gamma^4)+O(s^{1/4}\gamma^{-1}),
\end{align}
where $\tilde{H}_{ij}(x):=T_{kijl}{x^k x^l}{\abs{x}^{-4}}$.
But now, on $\Sp^3$,
\begin{align*}
    \partial^2_{\alpha \beta} \tilde{H}_{ij}(x)=T_{kijl}\Big[\delta_\alpha^k \delta_\beta^l+\delta_\beta^k \delta_\alpha^l-4x_\beta(\delta_\alpha^k x^l+x^k\delta_\alpha^l)-4\big(\delta_{\alpha\beta}x^k x^l+x_\alpha(\delta_\beta^k x^l+x^k \delta_\beta^l)\big)+24x_\alpha x_\beta x^k x^l\Big],
\end{align*}
so
\begin{align}
\notag
    \int_{\Sp^3}W_{g_M}^{\alpha i j \beta}(p)\partial_\alpha \partial_\beta \tilde{H}_{ij}\,d\sigma&=-W_{g_M}^{\alpha i j \beta}\int_{\Sp^3}\Big[(T_{\alpha i j \beta}+T_{\beta i j \alpha})-4(T_{\alpha i j k}+T_{kij \alpha})x_\beta x^k \\
    \notag
    & \qquad \qquad \qquad \qquad-4(T_{\beta i j k}+ T_{k i j \beta})x_\alpha x^k +24 T_{kijl}x_\alpha x_\beta x^k x^l\Big]\,d\sigma \\
    \label{eq:wm-est-coeff}
    &=-W_{g_M}^{\alpha i j \beta}(2\pi^2-2\pi^2-2\pi^2+2\pi^2) \big( T_{\alpha i j \beta} + T_{\beta i j \alpha}\big)=0,
\end{align}
where the last identity follows from \eqref{eq:brendleform-order2} and \eqref{eq:brendle-form}.
Substituting \eqref{eq:wm-est-1}, \eqref{eq:wm-est-form} and \eqref{eq:wm-est-coeff} into \eqref{eq:v-deriv-wm}, we finally obtain
\begin{align*}
    V'(s)=&-4\int_{M \backslash B_{s^{1/4}\gamma^{-1}}^{g_M}}\big( W^{g_M}_{\alpha i j \beta}\nabla^\alpha \nabla^\beta H^{ij}+\frac{1}{2}R_{g_M}^{\alpha \beta}W^{g_M}_{\alpha i j \beta}H^{ij}\big)\,dV_{g_M}\\
    &\qquad-\frac{\pi^2}{2}\gamma^{-4}\abs{W^M(p)}^2 +O(\gamma^4) + O(s^{1/4}\gamma^{-1}).
\end{align*}
Integrating by parts twice and using that $g_M$ is Bach-flat, we further have
\begin{align*}
    V'(s)=&4\int_{\partial B^{g_M}_{s^{1/4}\gamma^{-1}}}W^{g_M}_{n i j \beta}\nabla^\beta H^{ij}\,d\sigma_{g_M}-4\int_{\partial B^{g_M}_{s^{1/4}\gamma^{-1}}}\nabla^\beta W^{g_M}_{n i j \beta}H^{ij}\,d\sigma_{g_M} \\
    &\qquad-\frac{\pi^2}{2}\gamma^{-4}\abs{W^M(p)}^2 +O(\gamma^4) + O(s^{1/4}\gamma^{-1}),
\end{align*}
where here $n$ denotes the normal direction pointing towards the \emph{exterior} of the ball. But now the second integral is an $O(s^{1/4}\gamma^{-1})$, while, after scaling, the first integral satisfies
\begin{align*}
    \int_{\partial B^{g_M}_{s^{1/4}\gamma^{-1}}}W^{g_M}_{n i j \beta}\nabla^\beta H^{ij}\,d\sigma_{g_M}&=W^M_{\alpha i j \beta}(p)T^{kijl}\int_{\Sp^3}(\delta^\beta_k x_l + x_k \delta^\beta_l-4 x^\beta x_k x_l)x^\alpha\,d\sigma_{\Sp^3}+O(s^{1/4}\gamma^{-1}) \\
    =&W^M_{\alpha i j \beta}(p)\int_{\Sp^3}\Big[(T^{\beta i j k}+T^{kij\beta})x^\alpha x_k -4 T^{kijl}x^\alpha x^\beta x_k x_l\Big]\,d\sigma_{\Sp^3}+O(s^{1/4}\gamma^{-1}) \\
    =&W^M_{\alpha i j \beta}\Big(\frac{\pi^2}{2}-\frac{\pi^2}{3}\Big)\big(T^{\alpha i j \beta}+T^{\beta i j \alpha}\big)+ O(s^{1/4}\gamma^{-1}) \\
    =&\frac{\pi^2}{3}W^M_{\alpha i j \beta} T^{\alpha i j \beta} + O(s^{1/4}\gamma^{-1}),
\end{align*}
where we used \eqref{eq:brendleform-order2}, \eqref{eq:brendle-form} and the symmetry of $T$ in $\alpha , \beta$.
Substituting inside the formula for $V'(s)$ and taking the limit as $s\to 0^+$, we finally obtain \eqref{eq:V'(0)}.
\end{proof}

By looking at \eqref{eq:V(t)-def}, \eqref{eq:V(t)-taylor-exp}, \eqref{eq:V'(0)} and recalling that $t=a^4$, we obtain the following:

\begin{corollary}\label{cor:weyl-en-g_a-M}
    Let $g_a$ be given by \eqref{eq:g_a} (in particular $W=W^M(p)$ and $T_{kijl}=\frac{1}{2}\p^2_{kl}A_{ij}(0)$). Then
\begin{align}\label{eq:weyl-energy-g_a-expansion}
    \int_{M\backslash B_{a\gamma^{-1}}^{g_M}(p)}\abs{W^{\tilde{g}_M}}^2\,dV_{\tilde{g}_M}=\w^M(g_M)-\frac{\pi^2}{2}\frac{a^4}{\gamma^4}\abs{W^M(p)}^2+\frac{4}{3}a^4 W^{kijl} T_{kijl} + O(a^4\gamma^4)+ o(a^4),
\end{align}
provided $0\ll a \ll\gamma \ll 1$.
\end{corollary}

\begin{remark}
   The approach adopted here is not the only viable one. Indeed, it is also possible to directly add a quadratic correction to the (scaled and) inverted metric directly in order to obtain an expansion as in \eqref{eq:g_a}. In this way, it is no longer necessary to project the hessian on its \acc curvature part'' using Lemma \ref{lem:tensor-projection} (as done in Section \ref{sec:green-LCF}). However, we would then get additional terms when considering the correction in normal coordinates on $M$ during the proof of Lemma \ref{lem:V'(0)-computation}, which would make the above computations significantly more involved.
\end{remark}

\section{Conclusion of the proofs} \label{sec:conclusion}

After having constructed the metric corrections $g_a$ on $M$ (see \eqref{eq:g_a}) and $g_b$ on $Z$ (see \eqref{eq:g_b-definition}) and having estimated their Weyl energies in Corollary \ref{cor:weyl-en-g_a-M} and Corollary \ref{cor:Weyl-energy-g_b-Z} respectively, we are now in position to define a competitor metric $g_X$ on the connected sum $Z\# \overbar{M}$ and compute its Weyl energy. As we will see in Proposition \ref{prop:en-balance}, the difference between the Weyl energy of $g_X$ and that of $g_M$ will depend upon the interaction term \eqref{eq:A-interaction-form}, which can be shown to have the right sign (i.e. to be negative) after a proper choice of coordinates on $M$, see Lemma \ref{lem:inter-term-formula}.
This will allow us to prove Theorem \ref{thm:main2} in this setting, that is, for $Z$ diffeomorphic to a finite quotient of $\Sp^1 \times \Sp^3$. Then, in Proposition \ref{p:ZZ1} we will deal with the general case in which $Z$ is a connected sum of quotients as in Theorem \ref{thm:CTZ-LCF-classification}. Finally, we will show how to extend the result to the orbifold setting and prove Theorem \ref{thm:gluing-orbifold}.

\subsection{Definition and expansion of the metric on the connected sum}\label{subsec-gluing-metrics}

We start by defining the connected sum between $Z$ and $M$. Let $\{x^i\}$ denote normal (flat) coordinates for $g_Z$ at $q\in Z$, defined for $\abs{x}<\delta$.
With a slight abuse of notation, let us also denote by $\{x^i\}$ the inverted normal coordinates for $g_a$ defined at the beginning of Section \ref{sec:weyl-en-M}, which are defined for $\abs{x}>a/\eps$. 
As in the previous section, let us consider $\gamma>0$ satisfying $0<a\ll\gamma\ll\eps,\delta\ll1$. We now glue $M$ and $Z$ along the annular regions $\{x\mid \gamma/2 <\abs{x}<2\gamma\}$ (of both $M$ and $Z$, given by their respective coordinates) via the identification map $x\to x$.
We denote by $X$ the resulting smooth manifold.

\begin{remark}\label{rem:gluing-diffeo-X}
    By construction, $X$ is diffeomorphic to $Z\# \overbar{M}$, where $\overbar{M}$ denotes the manifold $M$ endowed with its opposite orientation. This is the same choice adopted in \cite{malchiodi-malizia-2025-pre-weyl}, see Remark 3.5, as well as in \cite{gursky-viaclovsky-2016-Advances}, see their Remark 9.1.
    This does not create any problem, since we can easily get a manifold diffeomorphic to $Z\# M$ by simply repeating the same procedure with $\overbar{M}$ in place of $M$.
\end{remark}

We now want to define a competitor metric $g_X$ on $X$. To begin, recall the expressions \eqref{eq:g_b-definition} for $g_b$ in normal (flat) coordinates at $q\in Z$ and \eqref{eq:g_a} for $g_a$ in inverted conformal normal coordinates at $p\in M$. 
We would like to glue these metrics at $\abs{x}=\gamma$, but in order to do so we first need to cut off some higher order error terms.

Since we got $A(0)=0$, $\nabla A(0)=0$ (see e.g. \eqref{eq:AgradA-vanish}), we can write
\begin{align*}
    A_{ij}(x)=\frac{1}{2}\p^2_{kl}A_{ij}(0) x^k x^l + \eta_{ij}(x), \qquad \text{for $\abs{x}<\delta$,}
\end{align*}
where
\begin{align}\label{eq:eta-err-est}
    \big| \nabla^k\eta_{ij}(x)\big|\leq C\abs{x}^{3-k}, \qquad \text{for $\abs{x}\leq \delta$, $k=0,\dots,4$.}
\end{align}
Let us define a smooth cutoff function $\chi_\gamma$ satisfying
\begin{align}\label{eq:chi_gamma-def}
    \chi_\gamma(r)=\begin{cases*}
        0 & \text{for $r\leq \frac{\gamma}{4}$} \\
        1 & \text{for $r\geq \frac{3}{4}\gamma$},
    \end{cases*}
     \qquad \qquad \big|(\chi_\gamma)^{(k)}\big|\leq C \gamma^{-k}, \quad k=0,\dots,4,
\end{align}
where $C>0$ does not depend upon $\gamma$.
We now define the modified metric  $\tilde{g}_a$ on $M\backslash\{p\}$ by
\begin{align}\label{eq:tildeg_a}
    (\tilde{g}_a)_{ij}(x):=\delta_{ij}+a^2\Big(-\frac{1}{3}W_{kijl}\frac{x^k x^l}{\abs{x}^4}+\frac{1}{2}\p^2_{kl}A_{ij}(0)x^k x^l\Big)+ \chi_\gamma(\gamma-\abs{x})\zeta^a_{ij}(x), \qquad \text{for $\abs{x}>\frac{a}{\eps}$},
\end{align}
with $\tilde{g}_a=g_a$ elsewhere on $M\backslash\{p\}$.
Similarly, we define the modified metric $\hat{g}_b$ on $Z$ by
\begin{align}\label{eq:hatg_b}
    (\hat{g}_b)_{ij}(x):=\delta_{ij}+b^2\Big(-\frac{1}{3}W_{kijl}\frac{x^k x^l}{\abs{x}^4}+\frac{1}{2}\p^2_{kl}A_{ij}(0) x^k x^l+\chi_\gamma(\abs{x}-\gamma)\eta_{ij}(x)\Big), \qquad \text{for }\abs{x}<\delta,
\end{align}
with $\hat{g}_b=g_b$ elsewhere on $Z$.
We now \underline{take $a=b$} and define 
\begin{align}\label{eq:g_X-definition}
    g_X:=\begin{cases*}
    g_a & \text{in $M\setminus\{x\mid\abs{x}>2a/\eps\}$} \\
        \tilde{g}_a & \text{in $\{x\mid 2a/\eps<\abs{x}<\gamma$\}} \\
        \hat{g}_b & \text{in $\{x\mid \gamma\leq \abs{x}\leq \delta/2\}$}\\
        g_b & \text{in $Z\setminus \{x\mid \abs{x}\leq\delta/2\}$}.
    \end{cases*}
\end{align}
By definition, $g_X$ is a smooth metric on $X$.

\begin{proposition}\label{prop:en-balance}
    Let $g_X$ be defined as in \eqref{eq:g_X-definition}, and assume that
    $0<a=b\ll \gamma \ll \eps,\delta\ll 1$. Then the following expansion holds:
    \begin{align}\label{eq:energy-balance}
        \w^X(g_X)-\w^M(g_M)=\frac{2}{3}a^4 W^{kijl}\p^2_{kl} A_{ij}(0)+ O(a^4\gamma)+o(a^4).
    \end{align}
\end{proposition}

\begin{proof}
    By Corollary \ref{cor:Weyl-energy-g_b-Z}, Corollary \ref{cor:weyl-en-g_a-M} and the definition of $g_X$, we see that
    \begin{align*}
        \w^X(g_X)-\w^M(g_M)=&\frac{2}{3}a^4 W^{kijl}\p^2_{kl} A_{ij}(0)+ O(a^4\gamma^2)+o(a^4) \\
        &+\int_{B_{2\gamma}\backslash B_{\gamma/8}}\abs{W^{g_X}}^2\,dV_{g_X}-\int_{B_{2\gamma}\backslash B_\gamma}\abs{W^{g_b}}^2\,dV_{g_b}-\int_{B_\gamma \backslash B_{\gamma/8}}\abs{W^{g_a}}^2\,dV_{g_a}.
    \end{align*}
    The proposition then follows from the next lemma.
\end{proof}

\begin{lemma}
    Under the assumptions of Proposition \ref{prop:en-balance} (in particular $b=a$), one has
    \begin{align*}
        \int_{B_{2\gamma}\backslash B_{\gamma/8}}\abs{W^{g_X}}^2\,dV_{g_X}-\int_{B_{2\gamma}\backslash B_\gamma}\abs{W^{g_b}}^2\,dV_{g_b}-\int_{B_\gamma \backslash B_{\gamma/8}}\abs{W^{g_a}}^2\,dV_{g_a}=O(a^4\gamma).
    \end{align*}
\end{lemma}
\begin{proof}
    We deal separately with the integrals in $B_{2\gamma}\backslash B_\gamma$ and those in $B_{\gamma}\backslash B_{\gamma/8}$.

\medskip

    \noindent
    $\bullet$ Consider 
    \begin{align}\label{eq:outer-int-diff}
        \int_{B_{2\gamma}\backslash B_\gamma}\abs{W^{g_X}}^2 \,dV_{g_X}-\int_{B_{2\gamma}\backslash B_\gamma}\abs{W^{g_b}}^2 \,dV_{g_b}=\int_{B_{2\gamma}\backslash B_\gamma}\big|W^{\hat{g}_b}\big|^2 \,dV_{\hat{g}_b}-\int_{B_{2\gamma}\backslash B_\gamma}\abs{W^{g_b}}^2 \,dV_{g_b},
    \end{align}
    and let us write $g_b=g_E + b^2F$ and $\hat{g}_b=g_E+b^2\hat{F}$, where we recall that $\hat{g}_b$ is given by \eqref{eq:hatg_b}.
    Since $b=a\ll \gamma$, we can Taylor-expand the Weyl energies in $b$ at $g_E$ as in the beginning of Section \ref{sec:Weyl-en-Z}. Recalling that $\dot{B}(F)=0$, expanding and integrating by parts we obtain
    \begin{align}
    \notag
        \int_{B_{2\gamma}\backslash B_\gamma}\abs{W^{g_b}}^2\,dV_{g_b}&=-2 a^4\int_{B_{2\gamma}\backslash B_\gamma}\dot{W}(F)_{\alpha i j \beta} \p^\alpha \p^\beta F^{ij}\,dx + O(a^6) \\
        \label{eq:intgb-err}
        &=2a^4\int_{\p( B_{2\gamma}\backslash B_\gamma)}\Big(\p^\alpha \dot{W}(F)_{\alpha i j \eta}F^{ij}-\dot{W}(F)_{\eta i j \beta}\p^\beta F^{ij}\Big) \nu^\eta \,d\sigma + O(a^6),
    \end{align}
where this time $\nu$ is the exterior unit normal vector. Similarly,
\begin{align}
\notag
    \int_{B_{2\gamma}\backslash B_\gamma}\big|W^{\hat{g}_b}\big|^2\,dV_{\hat{g}_b}=&-2 a^4\int_{B_{2\gamma}\backslash B_\gamma}\dot{B}_{ij}(\hat{F})\hat{F}^{ij}\,dx \\
    \label{eq:intgbhat-err}
        &+2a^4\int_{\p( B_{2\gamma}\backslash B_\gamma)}\Big(\p^\alpha \dot{W}(\hat{F})_{\alpha i j \eta}\hat{F}^{ij}-\dot{W}(\hat{F})_{\eta i j \beta}\p^\beta \hat{F}^{ij}\Big) \nu^\eta \,d\sigma + O(a^6).
\end{align}
We now decompose
\begin{align*}
    \hat{F}_{ij}(x)= F_{ij}(x)-\big(1-\chi_\gamma(\abs{x}-\gamma)\big)\eta_{ij}(x)=:F_{ij}(x)+E_{ij}(x),
\end{align*}
so that
\begin{align*}
    \dot{B}_{ij}(\hat{F})\hat{F}^{ij}=\dot{B}_{ij}(E)\hat{F}^{ij}.
\end{align*}
Recalling that $\dot{B}$ is a linear fourth-order operator, from the definition of $E$ and $\hat{F}$ (that depend upon $\eta$,$\chi_\gamma$ defined in \eqref{eq:eta-err-est}, \eqref{eq:chi_gamma-def}), we see that
\begin{gather*}
    \big|\hat{F}^{ij}(x)\big|\leq C\abs{x}^{-2}, \qquad \big|\dot{B}_{ij}(E)\big|\leq C \gamma^{-1}, \qquad \text{for $\gamma\leq\abs{x}\leq 2\gamma$}. 
\end{gather*}
As a consequence,
\begin{align}\label{eq:hatgb-int-dotb-err}
    \bigg|\int_{B_{2\gamma}\backslash B_\gamma}\dot{B}_{ij}(\hat{F})\hat{F}^{ij}\,dx\bigg|=\bigg|\int_{B_{2\gamma}\backslash B_\gamma}\dot{B}_{ij}(E)\hat{F}^{ij}\,dx\bigg|\leq C \gamma.
\end{align}
Next, we look at the boundary integrals in \eqref{eq:intgb-err}, \eqref{eq:intgbhat-err}. To begin, we notice that $F^{ij}=\hat{F}^{ij}$ near $\p B_{2\gamma}$, so the integrals in the outer boundary $\p B_{2\gamma}$ cancel out when taking the difference in \eqref{eq:outer-int-diff}. Thus it only remains to estimate the integrals on $\p B_\gamma$.

By linearity, $\dot{W}(\hat{F})_{\alpha i j \beta}=\dot{W}(F)_{\alpha i j \beta}+\dot{W}(E)_{\alpha i j \beta}$; moreover, again by \eqref{eq:eta-err-est}, \eqref{eq:chi_gamma-def}, we get the following estimates at scale $\abs{x}=\gamma$:
\begin{gather*}
    \big|\dot{W}_{\alpha i j \beta}(E)\big|\leq C\gamma, \qquad \big|\p^\alpha \dot{W}_{\alpha i j \eta}(E)\big|\leq C, \qquad \abs{E^{ij}}\leq C \gamma^3, \qquad \big|\p^\beta E^{ij}\big|\leq C\gamma^2.
\end{gather*}
Using these, we readily obtain
\begin{align*}
    \bigg|\int_{\p B_\gamma}\Big(\dot{W}(\hat{F})_{\eta i j \beta}&\p^\beta \hat{F}^{ij}-\dot{W}(F)_{\eta i j \beta}\p^\beta F^{ij}\Big)\nu^\eta\,d\sigma\bigg| \\
    &=\bigg| \int_{\p B_\gamma} \Big(\dot{W}(F)_{\eta i j \beta}\p^\beta E^{ij}+\dot{W}(E)_{\eta i j \beta} \p^\beta \hat{F}^{ij}\Big)\nu^\eta\,d\sigma\bigg| \leq C \gamma, 
\end{align*}
and 
\begin{align*}
    \bigg|\int_{\p B_\gamma}\Big(\p^\alpha \dot{W}(\hat{F})_{\alpha i j \eta}&\hat{F}^{ij}-\p^\alpha \dot{W}(F)_{\alpha i j \eta} F^{ij}\Big)\nu^\eta\,d\sigma\bigg| \\
    &=\bigg|\int_{\p B_\gamma}\Big(\p^\alpha \dot{W}({F})_{\alpha i j \eta}E^{ij}+\p^\alpha \dot{W}(E)_{\alpha i j \eta} \hat{F}^{ij}\Big)\nu^\eta\,d\sigma\bigg|\leq C\gamma.
\end{align*}
Using these estimates and \eqref{eq:hatgb-int-dotb-err} in \eqref{eq:outer-int-diff}, we obtain
\begin{align*}
    \int_{B_{2\gamma}\backslash B_\gamma}\abs{W^{g_X}}^2 \,dV_{g_X}-\int_{B_{2\gamma}\backslash B_\gamma}\abs{W^{g_b}}^2 \,dV_{g_b}=O(a^4\gamma).
\end{align*}

\medskip

\noindent
$\bullet$ Consider now the inner integrals
\begin{align}\label{eq:inner-err-integrals}
    \int_{B_\gamma \backslash B_{\gamma/8}}\abs{W^{g_X}}^2 \,dV_{g_X}-\int_{B_\gamma \backslash B_{\gamma/8}}\abs{W^{g_a}}^2 \,dV_{g_a}= \int_{B_\gamma \backslash B_{\gamma/8}}\abs{W^{\tilde{g}_a}}^2 \,dV_{\tilde{g}_a}-\int_{B_\gamma \backslash B_{\gamma/8}}\abs{W^{g_a}}^2 \,dV_{g_a}.
\end{align}
Recalling the definitions \eqref{eq:g_a}, \eqref{eq:tildeg_a} of $g_a, \tilde{g}_a$, if we locally write $g_a= g_E+ a^2V$, then 
\begin{align*}
    (\tilde{g}_a)_{ij}(x)=\delta_{ij}+ a^2 V_{ij}(x) + a^2\big( \chi_\gamma(\gamma-\abs{x})-1\big)\tilde{\zeta}_{ij}^a(x),
\end{align*}
where here (recall \eqref{eq:zeta^a-err-def})
\begin{align*}
    \tilde{\zeta}^a=a^{-2}\zeta^a, \quad \text{so that} \quad \big|\nabla^k \tilde{\zeta^a}\big|\leq C a \abs{x}^{3+k}, \quad k=0,\dots,4.
\end{align*}
Thus we see that the difference between $\tilde{g}^a$ and $g^a$ this time is given by a term of main order $a^3$ (instead of $a^2$ as before), so that we can argue as above to see that the error term coming from \eqref{eq:inner-err-integrals} is $o(a^4 \gamma)$. This concludes the proof.
\end{proof}

\subsection{Sign of the interaction term}\label{subsec:inter-sign}

As shown in Proposition \ref{prop:en-balance}, the key quantity appearing in the energy balance is given by a positive constant times the product between the hessian of ${A}$ and the Weyl tensor $W^M(p)$. 

\begin{lemma}\label{lem:inter-term-formula}
    Let $A$ be given as in Section \ref{sec:green-LCF}, and recall that $W:=W^M(p)$. Then, if $M$ is not self-dual or anti-self-dual, one has
    \begin{align}\label{eq:interaction-term-identity}
        W^{kijl}\p^2_{kl}{A}_{ij}(0)<0,
    \end{align}
    for a suitable choice of $p\in M$, of conformal normal coordinates for $g_M$ at $p$ and for $t-1$ small enough.
\end{lemma}
\begin{remark}
    We notice that the interaction term \eqref{eq:interaction-term-identity} is not affected by gauge terms on $Z$  generated by infinitesimal diffeomorphisms. Indeed, for each term of type $(\mathcal{L}_{X}g)_{ij}=\nabla_i\omega_j+\nabla_j\omega_i$, being $g_Z$ Euclidean around $0$ one has
    \begin{align*}
        W^{kijl}\p^2_{kl}(\mathcal{L}_{X}g)_{ij}=W^{kijl}\big(\p^3_{kli}\omega_j+\p^2_{klj}\omega_i\big)=0,
    \end{align*}
    by the symmetry properties of Weyl's tensor.
\end{remark}

In order to prove Lemma \ref{lem:inter-term-formula}, we need the following identity:

\begin{lemma}\label{lem:hessian-identity-W^P}
    Let $(e_1,\dots,e_4)$ be an orthonormal basis for $T_pM$. Then 
    \begin{align}
    \notag
        W^{kijl}\p^2_{kl}{A}_{ij}(0)&=-\frac{1}{3}C_2(t)W\stell W^P+O\big((t-1)^{-1}\big) \\
        \label{eq:hessian-identity-1}
        &=:-\frac{1}{3}C_2(t)W^{kijl}\big(W^P_{kijl}+W^P_{lijk}\big) +O\big((t-1)^{-1}\big), \qquad \text{as $t\to 1^+$,}
    \end{align}
    where $W^P:=W\circ P$ and $P\colon T_pM\to T_pM$ denotes the linear map given by $P(e_i):=(-1)^{\delta_{4i}}e_i$.
\end{lemma}
\begin{proof}
To see this, we just need to explicitly compute both sides of \eqref{eq:hessian-identity-1}. Being $C_2(t)=O\big((t-1)^{-4}\big)$ as $t\to 1^+$, we can forget about the lower order error terms in the following computations. For the LHS, using \eqref{eq:A-interaction-form} we have
\begin{align}
\notag
    W^{kijl}\p^2_{kl}{A}_{ij}(0)=&-\frac{1}{3}C_2(t)\bigg[W^{kijl}\big(W_{kijl}+W_{lijk}\big)-4W^{4ijl}\big(W_{4ijl}+W_{lij4}\big) \\
    \notag
    &\qquad \qquad \qquad-4 W^{kij4}\big(W_{kij4}+W_{4ijk}\big)+24W^{4ij4}W_{4ij4}\bigg] \\
    \notag
    =&-\frac{1}{3}C_2(t)\sum_{i,j}\bigg[\sum_{k,l\not=4}W_{kijl}\big(W_{kijl}+W_{lijk}\big)-3\sum_l W_{4ijl}\big(W_{4ijl}+W_{lij4}\big) \\
    \notag
     &+\sum_{k\not=4}W_{kij4}\big(W_{kij4}+W_{4ijk}\big)-4\sum_k W_{kij4}\big(W_{kij4}+W_{4ijk}\big)+ 24\big( W_{4ij4}\big)^2\bigg]  \\
     \notag
     =&-\frac{1}{3}C_2(t)\sum_{i,j}\bigg[\sum_{k,l\not=4}W_{kijl}\big(W_{kijl}+W_{lijk}\big)-3\sum_{l\not=4} W_{4ijl}\big(W_{4ijl}+W_{lij4}\big) \\
     \label{eq:weyl-id-rearr-1}
     &-3\sum_{k\not=4} W_{kij4}\big(W_{kij4}+W_{4ijk}\big)+ 10\big( W_{4ij4}\big)^2\bigg].
\end{align}

    As for the RHS of \eqref{eq:hessian-identity-1}, by definition of $P$ we have
    \begin{align}
    \notag
        W\stell W^P=&W^{kijl}\big(W_{kijl}+W_{lijk}\big)(-1)^{\delta_{4k}+\delta_{4i}+\delta_{4j}+\delta_{4l}} \\
        \notag
        =&\sum_{i,j}(-1)^{\delta_{4i}+\delta_{4j}}\Big[\sum_{k,l\not=4}W_{kijl}\big(W_{kijl}+W_{lijk}\big)-\sum_{k\not=4}W_{kij4}\big(W_{kij4}+W_{4ijk}\big) \\
        \notag
        & \qquad \qquad \qquad \qquad -\sum_{l\not=4}W_{4ijl}\big(W_{4ijl}+W_{lij4}\big)+W_{4ij4}\big(W_{4ij4}+W_{4ij4}\big)\Big] \\
        \notag
        =&\sum_{i,j,k,l\not=4}W_{kijl}\big(W_{kijl}+W_{lijk}\big)-\sum_{j,k,l\not=4}W_{k4jl}\big(W_{k4jl}+W_{l4jk}\big)-\sum_{i,k,l\not=4}W_{ki4l}\big(W_{ki4l}+W_{li4k}\big) \\
        \notag
        &+\sum_{k,l\not=4}W_{k44l}\big(W_{k44l}+W_{l44k}\big)-\sum_{i,j,k\not=4}W_{kij4}\big(W_{kij4}+W_{4ijk}\big)+\sum_{j,k\not=4}W_{k4j4}\big(W_{k4j4}+W_{44jk}\big) \\
        \notag
        &-\sum_{i,j,l\not=4}W_{4ijl}\big(W_{4ijl}+W_{lij4}\big)+\sum_{i,l\not=4}W_{4i4l}\big(W_{4i4l}+W_{li44}\big)+2\sum_{i,j\not=4}\big(W_{4ij4}\big)^2 \\
        \notag=&\sum_{i,j}\bigg[\sum_{k,l\not=4}W_{kijl}\big(W_{kijl}+W_{lijk}\big)-\sum_{k\not=4}W_{kij4}\big(W_{kij4}+W_{4ijk}\big)-\sum_{l\not=4}W_{4ijl}\big(W_{4ijl}+W_{lij4}\big)\bigg] \\
        \notag&-2\sum_{j,k,l\not=4}W_{k4jl}\big(W_{k4jl}+W_{l4jk}\big)-2\sum_{i,k,l\not=4}W_{ki4l}\big(W_{ki4l}+W_{li4k}\big) \\
      \label{eq:weyl-id-rearr-2}  &+2\sum_{j,k\not=4}\big(W_{k4j4}\big)^2+2\sum_{i,l\not=4}\big(W_{4i4l}\big)^2 +2\sum_{i,j}\big(W_{4ij4}\big)^2.
    \end{align}
    Relabeling the indices and using the symmetries of curvature tensors, we also have
    \begin{align*}
\sum_{j,k,l\not=4}W_{k4jl}\big(W_{k4jl}+W_{l4jk}\big)=&\sum_{j,k,l\not=4}W_{4klj}\big( W_{4klj}+W_{jkl4}\big)=\sum_{i,j,l\not=4}W_{4ijl}\big( W_{4ijl}+W_{lij4}\big) \\
=&\sum_{i,j}\sum_{l\not=4}W_{4ijl}\big( W_{4ijl}+W_{lij4}\big)-\sum_{i,l\not=4}W_{4i4l}\big(W_{4i4l}+W_{li44}\big),
    \end{align*}
    and 
    \begin{align*}
\sum_{i,k,l\not=4}W_{ki4l}\big(W_{ki4l}+W_{li4k}\big)=&\sum_{i,k,l\not=4}W_{ikl4}\big( W_{ikl4}+W_{4kli}\big)=\sum_{i,j,k\not=4}W_{kij4}\big( W_{kij4}+W_{4ijk}\big) \\
=&\sum_{i,j}\sum_{k\not=4}W_{kij4}\big( W_{kij4}+W_{4ijk}\big)-\sum_{j,k\not=4}W_{k4j4}\big(W_{k4j4}+W_{44jk}\big).
    \end{align*}
    Substituting in \eqref{eq:weyl-id-rearr-2} yields
    \begin{align*}
        W\stell W^P=&\sum_{i,j}\bigg[\sum_{k,l\not=4}W_{kijl}\big(W_{kijl}+W_{lijk}\big)-3\sum_{k\not=4}W_{kij4}\big(W_{kij4}+W_{4ijk}\big) \\
        & \qquad\qquad -3\sum_{l\not=4}W_{4ijl}\big(W_{4ijl}+W_{lij4}\big)\bigg]+10\sum_{i,j}\big(W_{4ij4}\big)^2.
    \end{align*}
    \eqref{eq:hessian-identity-1} now follows by comparing the above identity with \eqref{eq:weyl-id-rearr-1}.
\end{proof}

\begin{proof}[Proof of Lemma \ref{lem:inter-term-formula}]
    As above, assume $\{e_1,\dots,e_4\}$ is an orthonormal basis for $T_pM$ and let $\{e^1,\dots,e^4\}$ denote the dual basis. To begin we recall that, given two tensors $S,T\in (T_p^*M)^{\otimes4}$ with the symmetries of Riemann's tensor, they define symmetric endomorphisms $\widehat{S},\widehat{T}\in\mathrm{End}(\Lambda^2(T_p^*M))$ as in \eqref{eq:curv-operator-def}. 
    In particular, with this convention
    \begin{align}\label{eq:sc-prod-curv-oper-conv}
        \mathrm{tr}(\widehat{S}\widehat{T})=\frac{1}{4}\langle S,T\rangle=\frac{1}{4}S^{kijl}T_{kijl}.
    \end{align}

    Given the map $P$ defined in Lemma \ref{lem:hessian-identity-W^P}, i.e. $P(e_i):=(-1)^{\delta_{4i}}e_i$ (but this holds more generally $\forall P\in \mathrm{GL}(4,\R)$), we define the map $W^P:=W\circ P$, to which we can associate $\widehat{W}^P\in\mathrm{End}(\Lambda^2(T_p^*M))$.
    We can also extend $P$ to a linear map $\widehat{P}\in\mathrm{End}(\Lambda^2(T_p^*M))$ by letting
    \begin{align*}
        \widehat{P}(e^i\wedge e^j):=P^T(e^i)\wedge P^T(e^j),
    \end{align*}
    where $P^T\in\mathrm{End}(T_p^*M)$ denotes the transpose of $P$, defined by $P^T(e^i):=P^i_j e^j$. In this setting, a direct computation gives the following identity:
    \begin{align}\label{eq:weyl-P-operator-form}
        \widehat{W}^P=\widehat{P}\circ \widehat{W}\circ \widehat{P}^*,
    \end{align}
    where $\widehat{P}^*$ is the adjoint of $\widehat{P}$. 

    By \cite[Lemma 2]{derdzinski-1983-Compositio}, there exists an oriented orthonormal basis $\{e_1,\dots e_4\}$ for $T_pM$ such that, letting
    \begin{align}
    \notag
        \omega^+&=e^1\wedge e^2+ e^3\wedge e^4 &  \omega^-&= e^1\wedge e^2- e^3\wedge e^4 \\
        \label{eq:basis-forms}
        \eta^+&=e^1\wedge e^3+ e^4\wedge e^2 &  \eta^-&= e^1\wedge e^3- e^4\wedge e^2 \\
        \notag
        \theta^+&=e^1\wedge e^4+e^2\wedge e^3 & \theta^-&=e^1\wedge e^4- e^2\wedge e^3,
    \end{align}
    then $\{\omega^\pm,\eta^\pm,\theta^\pm\}$ is an orthogonal basis for $\Lambda^2_{\pm}(T_p^*M)$ and $\widehat{W}_{\pm}$ is diagonalized in such a basis. In particular, let us denote by $\lambda^\pm,\mu^\pm,\nu^\pm$ the eigenvalues of $\widehat{W}_{\pm}$ so that, in the basis \eqref{eq:basis-forms}, one has
    \begin{equation*}
    \widehat{W} = \begin{pmatrix}
\lambda^+ & & & & & \\
& \mu^+ & &  & &\\
& & \nu^+ &  & &\\
& & & \lambda^- & & \\
& & & & \mu^- & \\
& & & & & \nu^-
\end{pmatrix}.
\end{equation*}
If we now look at the action of $\widehat{P}$ on the elements of \eqref{eq:basis-forms}, we see that 
\begin{equation*}
\widehat{P}=
\begin{pNiceArray}{ccc|ccc}
\Block{3-3}{\text{\LARGE{0}}} & & & 1 & & \\
                       & & &   & 1 & \\
                       & & &   &   & -1 \\
\hline
1 & &    & \Block{3-3}{\text{\LARGE{0}}} & & \\
  & 1 &  &   & & \\
  & & -1 &   & & \\
\end{pNiceArray}.
\end{equation*}
Hence, by \eqref{eq:sc-prod-curv-oper-conv}, \eqref{eq:weyl-P-operator-form} we got 
\begin{align}\label{eq:weyl-eigenvalue-formula}
   W^{kijl}W^P_{kijl}=4 \mathrm{tr}(\widehat{W}\widehat{W}^P)=4 \mathrm{tr}(\widehat{W}\widehat{P}\widehat{W}\widehat{P}^*)=8(\lambda^+\lambda^-+\mu^+\mu^-+\nu^+\nu^-),
\end{align}
and the latter quantity can be made \emph{positive} (for a suitable choice of $p\in M$) whenever $M$ is \emph{not} self-dual or anti-self-dual, see the proof of Lemma 6.2 in \cite{malchiodi-malizia-2025-pre-weyl}.
We notice that under the assumptions of Theorem \ref{thm:main2} we can always find a point $p\in M$ where both $W_+(p)\not=0$ and $W_-(p)\not=0$. In fact, the  vanishing of $B_{ij}$, coupled with a proper conformal choice and the use of biharmonic coordinates, becomes an elliptic system. 
Therefore the metric $g$ has to be real analytic, see \cite[Theorem 12.1]{agmon-douglis-nirenberg-1964-CPAM-2} (and \cite{deturck-kazdan-1981-ENS-regularty-thms} for the case of Einstein metrics), so neither $W_+$ nor $W_-$ 
can vanish on open sets of $M$ without being identically zero. 

Finally, being $\widehat{W}^P$ diagonal as well, we can  apply \cite[Lemma 12.2]{gursky-viaclovsky-2016-Advances} to show that
\begin{align*}
    W\stell W^P=\frac{3}{2}W^{kijl}W^P_{kijl},
\end{align*}
which, together with \eqref{eq:hessian-identity-1} and \eqref{eq:weyl-eigenvalue-formula}, allows to obtain \eqref{eq:interaction-term-identity} provided $t-1$ is small enough.
\end{proof}

\subsection{Construction for a general locally conformally flat $Z$ of positive Yamabe class and conclusion of the proof of Theorem \ref{thm:main2}} \label{subsec:conn-sums}

Here we will exploit the construction in Section \ref{sec:green-LCF} to deal with the case of a 
general $Z$. Recall that, by Theorem \ref{thm:CTZ-LCF-classification}, $Z$ is a connected sum of 
quotients of the type $(\Sp^1_{t_i} \times \Sp^3)/{G_i}$ for some $t_i > 1$ and finite groups $G_i$. 

Let us focus our attention to the first summand, and let us set for brevity $(\Sp^1_{t} \times \Sp^3)/{G} = (\Sp^1_{t_1} \times \Sp^3)/{G_1} := Z_1$. Our goal is to extend the previous tensor $F$ from $Z_1$ to $Z$
 so that the Weyl energy is affected as little as possible. Our goal is to prove the following result.

 \begin{proposition}\label{p:ZZ1}
     Let $g_{Z_1}$ be a locally conformally flat metric on $Z_1$, and let    $g_b=g_{Z_1}+b^2 F$ ($b\ll 1$) 
     be the metric as at the beginning 
     of Section \ref{sec:Weyl-en-Z}. 
For $\tilde{q} \in Z_1, \tilde{q} \neq q$, we assume that $Z$ is obtained from $Z_1$ via a connected sum near $\tilde{q}$. 
Then for every $\eps > 0$  there exists $\delta > 0$ and a metric $g_\eps$ on $Z$ such that $g_\eps = g_b$ on $Z_1 \setminus B_\delta(\tilde{q})$ and such that 
\[
\left| \int_{Z} |W_{g_\eps}|^2_{g_\eps} dV_{g_\eps} - 
\int_{Z_1} |W_{g_b}|^2_{g_b} dV_{g_b} \right| < \eps. 
\]
Here $q$, as before, is the point of $Z_1$ where we 
perform a connected sum to $M$. 
\end{proposition}

In the proof of Theorem \ref{thm:main2} below, we will need a number of small 
parameters. We will first take $b$ to be small, and 
later choose $\eps$ much smaller.

\begin{proof}
    Without loss of generality, we can assume that $g_{Z_1}$ is Euclidean in $B_\delta(\tilde{q})$, and 
    we can choose $\delta$ so small that 
\[
 \int_{ B_\delta(\tilde{q})} |W_{g_b}|^2_{g_b} dV_{g_b}  < \frac{\eps}{2}. 
\]
We choose $0 < \tilde{\delta} \ll \delta$ and choose a radial cutoff function $\chi = \chi(r)$ in $\R^4$ such that 
\[
\begin{cases}
    \chi(r) = 0 & \hbox{ in } B_{\tilde{\delta}}(0); \\ 
    \chi(r) = 1 & \hbox{ in } \R^4 \setminus B_{{\delta}}(0); \\ 
    \int_{B_\delta(0)} \left(
    |\nabla^2 \chi|^2 + |\nabla \chi|^4 \right) dx 
    = o_{{\tilde{\delta}}/{\delta}}(1). 
\end{cases}
\]
Here $o_{{\tilde{\delta}}/{\delta}}(1)$ is a 
quantity that tends to zero as ${\tilde{\delta}}/{\delta}$ tends to zero: the latter property in the last formula can be satisfied by capacity reasons, see e.g. Theorem 2.1.6 in \cite{Ziemre-book-89}.

Recalling that we assume  $g_{Z_1}$ to be locally Euclidean near $\tilde{q}$, we define $g_\eps$ in $g_{Z_1}$-normal coordinates $x$ at $\tilde{q}$ as 
\[
g_\eps = dx^2 + \chi(|x|) b^2 F, 
\]
with $\tilde{\delta}$ to be determined later. 

Recall from \cite{malchiodi-malizia-2025-pre-weyl}, formula (5.41), that the following estimate holds true:
\begin{equation}\label{eq:weyl-est-rough}
        \abs{W^g}_g^2=W^{ijkl}W_{ijkl}\leq C\big(\lvert\nabla g^{-1}\rvert^2\abs{\nabla g}^2+ \abs{\nabla g}^4 +\lvert\nabla^2 g\rvert^2\big). 
    \end{equation}
Here we are denoting by $g$ a metric on $\R^4$ whose coefficient matrix is pointwise close to the Kronecker symbols, as in our case.

By \eqref{eq:weyl-est-rough} and our choice of $\chi$, we have that also 
\[
 \int_{B_\delta(\tilde{q})} |W_{g_\eps}|^2_{g_\eps} dV_{g_\eps}  < \frac{\eps}{2}, 
\]
whenever $\tilde{\delta}$ is chosen sufficiently small, depending on $b, \eps$ and $\delta$. 

\medskip

By our choice of $\chi(|x|) $, $g_\eps$  is Euclidean in $B_{\tilde{\delta}}(\tilde{q})$. 
We can therefore assume that $g_Z$ is obtained via a connected sum in this ball with the other manifolds $(\Sp^1_{t_i} \times \Sp^3)/{G_i}$, $i > 1$, and obtain a metric $g_\eps$ on the whole $Z$
as desired.
\end{proof}

We can now conclude the proof of Theorem \ref{thm:main2}.

\begin{proof}[Proof of Theorem \ref{thm:main2}]
We start by considering the case in which the manifold $Z$ is a single quotient of $\Sp^1 \times \Sp^3$. 
    As pointed out in Remark \ref{rem:gluing-diffeo-X}, the manifold $X$ defined at the beginning of Section \ref{sec:conclusion} is diffeomorphic to $Z\#\overbar{M}$. However, if we perform the same procedure with $Z$ and $\overbar{M}$, we end up with a manifold $Y\overset{\mathrm{diff}}{\cong}Z\#M$.  Applying Proposition \ref{prop:en-balance}, we obtain
    \begin{align}\label{eq:en-difference}
        \w^Y(g_Y)-\w^M(g_M)=\frac{2}{3}a^4 \big({W}^{\overbar{M}}(p)\big)^{kijl}\p^2_{kl} A_{ij}(0)+ O(a^4\gamma)+o(a^4).
    \end{align}
    If $(M,g_M)$ is not self-dual/anti-self-dual, then the same is true for $\overbar{M}$,
    so the proof now follows from Lemma \ref{lem:inter-term-formula} and a subsequent suitable choice of $0<a\ll\gamma\ll 1$. We therefore obtained the desired inequality $\w^Y(g_Y)-\w^M(g_M)<0$ 
    for our special choice of $Z$. 

    In the general case in which $Z$ is a connected sum of quotients of $\Sp^1 \times \Sp^3$, we call $Z_1$ one of these quotients and define $F$ on $Z_1$ as in Section \ref{sec:green-LCF}. If, as before, $q\in Z_1$ be the gluing point to $M$, we recall that 
    $g_Y = g_b$ away from $q$. Choosing then $\eps < \frac{1}{2}(\w^M(g_M)-\w^Y(g_Y))$ in Proposition \ref{p:ZZ1}, we obtain 
    the inequality $\w^Y(g_Y)-\w^M(g_M)<0$ for a general 
    locally conformally flat $Z$. 
\end{proof}

\subsection{Connected sum of a Bach-flat orbifold with a locally conformally flat orbifold and proof of Theorem \ref{thm:gluing-orbifold}}\label{subsec:orbifold-sums}

In this last subsection we show how to extend the proof of Theorem \ref{thm:main2} for some orbifolds.
We start by recalling the definition of orbifold:

\begin{definition}
    A \emph{Riemannian orbifold} $(M^n,g)$ is a topological space which is a smooth manifold with a smooth metric away from a finite set of singular points. At each singular point $p$, $M$ is locally diffeomorphic to a cone over a spherical space form $\Sp^{n-1}/\Gamma_p$, where $\Gamma_p\subset \mathrm{SO}(n)$ is a finite subgroup acting freely on $\Sp^{n-1}$. Moreover, we require that at each singular point the metric locally lifts to a smooth (up to the origin), $\Gamma_p$-equivariant metric on a small  Euclidean ball.
\end{definition}

As for the locally conformally flat four-manifolds with positive scalar curvature, we have a similar classification result for orbifolds of the same class:

\begin{theorem}[\cite{chen-tang-zhu-2012-JDG-4LCF-classification}, Corollary 5.3 and \cite{wang-2025-pre-LCForbifolds}, Theorem 1.6] \label{thm:classif-LCF-PSC-orbifolds}
    Let $(Z,g_Z)$ be a closed, oriented, locally conformally flat $4$-orbifold with positive scalar curvature. Then $Z$ is diffeomorphic to a connected sum of manifolds of type $\big(\R\times \Sp^3\big)/G$ and $\Sp^4/\Gamma$, where $G$ is a cocompact discrete subgroup of the isometry group of the standard metric on $\R\times \Sp^3$ with at most isolated singularities, and $\Gamma\subset \mathrm{SO}(4)\subset\mathrm{SO}(5)$ has two fixed antipodal points.
    In particular, a finite manifold cover of $Z$ is diffeomorphic to $\Sp^4$ or a connected sum of copies of $\Sp^1 \times \Sp^3$.
\end{theorem}

Here by manifold cover we mean that $Z$ is obtained as a quotient of a smooth manifold by the action of some group with fixed points. The next two lemmas shed some light on the structure of each singular factor in the decomposition.

\begin{lemma} \label{lem:football-orbif-count}
    All the factors $\Sp^4/\Gamma$ in Theorem \ref{thm:classif-LCF-PSC-orbifolds} are football orbifolds.
\end{lemma} 
\begin{proof}
    The lemma follows from the fact that if $\Gamma\subset \mathrm{SO}(5)$ acts on $\Sp^4$ with only isolated fixed points, then $\Gamma$ fixes exactly two antipodal points on $\Sp^4$ and acts freely on the orthogonal $\Sp^3$-section of $\Sp^4$, see Lemma 5.1 and the discussion at p.75 in \cite{chen-tang-zhu-2012-JDG-4LCF-classification}. 
\end{proof}

\begin{lemma}\label{lem:cylinder-orbifold-count}
    All the factors $\big(\R\times \Sp^3\big)/G$ in Theorem \ref{thm:classif-LCF-PSC-orbifolds} only possess $\Z_2$-orbifold singularities, and the number of orbifold points is either $0$ or $4$.
\end{lemma}
\begin{proof}
    By the discussion in  Section \ref{subsec:quotients-geom-setting}, we can assume that there exists a normal subgroup $H\trianglelefteq G$, of finite index and isomorphic to $\Z$, which acts on $\R \times \Sp^3$ via translations of length $mR$, for $m\in\Z$ and $R>0$ fixed, and which contains all translations in $G$. Therefore it is sufficient to consider finite quotients of the cylinder $([-R/2,R/2]\times \Sp^3)/\sim$ under the action of the quotient group $N:=G/H$. Looking at the proof of Lemma \ref{lem:quotient-map-cylinder}, we see that the only elements $\gamma\in N$ (all orientation-preserving) generating fixed points are maps acting as an inversion on the first factor and a rotation $O\in \mathrm{O}(4) \setminus \mathrm{SO}(4)$ on the second one, namely $\gamma(r,w)=(L-r,O(w))$, for $\abs{L}\leq R/2$. In this case, $\det (I-{O})=0$ and there exists a (unique up to sign\footnote{With $2$ vectors we could take combinations between them and obtain a positive-dimensional set of fixed points.}) vector $v\in \Sp^3$ satisfying $O(v)=v$. Then 
    \begin{align*}
        \Big(\frac{L}{2}, \pm v\Big), \quad \Big(\frac{L+R}{2},\pm v\Big),
    \end{align*}
    are four distinct fixed points for $\gamma$.
    Moreover, since $\gamma^2=(\mathrm{Id},O^2)$ and $O^2$ fixes $\pm v$, then necessarily $O^2=\mathrm{Id}$ as otherwise $\gamma^2$ would fix a positive-dimensional set. This means that $O$ has only eigenvalues $1$ (with multiplicity $1$) and $-1$ (with multiplicity $3$). Conversely, all elements in $N$ containing an inversion on the $\Sp^1$-factor are of this form and thus fix $4$ points exactly.

Now, let us write $N=N_+ \sqcup N_-$, where 
\begin{align*}
    N_\pm:=\big\{\gamma=(\rho, O)\in N \mid \det(\rho)=\pm 1\big\}.
\end{align*}
If $N_-=\emptyset$, then all elements in $N$ act freely and the quotient is a smooth manifold. Therefore assume from now on that $N_-\not=\emptyset$. Since the sign map 
\begin{align*}
    N\to\{\pm1\},  \quad \gamma=(\rho,O)\mapsto \det(\rho),
\end{align*}
is a homomorphism, then $\abs{N_-}=\abs{N}/2$.
Moreover, each fixed point has stabilizer of order $2$ (if nontrivial $\gamma_1\not=\gamma_2$ both fix $p=(s,w)$, then $\gamma_1 \gamma_2\not=\mathrm{Id}$ would fix $\Sp^1 \times \{w\}$, which is not possible), therefore we always have that, for any $\gamma_1\not=\gamma_2, \, \gamma_1,\gamma_2\in N_-$,  the fixed point sets of $\gamma_1$ and $\gamma_2$ are always pairwise disjoint and each fixed point generates an orbit of cardinality $\abs{N}/2$. Hence the total number of fixed points is $4 \cdot \abs{N_-}=2\abs{N}$, which, once divided by the orbit size $\abs{N}/2$, descends to $4$ orbifold points.
 This concludes the proof.
\end{proof}

\begin{proof}[Proof of Theorem \ref{thm:gluing-orbifold}]
    The proof essentially goes along the same lines of that of Theorem \ref{thm:main2}, therefore we will only highlight how the arguments extend to the orbifold case.
    
The crucial point is that Theorem \ref{thm:classif-LCF-PSC-orbifolds} allows us to extend the construction in Section \ref{sec:green-LCF} to this case whenever we have a $\Z_2$-orbifold point lifting to a cylinder quotient. Indeed, assume we have a singular tensor $F$ on $Z$ which behaves as in Proposition \ref{prop:F-exp-quotients} near the orbifold point. Then, since the quadratic term $\p^2_{kl}A_{ij}(0)x^k x^l$ will be  $\Z_2$-equivariant at $0$, we can immediately define the correction $H$ on $M$, singular at the orbifold point $p\in M$, as in \eqref{eq:H-tens-def}  (it's a purely local argument). At this point, we can perform all the energy expansions as in sections \ref{sec:Weyl-en-Z} and \ref{sec:weyl-en-M}, where the only change will be a different coefficient (depending only on $\abs{\Gamma}=2$, thus equal on $M$ and $Z$) multiplying each term in the energy expansions and caused by the different volume of geodesic balls around the orbifold points.
We can then define the metric $g_X$ on the connected sum and conclude the proof just as above; in particular, the proof of Lemma \ref{lem:inter-term-formula} only depends on a choice of coordinates on a local covering of the orbifold point, so the argument extends with no further issue.

\smallskip

Given the above discussion, we now focus on the construction of the singular section $F$ on $Z$.

\smallskip

Let us start by assuming that $Z$ is a single quotient of type $\big(\R\times \Sp^3\big)/G$. In this case, the group $G$ does not act freely but with fixed points, which descend indeed to four orbifold points by Lemma \ref{lem:cylinder-orbifold-count}. Up to composition with an isometry, we can assume that the orbifold point in $Z$ in which we would like to perform the gluing is $q=(0,S)\in\R\times \Sp^3$.
    We can then argue exactly as in Section \ref{subsec:green-quotients} and, after possibly shrinking the cylindrical variable, define a singular tensor $\xi$ on $\R^4\backslash\{0\}$ as in \eqref{eq:xi-quotient-def}, namely which is gauged in such a way to make the series of its pullbacks convergent and with covariant derivatives vanishing up to first order at $-e_4$. We can also add a further gauge term to also project the hessian on symmetrized curvature-type tensors. By construction, the series of pullbacks will be equivariant and thus define a tensor on $Z$ itself with all the desired properties. Note that the pullbacks by elements in the stabilizer of $(0,S)$ (which is a subgroup isomorphic to the orbifold group $\Z_2$) do not  modify the asymptotic behavior of $\xi$ at $(0,S)$ by virtue of the fact that the Weyl tensor is evaluated at the orbifold point $p\in M$, and therefore it is already $\Z_2$-invariant.

    \smallskip

    We now consider the general case in which $Z$ is diffeomorphic to one of the manifolds described in Theorem \ref{thm:classif-LCF-PSC-orbifolds}, namely
    \begin{align}\label{eq:Z-orbifold-decomposition}
        Z\overset{\mathrm{diff}}{\cong}\#_i \frac{\R\times \Sp^3}{G_i}\#_j\frac{\Sp^4}{\Gamma_j}.
    \end{align}
    If the orbifold point $q$ belongs to one of the $(\R\times\Sp^3)/G_i$ factors, then we can argue as before to define a singular tensor on that factor and then employ Proposition \ref{p:ZZ1} to define a tensor on the whole $Z$. Now, from Lemmas \ref{lem:football-orbif-count} and \ref{lem:cylinder-orbifold-count}, we know that the number of orbifold points with any given group $\Gamma$ is always an even number. Moreover, if there are at least $4$ orbifold points with group $\Z_2$, then we can assume that the orbifold point $q$ always belongs to a cylinder quotient $(\R\times \Sp^3)/G$ in the orbifold decomposition \eqref{eq:Z-orbifold-decomposition}. Indeed, the connected sum of two $\Z_2$-footballs is diffeomorphic to the quotient $(\Sp^1 \times \Sp^3)/N$, where e.g. $N=\{\mathrm{Id},\gamma\}$ and $\gamma(r,w)=(-r,O(w))$, where $O=\mathrm{diag}(1,-1,-1,-1)$. This concludes the proof.
\end{proof}

\begin{remark}\label{rem:orbif-gluing-theorem-generalization}
    We expect that Theorem \ref{thm:gluing-orbifold} can be extended to the more general case in which $Z$ is not conformally equivalent to a football orbifold. Indeed, by Theorem \ref{thm:classif-LCF-PSC-orbifolds} we know that any such orbifold admits a finite manifold cover diffeomorphic to a connected sum of copies of $\Sp^1 \times \Sp^3$. Then the hope is to define a singular correction $F_0$ on one of these cylinder factors, and define $F$ afterwards as the sum of the pullbacks of $F_0$ by the (finitely many) deck elements. However, there are two issues related to this approach that need to be addressed. The first is that we need to suitably \acc normalize'' the manifold cover (and thus $Z$) so that the orbifold point $q$ lifts to a point at the center of a narrow cylinder $\Sp^1_t \times \Sp^3$ with $t$ close to $1$. The second difficulty lies in the fact that, when $\abs{\Gamma}\geq 3$, the representation of $\Gamma$ as a group of matrices depends on our choice of coordinates on the orbifold chart, and in general there will be multiple representations, any two of which agree only up to conjugation with a suitable matrix in $\mathrm{SO}(4)$. Hence we need to choose orbifold charts which simultaneously provide the same \acc matrix representation'' of $\Gamma$ at $p\in M$ and $q\in Z$, while also diagonalizing the Weyl operator on $M$ and providing the right interaction on $Z$.
\end{remark}

\printbibliography[heading=bibintoc]

\Addresses

\end{document}